\documentclass[12pt,leqno]{article}
\usepackage{amsmath}
\usepackage{amssymb}
\usepackage{amsthm} 
\usepackage{epsf}
\usepackage{graphicx}
\usepackage{esint} 
\usepackage{dsfont}
\usepackage{hyperref} 
\hypersetup{backref,  pdfpagemode=FullScreen,  colorlinks=true} 

\numberwithin{equation}{section}

\newtheorem{lem}{Lemma}[section]
\newtheorem{pro}[lem]{Proposition}
\newtheorem{defi}[lem]{Definition}
\newtheorem{thm}[lem]{Theorem}
\newtheorem{cor}[lem]{Corollary}
\theoremstyle{remark}
\newtheorem{rem}[lem]{\bf Remark}

\newcommand{\nn}{\nonumber}
\newcommand{\ms}{\medskip}

\newcommand{\R}{\mathbb{R}}
\renewcommand{\H}{\mathcal H}

\newcommand{\bY}{\mathbb {Y}}
\newcommand{\bT}{\mathbb {T}}
\newcommand{\bV}{\mathbb {V}}
\newcommand{\bP}{\mathbb {P}}
\newcommand{\bH}{\mathbb {H}}

\renewcommand{\d}{\partial}
\newcommand{\dist}{\,\mathrm{dist}}
\newcommand{\sm}{\setminus}

\renewcommand{\i}{\subset}
\newcommand{\wt}{\widetilde}
\newcommand{\wh}{\widehat}

\newcommand{\av}[1]{\left| #1 \right|}

\newcommand{\1}{{\mathds 1}}

\begin{document}

\title{A monotonicity formula for minimal sets with a sliding boundary condition}
\author{
G. David\footnote{Univ Paris-Sud, Laboratoire de Math\'{e}matiques, 
UMR 8658 Orsay, F-91405}
\footnote{CNRS, Orsay, F-91405}
\footnote{Institut Universitaire de France},  
}
\date{}
\maketitle

\ms\noindent{\bf Abstract.}
We prove a monotonicity formula for minimal or almost minimal sets 
for the Hausdorff measure $\H^d$, subject to a sliding boundary constraint where competitors for 
$E$ are obtained by deforming $E$ by a one-parameter family of functions $\varphi_t$
such that $\varphi_t(x) \in L$ when $x\in E$ lies on the boundary $L$.
In the simple case when $L$ is an affine subspace of dimension $d-1$, 
the monotone or almost monotone functional is given by 
$F(r) = r^{-d} \H^d(E \cap B(x,r)) + r^{-d} \H^d(S \cap B(x,r))$,
where $x$ is any point of $E$ (not necessarily on $L$) and
$S$ is the shade of $L$ with a light at $x$.
We then use this, the description of the case when $F$ is constant, and a limiting argument,
to give a rough description of $E$ near $L$ in two simple cases.

\ms\noindent{\bf R\'esum\'e en Fran\c cais.}
On donne une formule de monotonie pour des ensembles minimaux ou presque minimaux
pour la mesure de Hausdorff $\H^d$, avec une condition de bord o\`u les 
comp\'etiteurs de $E$ sont obtenus en d\'eformant $E$ par une famille \`a
un param\`etre de fonctions $\varphi_t$ telles que $\varphi_t(x) \in L$ quand $x\in E$
se trouve sur la fronti\`ere $L$. Dans le cas simple o\`u $L$ est un sous-espace
affine de dimension $d-1$, la fonctionelle monotone ou presque monotone est donn\'ee par
$F(r) = r^{-d} \H^d(E \cap B(x,r)) + r^{-d} \H^d(S \cap B(x,r))$,
o\`u $x$ est un point de $E$, pas forc\'ement dans $L$, et $S$ est
l'ombre de $L$, \'eclair\'ee depuis $x$.
On utilise ceci, la description des cas o\`u $F$ est constante, et un argument de limite,
pour donner une description de $E$ pr\`es de $L$ dans deux cas simples.

\ms\noindent{\bf Key words/Mots cl\'es.}
Minimal sets, almost minimal sets, monotonicity formula, sliding boundary condition,
Plateau problem.

\ms\noindent

\tableofcontents

\section{Introduction}
\label{S1}

The central point of this paper is a monotonicity formula that is valid for
minimal sets with a sliding boundary condition, as defined in \cite{Sliding}, for
balls that are not necessarily centered on the boundary set and also when 
the boundary is not a cone with the same center as the balls.

In this introduction, we restrict to the special case of minimal sets of dimension $d$
in some open set $U \i \R^n$, and when the boundary condition is given by an 
affine plane $L$ of dimension at most $d-1$. The main monotonicity result of this
paper is also valid in more general situations, but the statements are more complicated
and the author is not certain that the extra generality will be used.

Let us say what we mean by sliding minimal sets in this simpler context.
We only consider sets $E$ that are closed in $U$, and 
have a locally finite Hausdorff measure, i.e., for which 
\begin{equation} \label{1.1}
\H^d(E \cap \overline B) < +\infty
\ \text{ for every closed ball $\overline B \i U$.} 
\end{equation}
See for instance \cite{Federer} of \cite{Mattila} for the 
definition of the Hausdorff measure $\H^d$; recall that this is
the same as surface measure for subsets of smooth $d$-dimensional surfaces.

We compare $E$ with images $\varphi_1(E)$, coming from
one parameter families $\{ \varphi_t \}$, $t\in [0,1]$,
of continuous functions defined on $E$, such that
\begin{equation} \label{1.2}
(x,t) \to \varphi_t(x) : E \times [0,1] \to U
\ \text{ is continuous,} 
\end{equation}
\begin{equation} \label{1.3}
\varphi_0(x) = x \ \text{ for } x\in E,
\end{equation}
\begin{equation} \label{1.4}
\varphi_1 \ \text{ is Lipschitz on } E,
\end{equation}
\begin{equation} \label{1.5}
\text{if we set } \, 
W_t =  \big\{ x\in E \, ; \, \varphi_t(x) \neq x \big\} 
\ \text{ and } \ 
\wh W = \bigcup_{0 \leq t \leq 1} W_t \cup \varphi_t(W_t),
\end{equation}
then $\wh W$ is contained in a compact subset of $U$, and finally
(the sliding boundary condition)
\begin{equation} \label{1.6}
\varphi_t(x) \in L \ \text{ for $0 \leq t \leq 1$ when $x\in E \cap L$.} 
\end{equation}
Such families $\{ \varphi_t \}$ will be called 
\underbar{acceptable deformations}; 
note that the notion depends on $E$, $U$, and $L$.

\begin{defi} \label{t1.1} 
We say that $E$ (a closed set in $U$) is a 
\underbar{sliding minimal set in $U$, with boundary} 
\underbar{condition given by $L$}, or in short that \underbar{$E\in SM(U,L)$},
when \eqref{1.1} holds, and 
\begin{equation} \label{1.7}
\H^d(E \cap \wh W) \leq \H^d(\varphi_1(E) \cap \wh W)
\end{equation}
for every acceptable deformation $\{ \varphi_t \}$, and where $\wh W$ is as in \eqref{1.5}.
\end{defi}

Thus, when we deform $E$ to get $F = \varphi_1(E)$, 
we require points of $L$ to stay on $L$. We do not require the
$\varphi_t$ to be injective.
We added the Lipschitz constraint in \eqref{1.4} mostly by tradition, 
but our results would hold without it,
because the resulting class of minimizers would be smaller.

This definition and the more general ones that will be given in Section \ref{S2} 
are modifications of Almgren's initial definition of ``restricted sets''
(see \cite{AlmgrenMemoir}), suited to fit a natural definition of Plateau
problems. See \cite{Sliding} for motivations. Maybe we should mention
that some of the sets that arise from some other natural minimization problems are 
also sliding minimal (or almost minimal) sets. This is the case, for instance,
of the sets that minimize $\H^d$ under the Reifenberg homology boundary
conditions (see \cite{R}, \cite{A68}, 
and more recently \cite{Fang}), or of the support
of size-minimizing currents. See for instance Section 7 of \cite{SteinConf} for a discussion and a proof of this fact.

We are interested in the regularity properties of sliding minimal sets 
near the boundary $L$, and a good monotonicity formula 
will clearly help. In \cite{Sliding}, it was checked that (among other things)
if $L$ is a cone (not necessarily an affine subspace) centered at the origin,
and satisfies some mild regularity constraints, and $E$ is a sliding minimizer as above, 
then the density
\begin{equation} \label{1.8}
\theta_0(r) = r^{-d} \H^d(E \cap B(0,r))
\ \text{ is nondecreasing on $(0,R_0)$} 
\end{equation}
when $R_0$ is such that $B(0,R_0) \i U$. 
Here and below, $B(x,r)$ denotes the open Euclidean ball
of radius $r$ centered at $x$. 
This is a rather easy extension of a very standard result that applies to minimal sets
(without boundary conditions), and even more general objects.
As usual, the monotonicity of $\theta_0$ is proved
by comparing $E$ with a cone, and the sliding boundary condition cooperates
well with this when $L$ is a cone.

Various consequences can be derived from this, 
but it will be good to have monotonicity properties for balls that are
not necessarily centered on $L$, and this is the main point of this paper.

An obvious difficulty with the extension of \eqref{1.8} when
$L$ is not a cone is that we cannot use the same density. For instance,
if $L = \big\{ (x,y,z) \in \R^3 \, ; \,  x = 1 \text{ and } z = 0 \big\}$,
the half plane $E = \big\{ (x,y,z) \in \R^3 \, ; \, x \leq 1 \text{ and } z = 0 \big\}$
is a sliding minimal set of dimension $2$ in $\R^3$, with boundary condition 
given by $L$. The function $\theta_0$ of \eqref{1.8} is constant on
$[0,1]$, and then decreasing. We shall add to $E$ a missing piece, so that we get
a nondecreasing function.

The case when $E$ is a half subspace space bounded by $L$ 
suggest that we symmetrize $E$, but this is not what we are going to do. 
Instead, we shall add to $\H^d(E)$ the measure of the shade of
$L$, seen from the origin. The shade in question is
\begin{equation} \label{1.9}
S = \big\{ y \in \R^n \, ; \, \lambda y \in L 
\text{ for some } \lambda \in [0,1] \big\}
\end{equation}
and the functional that we want to consider is
\begin{equation} \label{1.10}
F(r) = r^{-d} \big[\H^d(E \cap B(0,r)) + \H^d(S \cap B(0,r))\big].
\end{equation}
Our basic theorem is the following.

\ms
\begin{thm} \label{t1.2}
Let $U \in \R^n$ be open, 
$L$ be an affine space of dimension $m \leq d-1$, 
and $E$ be a sliding minimal set of dimension $d$ in $U$, with boundary 
condition given by $L$. Then the function $F$ defined by \eqref{1.10} is 
nondecreasing on $[0,\dist(0,\R^n \sm U))$.
\end{thm}

\ms 
See Theorem \ref{t7.1} for a generalization, where more general sets $L$ 
are allowed. For these more complicated sets $L$, we also change the formula for 
the added term in $L$; the shade only works well when the half lines starting from the
origin do not meet $L$ twice.

In the case mentioned above when $L$ is a line, $E$ is a half plane bounded by $L$, and $0 \in E \sm L$, our formula is exact, in the sense that $F$ is constant 
(the shade exactly compensates the missing half plane). 
The same thing is true when $E$ is a truncated $\bY$-set centered at the origin, 
such that $L$ is contained in one of the three branches of the $\bY$-set, and 
the truncation precisely consists in removing the interior of the shade $S$.
See near \eqref{1.28} for the definition of the $\bY$-sets.
Of course this is interesting because we believe that the truncated $\bY$-set 
is a sliding minimal set.

In contrast, the defect of the general monotonicity formula given in 
Theorem \ref{t7.1} below is that it is possibly not exact on any additional minimal set.

In the two applications that we give below, we shall see that knowing examples
where the formula is exact helps a lot; otherwise, it still gives some information,
but probably not precise enough.
One can thus object because the only cases where we know that the formula is exact
are the two examples given above (truncated planes and $\bY$-sets), plus
their products by orthogonal $(d-2)$-planes. This observation is right, but should 
probably be tempered by the fact that we do not know so many minimal
cones, and there are many places, in particular when $d=2$, 
where a sliding minimal set looks like one of the examples above.

Notice also that although we do not exclude the case when $0 \in L$,
Theorem \ref{t1.2} is not new in this case; we get that the shade $S$
is reduced to $L$, so $F$ coincides with $\theta_0$ and we rediscover \eqref{1.8}.
Similarly, when $r < \dist(0,L)$, $S \cap B(0,r)$ is empty, and we 
rediscover the more classical fact that $\theta_0$ is nondecreasing
for minimal sets (with no boundary condition), see for instance
\cite{Holder}, and which is true in much more general contexts
(for instance \cite{Allard}).

When $L$ is less than $d-1$-dimensional, $H^d(S) = 0$ and we get again that
$\theta_0$ is monotone, but this is not an impressive result: the sliding condition
\eqref{1.6}, applied on a set $L$ of dimension $< d-1$, does not seem 
coercive enough to allow many sliding minimal sets that are not equal $\H^d$-almost
everywhere to a minimal set.

\ms
A useful monotonicity formula should probably come with a toolbox, so we shall
try to give a few connected tools.  The description of the next results will be slightly 
simpler with the notion of coral sets. 
Let $E$ be closed in $U$; we denote by $E^\ast$ the closed support of 
the restriction of $\H^d$ to $E$. That is, 
\begin{equation} \label{1.11}
E^\ast = \big\{ x\in E \, ; \, \H^d(E \cap B(x,r)) > 0
\text{ for every } r > 0 \big\}.
\end{equation}
Some times we call $E^\ast$ the core of $E$, and we shall say
that $E$ is \underbar{coral} when $E = E^\ast$. It follows easily
from the definition of $E^\ast$ that $\H^d(E \sm E^\ast) = 0$, and since 
$E^\ast \in SM(U,L)$ when $E^\ast \in SM(U,L)$ (see the discussion
in Section \ref{S2}, just below \eqref{2.8}),  
we may restrict our attention to coral minimal sets. The advantage is that 
coral sets are a little cleaner and easier to describe: from the description
above, we see that a general minimal set is the union of a coral minimal set
and a set of $\H^d$ measure $0$, and this negligible set may be ugly. In fact,
if we start from any (coral if we want) minimal set and add to it any $\H^d$-null set,
it follows from the definitions that the resulting set is still minimal as long as it is closed.

Our first complement to Theorem \ref{t1.2} deals with the case when our functional
$F$ is constant on an interval.

\ms
\begin{thm} \label{t1.3}
Let $U \subset \R^n$, $L$, and $E$ be as in Theorem \ref{t1.2},
and suppose in addition that $E$ is coral, and $0 < R_0 < R_1$ are such that
$B(0,R_1) \i U$ and $F$ is constant on the interval $(R_0,R_1)$.
Set $A = B(0,R_0) \sm B(0,R_1)$. Then 
\begin{equation} \label{1.12}
\H^d(A \cap E \cap S) = 0
\end{equation}
and, if $X$ denotes the cone over $A\cap E$, i.e.,
\begin{equation} \label{1.13}
X = \big\{ \lambda x \, ; \, \lambda \geq 0 \text{ and } x\in A\cap E \big\},
\end{equation}
then
\begin{equation} \label{1.14}
A \cap X\sm S \i E.
\end{equation}
If in addition $R_0 < \dist(0,L)$, then $X$ is a coral minimal set
in $\R^n$ (with no boundary condition), and 
\begin{equation} \label{1.15}
\H^d(S \cap B(0,R_1)\sm X) = 0.
\end{equation}
\end{thm}

\ms
Notice that $E \cap A \i X$ by definition, and
then \eqref{1.12} and \eqref{1.14} say that $E$ and $X \sm S$
coincide in $A$, modulo a set of vanishing $\H^d$-measure.

The last part is interesting, because with some luck it will allow us to
identify $X$, and then $A \cap E$. 

In the present case when $L$ is an affine subspace, \eqref{1.15}
says that if its dimension is $d-1$ and if $B(0,R_1)$ meets $L$,
then in fact $X$ contains (the cone over) $L \cap B(0,R_1)$. But
we stated the result like we did because it stays true for more 
general sets $L$. See Theorem \ref{t8.1}.

The additional information \eqref{1.15}
is not necessarily good news, because although it gives some extra information
when the assumptions of the theorem are satisfied, it also says that
the monotonicity formula is not exact when $R_0 < \dist(0,L)$
and $E$ coincides near $\d B(0,R_0)$ with a minimal cone that does
not satisfy \eqref{1.15}. For instance, if $n=3$, $d=2$, $L$ is a line
that does not contain $0$, and $E$ is a plane through the origin that
does not contain $L$, then $F(r)$ is strictly increasing for $r > \dist(0,L)$.

\ms
We will also be interested in almost monotonicity results for almost
minimal sets. We just give a simple statement here, and refer to
Theorem \ref{t7.1} for a more general result, 
but also more complicated to state.

Even this way we need some definitions. 
Almost minimality will be defined in terms of some gauge function
$h : (0,+\infty) \to [0,+\infty]$ (we allow $h(r) = 0$, which corresponds
to minimal sets, and $h(r) = +\infty$, which is a brutal way of saying that
we have no information at that scale). We always restrict to nondecreasing functions
$h$, with
\begin{equation} \label{1.16}
\lim_{r \to 0} h(r) = 0,
\end{equation}
and for our almost monotonicity result we shall assume that
$h$ satisfies the Dini condition
\begin{equation} \label{1.17}
\int_{0}^{r_0} h(r) {dr \over r} < +\infty
\ \text{ for some } r_0>0.
\end{equation}

\ms
\begin{defi} \label{t1.4} 
Let $E$ be a closed set in $U$, such that \eqref{1.1} holds, 
$L$ be a closed set in $U$, and 
let $h : (0,+\infty) \to [0,+\infty]$ be a nondecreasing function such that \eqref{1.16} holds. We say that $E$ is a \underbar{sliding almost minimal set} 
with boundary condition defined by $L$ and gauge function $h$,
and in short we write \underbar{$E \in SAM(U,L,h)$}, when
\begin{equation} \label{1.18}
\H^d(E \cap \wh W) \leq \H^d(\varphi_1(E)\cap \wh W) + r^d h(r)
\end{equation}
whenever $\{ \varphi_t \}$, $0 \leq t \leq 1$, is an acceptable
deformation such that the set $\wh W$ of \eqref{1.5} is contained in
a ball of radius $r$.
\end{defi}

\ms
Other, slightly different, notions exist, and will be treated the same way.
See Section \ref{S2}. Here is the generalization of Theorem \ref{t1.2}
to the classes $SAM(U,L,h)$.

\ms 
\begin{thm} \label{t1.5}
There exist constants $a > 0$ and $\tau > 0$, which depend only 
on $n$ and $d$, with the following property. 
Let $U$ and $L$ be as above, and let $E \in SAM(U,L,h)$ be a sliding almost 
minimal set, for some gauge function $h$ such that \eqref{1.17} holds. Suppose that
\begin{equation} \label{1.19}
\text{$0 \in E^\ast$, $B(0,R_1) \i U$, and $h(R_1) < \tau$,}
\end{equation}
and set
\begin{equation} \label{1.20}
A(r) = \int_{0}^{r} h(t) {dt \over t}
\ \text{ for } 0 < r  \leq  R_1 \, ; 
\end{equation}
then
\begin{equation} \label{1.21}
F(r) \, e^{aA(r)} \text{ is a nondecreasing function of } r\in (0,R_1).
\end{equation}
\end{thm}

\ms
See Theorem \ref{t7.1} for a more general version, where we
allow more general boundary sets $L$, and also slightly different definitions of 
almost minimality.

Notice that because of \eqref{1.17}, $e^{aA(r)}$ tends to $1$
when $r$ tends to $0$, so \eqref{1.21} can really be interpreted
as an almost monotonicity property. In particular, we get that
$\lim_{r \to 0} F(r)$ exists and is finite; it is positive because $E^\ast$ is
locally Ahlfors regular (see Section \ref{S2}).

See for instance Proposition 5.24 on page 101 of  \cite{Holder}
for its analogue when $L = \emptyset$, and Theorem 28.15
in \cite{Sliding} for the case of sliding almost minimal sets, but
with $\theta_0$ and balls centered on $L$.

\ms
The second important element of our toolbox says that when 
$E$ is as above and $F(r)$ varies very little on an interval,
$E$ looks a lot like a sliding minimal set $E_0$ for which $F$ is 
constant on a slightly smaller interval. Hence, in many cases,
Theorem \ref{t1.3} says that $E_0$ and $E$ look like truncated minimal cones.
Here is the simplest result that corresponds to an interval
$[0,r_1)$.

\ms
\begin{thm} \label{t1.6}
For each choice of $L$ and $r_1>0$ and each small $\tau > 0$,
we can find $\varepsilon > 0$, which depends only on $\tau$, $n$, $d$, $L$, 
and $r_1$, with the following property. 
Let $E \in SAM(U,L,h)$ be a coral sliding almost minimal set in the open set $U$,  
with boundary condition defined by $L$ and some nondecreasing gauge function $h$. Suppose that
\begin{equation} \label{1.22}
\text{ 
$B(0,r_1) \i U$ and $h(r_1) < \varepsilon$,}
\end{equation}
and 
\begin{equation} \label{1.23}
F(r_1) \leq \varepsilon + \inf_{0 < r < 10^{-3} r_1} F(r).
\end{equation}
Then there is a coral set $E_0 \in SM(B(0,r_1),L)$ (i.e., $E_0$ is sliding minimal 
in $B(0,r_1)$, with boundary condition defined by $L$), such that
\begin{eqnarray} \label{1.24}
\text{ the analogue of $F$ for the set $E_0$ is constant on } (0,r_1),
\end{eqnarray}
\begin{equation} \label{1.25}
\dist(y, E_0) \leq \tau r_1 \ \text{ for } y \in E\cap B(0,(1-\tau)r_1),
\end{equation} 
\begin{equation} \label{1.26}
\dist(y,E) \leq \tau r_1 \ \text{ for } y \in E_0 \cap B(0,(1-\tau)r_1),
\end{equation}
and
\begin{eqnarray} \label{1.27}
\av{\H^d(E \cap B(y,t)) - \H^d(E_0 \cap B(y,t))} \leq \tau r_1^d &&
\nonumber \\
&&\hskip-7cm
\ \text{ for all $y\in \R^n$ and $t>0$ such that }
B(y,t) \i B(0,(1-\tau)r_1).
\end{eqnarray}
\end{thm}

\ms
Notice that we do not need to assume \eqref{1.17} here.
In many cases, in particular if \eqref{1.17} holds, we know that
$\lim_{r \to 0} F(r)$ exists, and  we could just require that 
$F(r_1) \leq \varepsilon + \lim_{r \to 0} F(r)$ instead of \eqref{1.23}.
Finally observe that \eqref{1.27} does not say much when $t$ is much smaller 
than $r$, because the error term in \eqref{1.27} is $\tau r_1^d$, not $\tau t^d$.

See Theorem \ref{t9.1} for a generalization of this to more general boundary sets 
$L$, and Theorem \ref{t9.7} for the analogue of Theorems \ref{t1.6} and \ref{t9.1}
when $F$ is only assumed to be nearly constant on some interval $(r_0,r_1)$,
$r_0 > 0$.

Once we have Theorem \ref{t9.1}, we can use Theorem \ref{t1.3}
or Theorem \ref{t8.1} to get more information on $E_0$; we do not do this
in this introduction because the amount of information that we get depends on $L$,
in particular through $\dist(0,L)$ and $r_1$.

Theorems \ref{t1.6} and  \ref{t9.1} generalize
Proposition 7.24 in \cite{Holder} (a version without boundary set $L$)
and Proposition 30.19 in \cite{Sliding} (a version with the density
$\theta_0$ and balls centered on $L$). All these results are fairly easy to
obtain by compactness, because we have results that say that
limits of almost minimal sets, with a fixed gauge function $h$, are
almost minimal with the same gauge function.

It is unpleasant that, in both Theorems \ref{t1.6} and \ref{t9.1}, 
$\varepsilon$ depend on the specific choice of $L$ and $r_1$.
In the present case where we assume that $L$ is an affine subspace, 
the dependence is in terms of $r_1^{-1}\dist(0,L)$, and we can try to eliminate 
this distance from the compactness argument that leads to Theorem \ref{t1.6}.
Still we get three different regimes. When $r_1 < \dist(0,L)$, the boundary condition 
does not play a role and we can use results from \cite{Holder}. 
When $\dist(0,L)<< r_1$,
we shall find it more convenient in practice to reduce to the case when $0 \in L$,
and then use results from \cite{Sliding}. So we concentrate on the intermediate case when
$\dist(0,L) < r_1 < C \dist(0,L)$, with $C$ large, and then we get an approximation of $E$
by a set $E_0 = \overline{X_0 \sm S}$, where $X_0$ is a minimal cone
(without boundary condition) and $S$ is the shade of an affine subspace.
This is Corollary \ref{t9.3}, where we also include the case when $L$ is not exactly
an affine space, but is very close to one (in a bilipschitz way), to allow subvarieties as well.

\ms
We shall give two rather elementary applications of the results above.
Both concern the behavior of a sliding almost minimal set $E$, with a 
boundary condition defined by a $(d-1)$-dimensional space $L \i \R^n$, 
and in a small ball where $E$ looks simple enough. In both cases, we shall not
try to give an optimal result, but instead explain how the results of the first part
can be useful. A motivation for this type of results is that they provide 
first steps in the study of specific boundary behavior of minimal sets, 
subject to a Plateau condition.
See Figure 13.9.3 on page 134 of \cite{Mo}, or Figure 5.3 in \cite{LM}, 
for a list of conjectured behaviors for a
$2$-dimensional minimal set bounded by a curve.

We start with some notation which is common to the two results.
We are given a $(d-1)$-dimensional vector space $L$ in $\R^n$, and we shall use
some classes of minimal cones. First we denote by $\bP_0 = \bP_0(n,d)$ 
the set of $d$-planes through the origin, and by $\bP= \bP(n,d)$ the set 
of all affine $d$-planes. 

Next, $\bH(L)$ denotes the class of
closed half-$d$-planes bounded by $L$. That is, to get $H \in \bH(L)$ we pick
a $d$-plane $P$ that contains $L$, select one of the two connected components
of $P \sm L$, and let $H$ be the closure of this component.

For the sake of Corollary \ref{t1.8}, we also denote by $\bV(L)$ the set of
\underbar{cones of type $\bV$} bounded by $L$. These are the unions 
$V = H_1 \cup H_2$, where $H_1$ and $H_2$ are elements of $\bH(L)$ that
make an angle at least $2\pi/3$ along $L$. This last means that if
$v_i$ denotes the unit vector which lies in $H_i$ and is orthogonal to 
$L$, then $\langle v_1,v_2\rangle \leq -1/2$. We add this angle constraint
because it is easy to see that if it fails, then $V$ is not a sliding minimal set
with boundary condition given by $L$. We believe that the elements of $\bV(L)$
are sliding minimal sets, but we shall not check this here.

Let us also define the set $\bY_0 = \bY_0(n,d)$ of \underbar{minimal cones 
of type $\bY$} that are centered at $0$. 
We first say that $Y \in \bY_0(n,1)$ when
$Y$ is the union of three half lines that start from $0$, are contained in some 
plane $P=P(Y)$, and make $2\pi/3$ angles with each other at the origin. 
For $d > 1$, $\bY_0(n,d)$ is the set of products $Y \times W$, where 
$Y \in \bY_0(n,1)$ and $W$ is a vector space of dimension $(d-1)$ 
that is orthogonal to the $2$-plane $P(Y)$. Finally, we set
\begin{equation} \label{1.28}
\bY_x(n,d) = \big\{ x+Z \, ; \, Z\in \bY_0(n,d) \big\}
\end{equation}
(the cones of type $\bY$ centered at $x$). By $\bY$-set,
we usually mean an element of any $\bY_x(n,d)$.

We are almost ready for our first application, which tries to say that if 
$E$ is a coral sliding almost minimal set in $B(0,3)$, with a sufficiently small 
gauge function, and if it is close enough in $B(0,3)$ to 
a half plane $H \in \bH(L)$, then  
$E$ is H\"older-equivalent to $H$ in $B(0,1)$. The initial distance from $E$ to $H$
will be expressed in terms of the following very useful, dimensionless local version 
of the Hausdorff distance: when $E$ and $F$ are two closed sets, we set
\begin{equation} \label{1.29}
d_{x,r}(E,F) = {1 \over r} \sup_{y\in E \cap B(x,r)} \dist(y,F)
+ {1 \over r} \sup_{y\in F \cap B(x,r)} \dist(y,E),
\end{equation}
where by convention we set $\sup_{y\in E \cap B(x,r)} \dist(y,F) = 0$ when
$E \cap B(x,r) = \emptyset$, for instance.
As we shall see, we express the conclusion in terms of distances
$d_{x,r}$ and Reifenberg approximation condition; we shall explain why
after the statement.

\ms
\begin{cor} \label{t1.7}
For each small $\tau > 0$ we can find $\varepsilon > 0$,
which depend only on $\tau$, $n$ and $d$, with the following property. 
Let $L$ be a vector $(d-1)$-plane and let $E$ be a coral
sliding almost minimal set in $B(0,3)$, with boundary condition given by $L$ and 
a gauge function $h$ such that \eqref{1.17} holds. 
Suppose that 
\begin{equation} \label{1.30}
\int_0^{3} {h(t) dt \over t} < \varepsilon
\end{equation}
and 
\begin{equation} \label{1.31}
d_{0,3}(E,H) \leq \varepsilon
\ \text{ for some $H \in \bH(L)$.}
\end{equation}
Then 
\begin{equation} \label{1.32}
L \cap B(0,2) \i E,
\end{equation}
and for each $x\in E \cap B(0,2)$ and $0 < r < 1/2$, 
we can find a set $Z = Z(x,r)$ with the following properties:
\begin{equation} \label{1.33}
\text{if $r \leq \dist(x,L)/2$, then $Z(x,r)$ is a plane through $x$;}
\end{equation}
\begin{eqnarray} \label{1.34}
&\,&\text{if $\dist(x,L)/2 < r \leq \tau^{-1} \dist(x,L)$, then}
\nn\\
&\,&\hskip2cm
\text{$Z(x,r)$ is the element of $\bH(L)$ that contains $x$;}
\end{eqnarray}
\begin{equation} \label{1.35}
\text{if $r > \tau^{-1} \dist(x,L)$, 
then $Z(x,r) \in \bH(L)$;}
\end{equation}

\begin{equation} \label{1.36}
d_{x,r}(E,Z) \leq \tau,
\end{equation}
and 
\begin{eqnarray} \label{1.37}
&\,&\big|\H^d(E \cap B(y,t)) - \H^d(Z\cap B(y,t))\big| \leq \tau r^d
\nn\\
&\,&\hskip2cm
\text{ for $y\in \R^n$ and $t>0$ such that } B(y,t) \i B(0,(1-\tau)r).
\end{eqnarray}\end{cor}

\ms
Recall that coral is defined near \eqref{1.11}. We decided to work in $B(0,3)$
to simplify the statement, but by translation and dilation invariance we could easily
get a statement for balls $B(x_0,r_0)$, with $x_0 \in L$ and 
$\int_0^{3r_0} h(t) { dt \over t} < \varepsilon$ in \eqref{1.30}.

Notice that we do not assume that $0 \in E$ (or that $E \cap L \neq \emptyset$);
this comes as a conclusion.

We claim that the conclusion of Corollary \ref{t1.7} (even without \eqref{1.37})
probably implies that $E$ is H\"older-equivalent to $H$ in $B(0,1)$, as in the 
topological disk theorem of Reifenberg \cite{R}. 
More precisely, we claim that for each $\eta > 0$, we should find
$\tau> 0$ such that if the conclusion of Corollary \ref{t1.7}
holds (without \eqref{1.37}), then there is a H\"older homeomorphism $f$ of 
$\R^n$ such that
\begin{equation} \label{1.38}
f(x) = x \ \text{ for } x\in L \text{ and for } x\in \R^n \sm B(0, 2),
\end{equation}
\begin{equation} \label{1.39}
 |f(x)-x| \leq \eta   \ \text{ for } x\in \R^n,
\end{equation}
\begin{equation} \label{1.40}
(1-\eta)  |x-y|^{1+\eta} \leq |f(x)-f(y)| 
\leq (1+\eta) |x-y|^{1-\eta}
\ \text{ for } x, y \in B(0, 3),
\end{equation}
and 
\begin{equation} \label{1.41}
E \cap B(0, 1) \i f(H \cap B(0,1+\eta)) \i E.
\end{equation}
We shall not prove this claim here, and unfortunately do not know of a proof
in the literature. However, a small modification (mostly a simplification)
of the argument in \cite{DDT} should give this, and since this 
is not central to the present paper, we shall leave the claim with no proof
for the moment, and content ourselves with the conclusions of Corollary \ref{t1.7}.

The corollary probably stays true when $L$ is a flat enough smooth submanifold
of dimension $d-1$, but we shall not try to pursue this here; see
Remark \ref{t11.4} though.

The local H\"older regularity that we claim here is probably far from 
optimal; we could expect slightly better than $C^1$ regularity, 
the proof given below obviously does not give this, but at least it is fairly simple.
Here the important issue is near $L$; far from $L$, we can deduce additional regularity 
on $E$ from its flatness (given by Corollary \ref{t1.7}) and regularity results
from \cite{AlmgrenMemoir} or \cite{Allard}, but at this point we cannot exclude that
$E$ turns around $L$ infinitely many times at some places.

Corollary \ref{t1.7} should be the simplest from a series of
results that give a local description of $E$ when we know that
$E$ looks like a given minimal cone in a small ball centered on $L$.
The next case would be when $E$ is close to a $d$ plane through
$L$, or a set of $\bV(L)$. But new ingredients seem to be needed for this; 
the author will try to investigate the special case when $d=2$, 
for which we have more control on the geometry and the list of minimal cones.

\ms
For our second application, we work on $B(0,3)$ again to simplify the statement, 
assume that the coral almost minimal set $E \in SAM(B(0,3),L,h)$ is very close 
to a cone of type $\bV$ in $B(0,3)$ and that the gauge function $h$ 
is small enough, and get some constraints on the behavior of $E$ 
near its singular points (if they exist). 
Since we don't know whether all the plain minimal cones 
(i.e., with no boundary condition) of dimension $d$ in $\R^n$ that have a 
density at most $3\omega_d/2$ are necessarily 
cones of type $\bY$, we restrict to dimensions $n$ and $d$ for which we know 
that this is the case.

\ms
\begin{cor} \label{t1.8}
Suppose that $d=2$, or that $d=3$ and $n=4$. 
For each choice of $N > 1$ and  $\tau > 0$ we can find $\varepsilon > 0$,
which depends only on $N$, $\tau$ and $n$, with the following property. 
Let $L$ be a vector $(d-1)$-plane, and let $E$ be a coral sliding almost 
minimal set in $B(0,3)$, with sliding condition defined by $L$ and 
a gauge function $h$ such that \eqref{1.17} holds. 
Suppose that 
\begin{equation} \label{1.42}
\int_0^{3} {h(t) dt \over t} < \varepsilon
\end{equation}
and that we can find $V \in \bV(L)$ such that
\begin{equation} \label{1.43}
d_{0,3r}(E,V) \leq \varepsilon.
\end{equation}
Then for each $x\in E \cap B(0,1) \sm L$,
\begin{equation} \label{1.44}
\theta_x(0) := \lim_{r \to 0} r^{-d} \H^d(E \cap B(x,r))
\leq {3 \omega_d \over 2} + \tau.
\end{equation}
In addition, let $x\in E \cap B(0,1) \sm L$ be such that 
$\theta_x(0) > \omega_d$, and set $\delta(x) = \dist(x,L)$. 
Then $\delta(x) \leq N^{-1}$,
and if $Y$ denotes the cone of $\bY_x(n,d)$ (see the definition
\eqref{1.28}) that contains $L$, and 
\begin{equation} \label{1.45}
W = \overline{Y \sm S_x}, \ \text{ with } 
S_x = \big\{ y\in \R^n \, ; \, x + \lambda (y-x) \in L 
\text{ for some } \lambda \in [0,1] \big\},
\end{equation}
then 
\begin{equation} \label{1.46}
d_{x,2N\delta(x)}(W,E) \leq \tau
\end{equation}
and 
\begin{eqnarray} \label{1.47}
&\,&\big|\H^d(E \cap B(y,t)) - \H^d(W\cap B(y,t))\big| \leq \tau \delta(x)^d
\nn\\
&\,&\hskip2cm
\text{ for $y\in \R^n$ and $t>0$ such that } B(y,t) \i B(x,2N\delta(r)).
\end{eqnarray}\end{cor}

\ms
Recall that $\omega_d = \H^d(\R^d \cap B(0,1))$ is the density of a $d$-plane.
The existence of the limit $\theta_x(0)$ in \eqref{1.44} is classical; for instance, it follows
from \eqref{1.42} and Proposition 5.24 in \cite{Holder}. Our condition that 
$\theta_x(0) > \omega_d$ is another way to say that $x$ is 
a singular point of $E$, and (because we may take $\tau$ arbitrarily small)
\eqref{1.44} just forbids certain types
of singularities, such as points of type $\bT$ when $d=2$ and $n=3$.

Our assumption on the dimensions is a little strange
and probably too conservative; we shall prove the result as soon as
$n$ and $d$ satisfy the assumption \eqref{10.9}, which says that all the plain
minimal cones with density at most ${3 \omega_d \over 2}$ are $d$-planes 
or elements of $\bY_0(n,d)$; see Proposition \ref{t12.1}.
The assumption \eqref{10.9} is officially satisfied when $d=2$ and 
when $d=3$ and $n=4$ (hence the statement above), but the author
believes that the proof of the case when $d=3$ and $n=4$ that was given
by Luu in \cite{Luu1} also works when $d=n-1$ and $n\leq 6$, hence
Corollary \ref{t1.8} should be valid in these dimensions as well, and maybe some 
other ones.

Corollary \ref{t1.8} gives a good description of $E$ in $B(x,2N\delta(r))$. 
The proof will also show that we have a similar description in balls $B(x,r)$, 
$\delta(x) < r < 2N\delta$ (see Lemma \ref{t12.5}), which we could even 
get form our main statement by taking $\tau$ even smaller, depending on $N$.
In Lemma \ref{t12.4}, we will also show that in the balls
$B(x,r)$, $r < \delta(x)$, $E$ is well approximated by a cone 
$Y(x,r) \in \bY_x(n,d)$.
If $d=2$ we can also use the description of 
$E$ in $B(x,2\delta(r))$, and the local regularity result of \cite{epi},
to show that $E$ is also $C^1$-equivalent to a cone of type
$\bY_x(n,d)$ in $B(x,\delta(x)/2)$, say. When $d=3$ and $n=4$
we can use the local regularity result of \cite{Luu1} instead, and we 
only get that $E$ is H\"older-equivalent to a cone of type $\bY_x(n,d)$
in $B(x,\delta(x)/2)$.

We can also deduce a reasonable description of $E$ in 
the part of $B(x,N\delta(r)/2)$ that lies closer to $L$
than the singular set of $Y$, because \eqref{1.46} allows us 
to apply Corollary \ref{t1.7} there (notice that near $L$,
$W$ coincides with a half $d$-plane).

Our proof will also gives some control on the shape of $E$ in
balls $B(x,r)$, $N\delta(r) \leq \delta(x) \leq {1\over 2}$, 
where it will be shown that $E$ looks a lot like a sliding minimal cone 
with boundary $L$ and density close to $\omega_d$; when $d=2$, for instance,
we believe that this should mean that $E$ is close to a set of $\bV(L)$
in these balls. See Proposition \ref{t12.7} and Remark \ref{t12.8}. 

But unfortunately, with the methods of this paper, it seems very hard 
to get a description of $E$ near the regular points of $E$, and in particular 
in balls that are far from the singular set of $E$ 
(i.e., the points $x\in E$ such that $d(x) > \omega_d$).
In such balls, we do not know that the functional $F(r)$ is nearly constant
(it may increase from $\omega_d$, its value for small $r$, to $3 \omega_d/3$,
its approximate value for $r=1/2$), and we don't know what $E$ looks like then.

An instance of the situation of Corollary \ref{t1.8} is when 
$n=3$, $d=2$, and $E$ looks a lot, in $B(0,3)$, like a plane that contains $L$.
In this case, the following possibility seems to be expected by specialists, as a typical
way for a soap film to leave a curve. The reader may first look at Figure 1.1, 
which we borrowed from J. Sullivan's site, and which explains how a soap film 
may leave a cylindrical boundary. Then Figure 1.2 
tries to show how the picture may be deformed when the inner radius of the 
cylinder gets smaller. In both cases $E$ has a singular point on the cylinder,
near which $E$ is composed ot three walls that are perpendicular to the cylinder
and make $120^\circ$ angles with each other.
The bottom curve that turns around the cylinder would become more and more
vertical, and the triangular wall would get thinner.
At the limit, we would get a large piece of $E$ that crosses $L$ 
tangentially at the origin, plus a thin triangular piece
that connects the upper part of $L$ to the main piece, and meets it along
a curve $\Gamma$ where $E$ has singularities of type $\bY$. 
See Figure 1.3 for a sketch of the limiting set $E$,
that would be sliding minimal, and Figure 1.4 for its sections by some vertical planes.
We also refer to \cite{B} for additional detail on the description.
The initial goal of the author, when starting this paper, was not to control
the curve $\Gamma$, as we can do 
(at least when $d=2$ and when $d=3$ and $n=4$)
as a consequence of Corollary~\ref{t1.8},
but to show that it does not exist, at least when $V$ is a plane 
or the angle between its two branches is larger than $2\pi/3$. 
But apparently this attempt fails. 
\vfill
\vskip-0.4cm
\includegraphics [width=3.8cm]{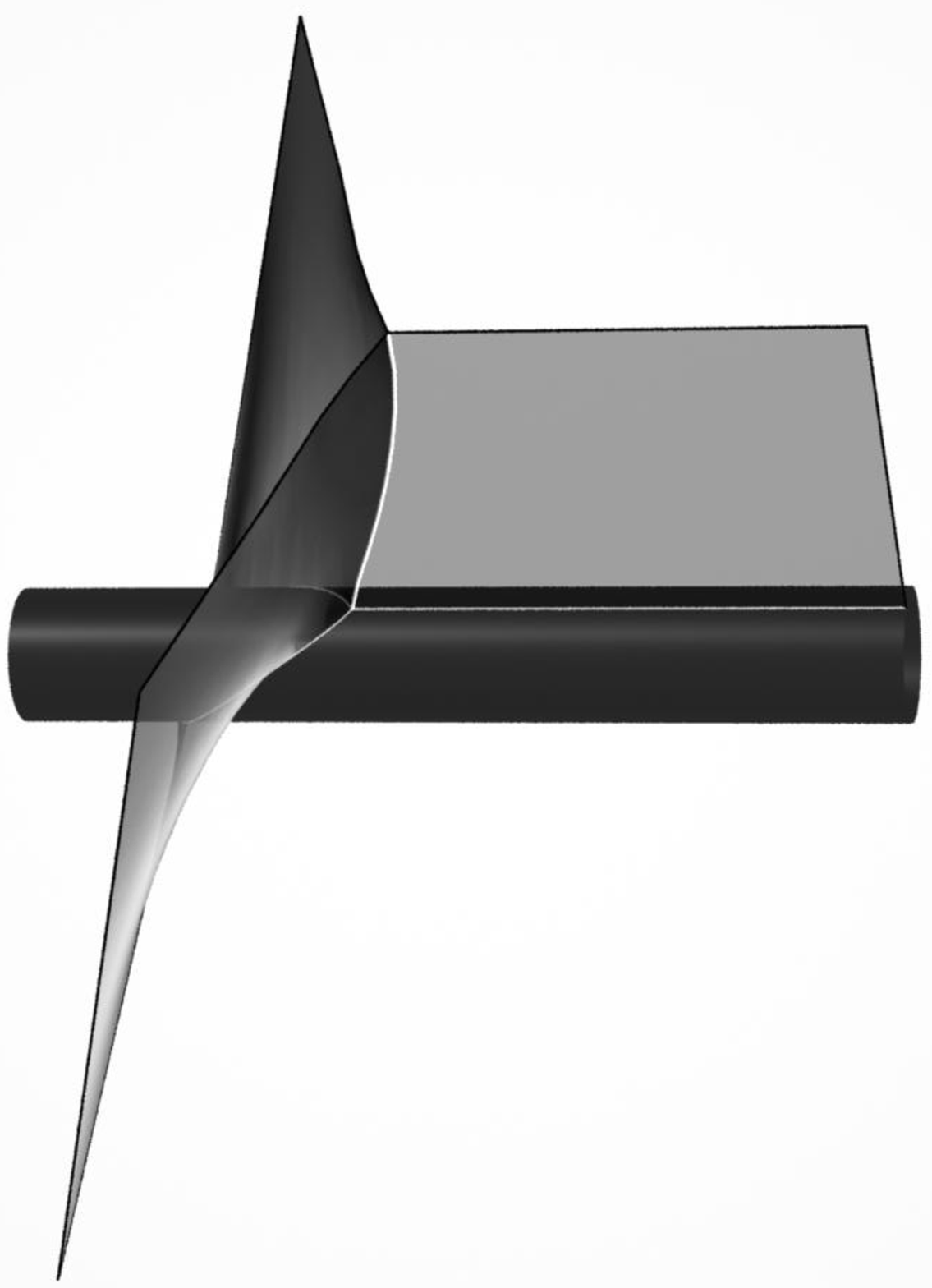}
\hskip1.3cm
\includegraphics[width=11cm]{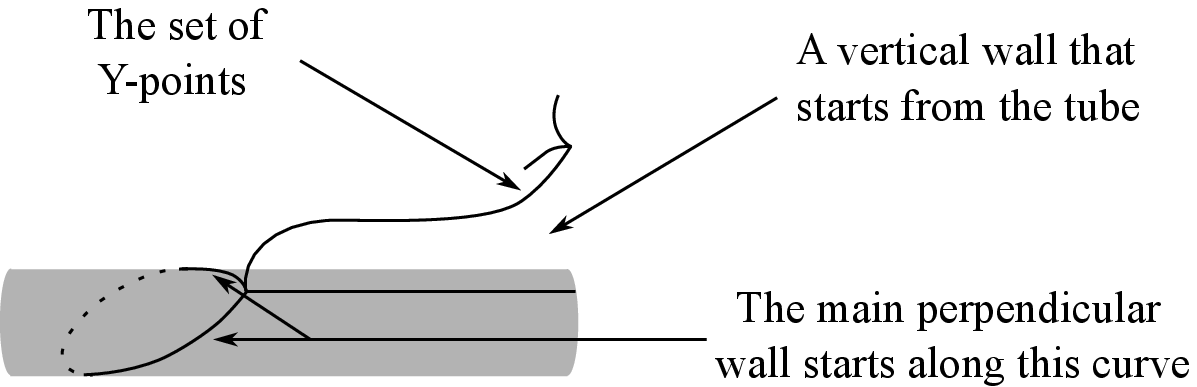}

{\bf Figure 1.1(left).} A soap film leaves a cylinder (Picture by J. Sullivan).
\par
{\bf Figure 1.2 (right).} The same picture with more tilt.

\vskip0.5cm
\centerline{
\includegraphics[width=13cm]{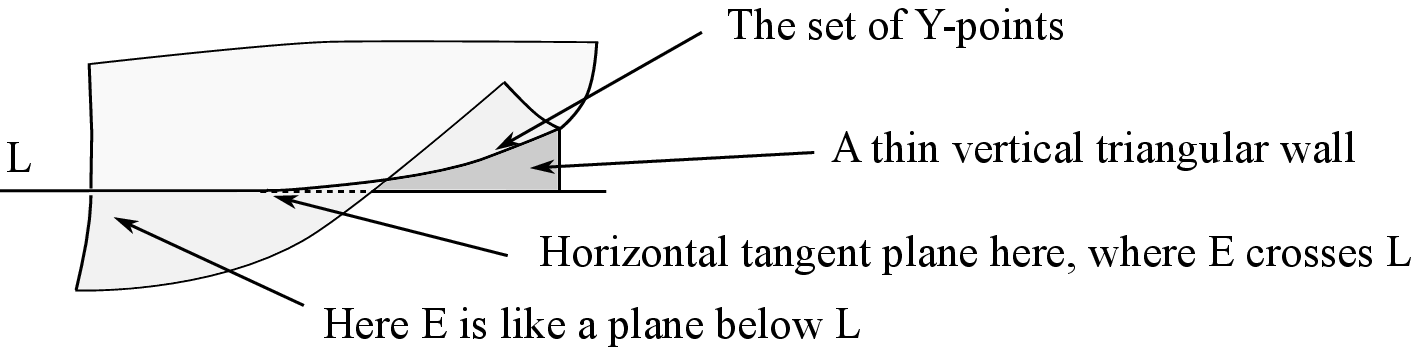}
}
\par
{\bf Figure 1.3.} The limiting set $E$.
\vfill
\ms\ms
\centerline{
\includegraphics[width=9cm]{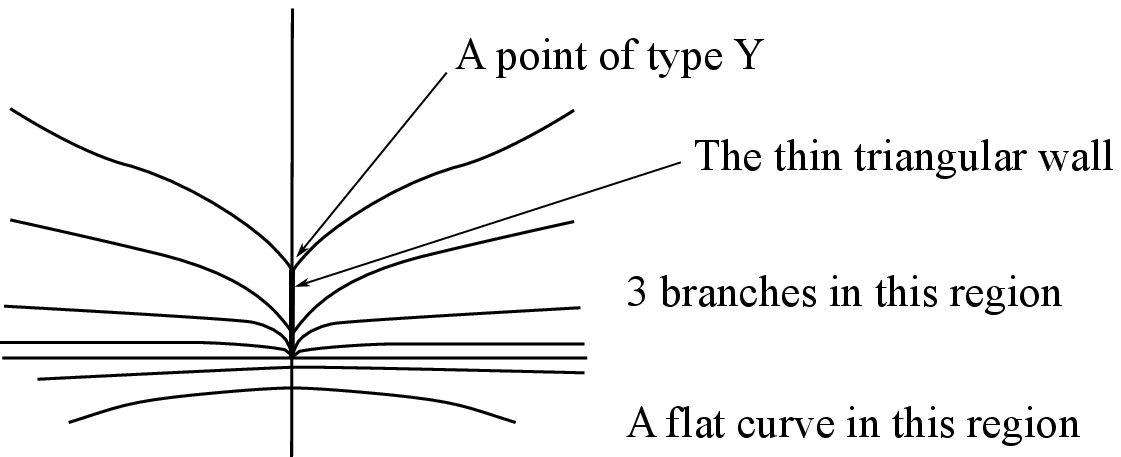}
}
\ms
{\bf Figure 1.4.} Sections of $E$ by vertical planes.

\ms
Another possible behavior of a soap film that may perhaps occur
is the one depicted by Figure 1.5. On the left, $E$ looks like
a set of $\bV(L)$; at the origin, $E$ has a blow-up limit $V_0 \in \bV(L)$, 
possibly with a different angle; on the left,
$E$ is composed of a thin triangular vertical wall, as before, plus
two surfaces that meet it along a singular set of points of type $\bY$.
When the two faces of $V_0$ make $120^\circ$ angles along $L$,
this picture is not shocking at all; the question concerns the case when
this angle is larger, which looks more surprising (why this apparent discontinuity
in angles?), but not more than the example of Figure 1.3. Notice however that in
this case, $E$ is still attached to $L$ on the left of the origin, so maybe it 
does not pull itself down as much.

\ms\ms
\centerline{
\includegraphics[width=13cm]{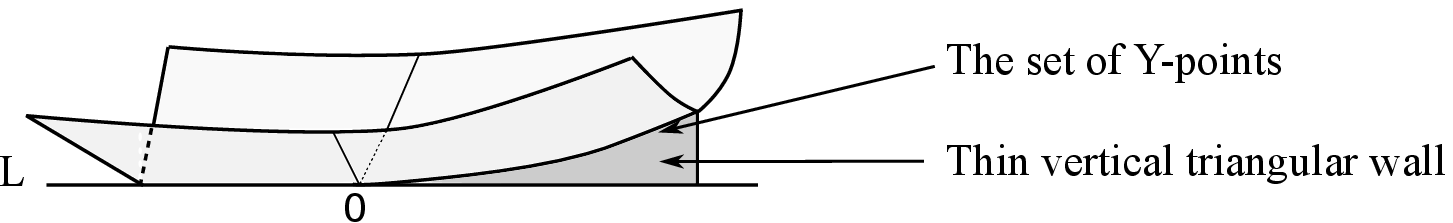}
}
\ms
{\bf Figure 1.5.} A minimal set $E$ with a blow-up limit of type $\bV$ at $0$.

\ms
Maybe we should also say that experiments with soap should not help much here, 
to decide whether the pictures above are realistic, because capillarity plays a 
strong role for thin wires.

\ms
The rest this paper will be devoted to the proof of the various results
mentioned above, often in more generality than in the statements above.

Section \ref{S2} mostly records notation and definitions for of sliding almost 
minimal sets, and rapidly recalls some results from \cite{Sliding}.

In Section \ref{S3} we give general conditions on the boundary sets that will 
allow us to prove a monotonicity property. These should never be a restriction
in practice; the main difficulty will be to find situations where the 
monotonicity property is useful.

Then we prove the monotonicity and near monotonicity properties,
in Sections \ref{S4}-\ref{S7}, with a comparison argument which looks
technical (because we think we need to be careful when taking some limits),
but whose main point is to compare $E$ with a deformation of $E$ in
$B(0,r)$ which is as close as
possible to the cone over $E \cap \d B(0,r)$. We need to add a piece
to that cone (for instance, the cone over $L \cap \overline B(0,r)$), 
because this way we can deform on the new set; this is why we end up adding
a term to the density $\theta_0(r)$ in the definition \eqref{1.10} of $F(r)$.

In Section \ref{S8} we show that for sliding minimal sets $E$, and when the
functional $F$ is constant on an interval, $E$ comes from a cone in the corresponding
annulus (we may have to remove a piece of the shade of $L$, for instance if 
$E$ is a cone of type $\bY$ that contains $L$, truncated by $L$, and this is
why we refer to Theorem \ref{8.1} for a more precise statement.

In Section \ref{S9} we deduce, from Theorem \ref{t8.1} and a compactness
argument, that $E$ is well approximated by truncated minimal cones when 
it is a sliding almost minimal set and the functional $F$ is almost constant
on an interval. See Theorem \ref{t9.1} for the statement with a fixed 
(but general) boundary and an approximation result in a ball,
and Theorem \ref{t9.7} for a variant of Theorem~\ref{t9.1} where $F$ is nearly
constant on an interval that does not start from the origin, and then the 
approximation only takes place in an annulus.

Corollary \ref{t9.3} is a more precise and uniform version of Theorem \ref{t9.1}
but where $L$ is assumed to be very close to affine subspace, and which is neither
too close nor to far from the origin.

In Section \ref{S10} we prepare the two applications and discuss a few simple
properties (true or to be assumed) of plain minimal cones that will be used
in Sections \ref{S11} and \ref{S12}.

We prove Corollary \ref{t1.7} (the case when $E$ looks like a half plane)
in Section \ref{S11} and Corollary \ref{t1.8} 
(the case when $E$ looks like a $\bV$-set) in Section \ref{S12}.

\ms
\noindent{\bf Acknowledgments.}
The author wishes to thank the Institut Universitaire de France, and the ANR
(programme blanc GEOMETRYA, ANR-12-BS01-0014) for their generous support,
M. Christ and the department of Mathematics of the University of California at Berkeley for
their hospitality during the conception stage of this paper, and J. Sullivan for allowing
him to use Figure 1.1 above. 

\section{Three types of sliding almost minimal sets} \label{S2}

We shall be working with more general sliding almost minimal sets
than described in the introduction. In this section, we describe a
set of assumptions that we import from \cite{Sliding}, and that should be
(more than) general enough for us here.
We do this because we do not necessarily want to restrict to
the case of a unique boundary set $L$ which is an affine subspace
of dimension at most $d-1$, yet did not decide of an optimally nice set of 
assumptions, and need some results from \cite{Sliding} anyway.

The (fairly weak) assumptions presented in this section will be satisfied
if there is a unique boundary $L$, $E$ is a coral sliding almost minimal set,
as described in Definition \ref{t1.4}, and there is a bilipschitz mapping
$\psi : U \to B(0,1) \i \R^n$ such $\psi(L)$ is the intersection with
$B(0,1)$ of a vector subspace of dimension at most $d-1$.

Also, we shall present two minor variants of the definition of sliding
almost minimal sets, relative to the way we do the accounting in \eqref{1.18}.
The reader that would only be interested in the case presented in the
introduction may skip the rest of this section, and will probably not be disturbed 
afterwards.

In \cite{Sliding} (and from now on), we are in fact given a finite collection of 
boundary pieces $L_j$, $0 \leq j \leq j_{max}$ (and not just one as above).
We say that 
\begin{equation} \label{2.1}
\text{the sets $L_j$ satisfy the \underbar{Lipschitz assumption}
in the domain $U$}
\end{equation}
when there is a constant $\lambda$ (a scale normalization)
and a bilipschitz mapping $\psi : \lambda U \to B(0,1)$, such that each of the
sets $\wt L_j = \psi(\lambda L_j) \i B(0,1)$ coincides in $B(0,1)$
with a finite union of (closed) dyadic cubes. 
We allow dyadic cubes of different dimensions in a single $L_j$. 
We refer to Section 2 of \cite{Sliding}, and in particular  
Definition 2.7 and above, for additional detail.

By convention, the first set $L_0$ in \cite{Sliding}
was a larger set than the others, which plays the role of a closed domain 
where things happen. See (1.1) in \cite{Sliding}.
This is a convenience, not a constraint, because we may take
$\wt L_0 = \R^n$ and $L_0 = U$ if there is no need for this.
Because of the constraint below, we shall not use this convenience here,
and use $L_0 = U$. Indeed, for our monotonicity formula to make sense,
we shall need to know that $\H^d(L_j) = 0$ for all the $j \geq 0$
that matter, so we shall always assume that
\begin{equation} \label{2.2}
L_0 = U \text{ and all the dyadic cubes that compose the $\wt L_j$,
$j \geq 1$, are of dimensions} <d.
\end{equation}
For many results in \cite{Sliding}, an additional technical assumption
(namely, (10.7) in \cite{Sliding}) is needed; the reader does not need
to worry, this condition, which concerns faces of dimensions larger
than $d$, is automatically satisfied because of \eqref{2.2}.

The author is aware that \eqref{2.1} is complicated; in particular, it
is not true that if the $L_j$ satisfy the Lipschitz assumption in $U$,
then their restriction to a smaller domain $V \i U$ also
satisfy the Lipschitz assumption. For one thing, $V$ may not be
bilipschitz-equivalent to a ball. This is not a major issue, because we
are interested in local properties, and we can always restrict first
to a domain $U$ for where the $L_j$ satisfy the Lipschitz assumption,
and then apply the desired results. Even when $U$ is a ball and $L$ is an
affine subspace of dimension $< d$, we may have to restrict to a slightly
smaller ball to make sure that \eqref{2.1} and \eqref{2.2} hold
(think about the case when $L$ is nearly tangent to $\d B$), but in the rest
of the paper we shall pretend that this has been done and use results where
 \eqref{2.1} and \eqref{2.2} hold without explaining again. Notice also that
 if $L$ does not meet ${1\over 2}B$, we may deduce information on $E\cap {1\over 2}B$
from results about plain almost minimal sets (with no boundary) anyway.

The definition of an \underbar{acceptable deformation} is the same as above, 
except that we replace the unique condition \eqref{1.6} with
\begin{equation} \label{2.3}
\varphi_t(x) \in L_j  \text{ for } 0 \leq t \leq 1
\ \text{ when } w\in E \cap L_j, 1 \leq j \leq j_{max}.
\end{equation}
We removed $j=0$, because \eqref{2.3} is a tautology for $j=0$,
by \eqref{2.2}. 

In \cite{Sliding} we have three slightly different notions of sliding almost
minimal sets (with the given sets $L_j$ and a given gauge function $h$),
which we recall now.

\ms
\begin{defi} \label{t2.1}
Let $E$ be closed in $U$ and satisfy \eqref{1.1}. For each acceptable 
deformation $\{ \varphi_t \}$, $0 \leq t \leq 1$ 
(with \eqref{1.6} replaced by \eqref{2.3}), set
\begin{equation} \label{2.4}
W_1 = \big\{ x\in E \, ; \, \varphi_1(x) \neq x \big\}.
\end{equation}
Also denote by $r$ the radius of a ball that contains the set
$\wh W$ of \eqref{1.5}. 
We say that $E$ is a (sliding) $A$-almost minimal set 
(and write $E \in SAM(U,L_j,h)$) when
\begin{equation} \label{2.5}
\H^d(W_1) \leq \H^d(\varphi_1(W_1)) + h(r) r^d
\end{equation}
for all choice of acceptable deformations $\{ \varphi_t \}$ and $r$
as above. 

We say that $E$ is $A'$-almost minimal (and write $E \in SA'M(U,L_j,h)$)
when instead
\begin{equation} \label{2.6}
\H^d(E \sm \varphi_1(E)) \leq \H^d(\varphi_1(E)\sm E) + h(r) r^d
\end{equation}
for $\{ \varphi_t \}$ and $r$ as above, and that
$E$ is $A_+$-almost minimal ($E \in SA_+M(U,L_j,h)$) when instead
\begin{equation} \label{2.7}
\H^d(W_1) \leq (1+h(r)) \H^d(\varphi_1(W_1)).
\end{equation}
Let us finally set 
$SA^\ast M(U,L_j,h) = SAM(U,L_j,h) \cup SA'M(U,L_j,h) \cup SA_+M(U,L_j,h)$
(where we don't care about which definition is taken).
\end{defi}

\ms
See Definition 20.2 in \cite{Sliding}. The definition \eqref{1.18}
that we gave above is the definition of an $A'$-almost minimal
set (notice that since $E$ and $\varphi_1(E)$ coincide on
$U \sm \wh W$, this set does not play any role in \eqref{2.6}, or
just refer to (20.7) in \cite{Sliding}).

It turns out that $A_+$-almost minimal implies 
$A$-almost minimal and $A'$-almost minimal
(but in a slightly smaller domain, and with a slightly larger gauge function);
see the comments below (20.7) in \cite{Sliding}. 
Moreover, the two notions of $A$-almost minimal and $A'$-almost minimal
are equivalent (with the same gauge function); see Proposition 20.9
in \cite{Sliding}. In the introduction we chose to give the $A'$-definition because
it looks simpler, but finally this does not matter.

When $h=0$, the three notions coincide, and yield the sliding minimal
sets defined in the introduction.

The main reason why we give all this notation is that we shall be using
some results from \cite{Sliding}, which we quote now. 
For our convenience, we shall always assume that 
\begin{equation} \label{2.8}
\text{$E$ is coral and $E \in SA^\ast M(U,L_j,h)$,}
\end{equation}
which means that $E$ is a coral (sliding) almost minimal set in $U$,
with boundary conditions given by the $L_j$ and the gauge function $h$,
and with any of the three definitions. In some rare occasions, we shall
need to specify. We like to ask $E$ to be coral, because this way we
don't need to worry about the fuzzy set $E \sm E^\ast$, 
with $E^\ast$ as in \eqref{1.11}, because $E = E^\ast$. 
We know that we don't lose anything, because 
Proposition~3.3 in \cite{Sliding} says that $E^\ast$ is almost minimal 
when $E$ is almost minimal, with the same gauge function. 
More precisely, that proposition is stated for
more general quasiminimal sets, but its proof works for almost minimal sets,
and we can also deduce the result for almost minimal sets by comparing
the definitions of quasiminimal sets (Definition 2.3 in \cite{Sliding}) and
almost minimal sets (Definition 20.2).

The first main property of $E$ that we shall use is its local Ahlfors regularity.
There exist constants $\eta_0 > 0$ and $C \geq 1$ (that depend on $U$,
and the $L_j$) such that 
\begin{equation} \label{2.9}
C^{-1} r^d \leq \H^d(E\cap B(x,r)) \leq C r^d
\end{equation}
when $x\in E$ and $r> 0$ are such that
\begin{equation} \label{2.10}
B(x,2r) \i U \ \text{ and } \  
h(2r) \leq \eta_0.
\end{equation}
This is Proposition 4.74 in \cite{Sliding}; since the proposition is stated
for (more general) quasiminimal sets, use the comment below (20.8) in \cite{Sliding},
or compare directly with Definition 2.3 in \cite{Sliding}.

We shall also use the fact that
\begin{equation} \label{2.11}
E \text{ is a rectifiable set of dimension $d$;} 
\end{equation}
see Theorem 5.16 in \cite{Sliding}, and observe that we can trivially reduce
to small balls $B(x,r)$ such that $h(2r)$ is as small as we want, because
the rectifiability of $E$ is a local property.

We shall later quote from \cite{Sliding} various limiting theorems,
but we shall only mention them when we need them.

\section{Specific assumptions on $L$; local retractions}
\label{S3}

We shall now describe the additional assumptions that we shall
make on our sets $L_j$ (or the unique $L$ of the introduction),
so that the proof below runs reasonably smoothly.
We start with some notation.

We assume that $0 \in U$, and for $r > 0$ such that
$\overline B(0,r) \i U$, set
\begin{equation} \label{3.1}
L(r) = \overline B(0,r) \cap \Big\{\bigcup_{1 \leq j \leq j_{max}} L_j \Big\},
\end{equation}
where the $L_j$ are the boundary pieces of Section \ref{S2}.
If we have just one boundary set $L$, as the introduction,
just take $L(r) = \overline B(0,r) \cap L$.

We shall often use the truncated cone
\begin{equation} \label{3.2}
L^\ast(r) = \overline B(0,r) \cap 
\big\{ \lambda z, \, ; \, z\in L(r) \text{ and } \lambda \geq 0 \big\}
\end{equation}
over $L(r)$, and then its trace
\begin{equation} \label{3.3}
L^\circ(r) = \d B(0,r) \cap L^\ast(r)
= \d B(0,r) \cap 
\big\{ \lambda z, \, ; \, z\in L(r) \text{ and } \lambda > 0 \big\}
\end{equation}
(think about the shadow of $L \sm \{0 \}$ on $\d B(0,r)$, with a light at
the origin).

Our first additional assumption is that
\begin{equation} \label{3.4}
\H^d(L^\ast(\rho)) < +\infty \ \text{ for some } \rho > 0.
\end{equation}
Notice that as soon as we have this, then 
\begin{equation} \label{3.5}
\H^d(L^\ast(r)) < +\infty \ \text{ for every } r > 0.
\end{equation}
Indeed, $L^\ast(r) \i L^\ast(\rho)$ for $r \leq \rho$,
and if $r > \rho$, $L^\ast(r)$ is contained in the union of two
cones. The first one is the cone over $L(\rho)$, which we control by \eqref{3.4}.
The second one is the cone $C$ over $L(r) \sm L(\rho)$. But by
\eqref{2.2}, $L(r) \sm L(\rho)$ is contained in a finite union of
bilipschitz images of cubes of dimensions at most $d$; then
$\H^{d-1}(L(r) \sm L(\rho)) < +\infty$, and (for instance by
the area theorem and because $L(r) \sm L(\rho)$ stays away from
the origin) $\H^d(C \cap \overline B(0,r)) < +\infty$; \eqref{3.5}
follows.

If $L(r)$ does not contain the origin, \eqref{3.4} holds trivially. 
But we also want to allow the case when $0$ lies in some $L_j$, $j \geq 1$,
so we include \eqref{3.4} as an assumption. 

\ms
We shall play with the sets $L^\ast(r)$, in particular to deform parts
of $B(0,r)$ onto subsets of $L^\ast(r)$, and it will be good to have
some local retractions.

\ms
\begin{defi} \label{t3.1}
Let $r> 0$ be such that $\overline B(0,r) \i U$. We say that
$r$ \underbar{admits a local retraction} when we can find 
constants $\tau_0 >0$ and $C_0 \geq 1$, and a $C_0$-Lipschitz
function $\pi_0$, defined on the set
\begin{equation} \label{3.6}
R(r,\tau_0) = \big\{ x\in \d B(0,r) \, ; \dist(x,L^\circ(r)) \leq \tau_0\big\},
\end{equation}
such that
\begin{equation} \label{3.7}
\pi_0(x) \in L^\circ(r) \ \text{ for } x\in R(r,\tau_0),
\end{equation}
and
\begin{equation} \label{3.8}
\pi_0(x) = x \ \text{ for } x\in  L^\circ(r).
\end{equation}
\end{defi}

\ms
Our typical assumption will be that almost every $r$ in the
interval of interest admits a local retraction. We won't need 
uniform bounds on $\tau_0$ and $C_0$, because these two constants
will disappear in a limiting argument.

Admittedly, this is not such a beautiful condition, but we only expect to
use our results with fairly simple sets $L$, and then the local retractions
will be easy to obtain. For instance, if $L$ is an affine subspace,
or is convex, $\pi_0$ is very easy to construct.

Probably a Lipschitz retraction on a neighborhood of $L^\ast(r)$
would have been easier to use, but we decided to use retractions on spheres, 
just for the hypothetical case when we would be interested in a set $L$ with a loop,
for which the cone $L^\ast(r)$ then has a small loop at $0$, making retractions of
$L^\ast(r)$ near $0$ hard to get.

\section{The main competitor for monotonicity} \label{S4}

The typical way to obtain monotonicity results like \eqref{1.8}
is to compare $E$ with a cone which has the same trace on a sphere $\d B(0,r)$.
Here we cannot quite do that when the $L_j$ are not cones centered at $0$,
so we shall need to add an extra piece to the cone (essentially, the set 
$L^\sharp$ below).

The assumptions for this section are the following. We work in a closed ball
$\overline B = \overline B(0,r)$, and we assume that there is an open
set $U$ and boundary sets $L_j$, $0 \leq j \leq j_{max}$, such that
$\overline B \i U$, \eqref{2.1},  \eqref{2.2}, and \eqref{3.4} hold, 
and $E$ is a coral sliding almost minimal set in $U$, as in \eqref{2.8}.
Set
\begin{equation} \label{4.1}
L = L(r) = \overline B \cap \Big\{\bigcup_{1 \leq j \leq j_{max}} L_j \Big\},
\end{equation}
(as in \eqref{3.1}),
\begin{equation} \label{4.2}
L^\ast = L^\ast(r) 
= \overline B \cap \big\{ \lambda z, \, ; \, z\in L \text{ and } \lambda \geq 0 \big\}
\end{equation}
(as in \eqref{3.2}), and also
\begin{equation} \label{4.3}
L^\sharp = L^\sharp(r)
= \big\{ \lambda z, \, ; \, z\in L \text{ and } \lambda \in [0,1] \big\}
\i \overline B.
\end{equation}

In this section we just construct some acceptable deformations
in the ball $\overline B$; the idea is to compare $E \cap B$ with
the union of $L^\sharp$ and the cone over $E \cap \d B$,
but we shall not do this exactly, and instead be very prudent and deform
$E \cap B$ to sets that tend to this union.
We shall only see later how to use the deformation of this
section to prove our monotonicity results.

We assume that
\begin{equation} \label{4.4}
\text{ $r$ admits a local retraction,}
\end{equation}
as in Definition \ref{t3.1}. Let us recall what this means with the
present notation. Set
\begin{equation} \label{4.5}
R(\tau) = \big\{ x\in \d B \, ; \, \dist(x,L^\ast \cap \d B) 
\leq \tau \big\}
\end{equation}
for $\tau > 0$. If $L \cap \overline B = \emptyset$, 
we take $L^\ast = \emptyset$ and $R(\tau) = \emptyset$.
Definition \ref{t3.1} gives us a $C_0$-Lipschitz mapping 
$\pi_0 : R(\tau_0) \to L^\ast \cap \d B$,
such that
\begin{equation} \label{4.6}
\pi_0(x) = x \ \text{ for } x\in L^\ast \cap \d B.
\end{equation}

\ms
Our construction will have four parameters, a small $\tau$ 
that controls distances to $L^\ast$, and three radii $r_j$, $j=0, 1, 2$,
with $r \geq r_0 > r_1 > r_2$ (and $r-r_2$ very small).
Later on, we shall take specific values for the $r_j$ and take limits twice.
We shall take $\tau <{1/4} \min(\tau_0, r)$, but rapidly it will tend to $0$.

\ms
Our first task is to use $\pi_0$ to construct a new mapping $\pi$, 
which will be defined on the whole $\d B$. Set
\begin{equation} \label{4.7}
\pi(x) = \pi_0(x) \ \text{ for } x\in R(\tau)
\end{equation}
and
\begin{equation} \label{4.8}
\pi(x) = x \ \text{ for } x\in \d B \sm R(2 \tau).
\end{equation}
In the remaining region $R(2\tau) \sm R(\tau)$, set
\begin{equation} \label{4.9}
\alpha(x) = {\dist(x,L^\ast \cap \d B) \over \tau} - 1 \in [0,1]
\end{equation}
and then
\begin{equation} \label{4.10}
\pi(x) = \alpha(x) x + (1-\alpha(x)) \pi_0(x), 
\end{equation}
which is well defined because $x\in R(\tau_0/2)$.
Let us check that 
\begin{equation} \label{4.11}
\text{$\pi$ is $12C_0$-Lipschitz on $\d B$.}
\end{equation}
First notice that \eqref{4.10} is still valid for $x\in R(2 \tau_0)$,
if we set $\alpha(x) = 0$ on $R(\tau)$ and 
$\alpha(x) = 1$ on $R(2 \tau_0) \sm R(2 \tau)$. 
Next we show that $\pi$ is $12C_0$-Lipschitz on $R(2 \tau_0)$.
Consider $x, y\in R(2 \tau_0)$, and notice that
\begin{eqnarray}  \label{4.12}
|\pi(x)-\pi(y)| &=& \big| [\pi_0(x) - \pi_0(y)]+ \alpha(x)[x - \pi_0(x)]
- \alpha(y)[y - \pi_0(y)]\big|
\nonumber \\
&\leq& C_0 |x-y| + \alpha(x) |x - \pi_0(x) - y + \pi_0(y)|
+ |\alpha(x)-\alpha(y)| |y - \pi_0(y)|
\nonumber \\
&\leq& C_0 |x-y| + \alpha(x) (1+C_0) |x-y| + \tau^{-1}|x-y| |y - \pi_0(y)|,
\end{eqnarray}
where we used \eqref{4.9} to get a bound on $|\alpha(x)-\alpha(y)|$. 
If in addition, $y \in R(4\tau)$, we can find $z\in L^\ast \cap \d B$ 
such that $|z-y| \leq 4 \tau$; then
\begin{equation} \label{4.13}
|y - \pi_0(y)| \leq |y-z| + |\pi_0(z) - \pi_0(y)| \leq (1+C_0) 4\tau
\end{equation}
because $\pi_0(z) = z$ by \eqref{4.6}; altogether
\begin{equation} \label{4.14}
|\pi(x)-\pi(y)| \leq 6 (1+C_0) |x-y| \leq 12 C_0 |x-y|
\ \text{ for $x \in R(2 \tau_0)$ and } y \in R(4\tau).
\end{equation}
We get the same bound when  $x \in R(4\tau)$ and $y \in R(2 \tau_0)$
(just exchange $x$ and $y$ in the estimate). And when both $x$ and $y$ lie in 
$R(2\tau_0) \sm R(4\tau)$, $|\pi(x) - \pi(y)| = |x-y|$ by \eqref{4.8}; thus
$\pi$ is $12 C_0$-Lipschitz on $R(2\tau_0)$.

To complete the proof of \eqref{4.11}, we still need to estimate
$|\pi(x)-\pi(y)|$ when $x \in \d B \sm R(2\tau_0)$.
When $y\in R(2\tau)$,
\begin{eqnarray} \label{4.15}
|\pi(x) - \pi(y)| &=& |x-\pi(y)| \leq |x-y| + |y-\pi(y)|
\leq |x-y| + |y-\pi_0(y)| 
\nonumber \\
&\leq& |x-y| + 4 (1+C_0) \tau \leq |x-y| + 2 (1+C_0) |x-y|
\end{eqnarray}
by \eqref{4.8}, \eqref{4.10}, and \eqref{4.13}, and because 
$\dist(x,L^\ast \cap \d B) \geq 2 \tau_0 \geq 4\tau$
and $\dist(y,L^\ast \cap \d B) \leq 2\tau$.
When $y \in \d B \sm R(2\tau)$, $\pi(x) - \pi(y) = x-y$ 
by \eqref{4.8}, and we are happy too.
So \eqref{4.11} holds.

Next we extend $\pi$ to $B$ by homogeneity, i.e., set
\begin{equation} \label{4.16}
\pi(\lambda x) = \lambda \pi(x)
\ \text{ for } x\in \d B \text{ and } 0 \leq \lambda \leq 1.
\end{equation}
This completes our definition of $\pi$. Notice that 
\begin{equation} \label{4.17}
\text{$\pi$ is $13C_0$-Lipschitz on $\overline B$;}
\end{equation}
this time, the simplest is to compute the radial and tangential
derivatives, and notice that $|\pi(x)| \leq |x|$
(by \eqref{4.10} and because $|\pi_0(x)| = r$ on $\d B$).
Set
\begin{equation} \label{4.18}
B^\ast = \overline B \sm \{ 0 \}
\ \text{ and } \ 
d(x) = \dist\Big({r x \over |x|},L^\ast \cap \d B \Big)
\text{ for } x\in B^\ast;
\end{equation}
we use this to measure a radial distance to $L^\ast$.
Notice that $d$ is locally Lipschitz on $B^\ast$, with
\begin{equation} \label{4.19}
|\nabla d(x)| \leq {r \over |x|}.
\end{equation}
Then set
\begin{eqnarray} \label{4.20}
R_1 &=& \big\{ x\in B^\ast \, ; d(x) \leq \tau \big\}
= \big\{ x\in B^\ast \, ; {r x \over |x|} \in R(\tau) \big\},
\nonumber\\
R_2 &=& \big\{ x\in B^\ast \, ; \tau < d(x) \leq 2\tau \big\}
= \big\{ x\in B^\ast \, ; {r x \over |x|} \in R(2\tau)\sm R(\tau) \big\},
\text{ and}
\\
R_3 &=& B^\ast \sm [R_1 \cup R_2] 
= \big\{ x\in B^\ast \, ; {r x \over |x|} \in \d B\sm R(2\tau) \big\}.
\nonumber
\end{eqnarray}
In the easier special case when $L = \emptyset$ and $L^\ast = \emptyset$,
we just have $R_1 = R_2 = \emptyset$, and we can take $d(x)=+\infty$.
Notice that 
\begin{equation} \label{4.21}
\pi(x) = x \ \text{ for } x\in L^\ast
\end{equation}
by \eqref{4.6}, \eqref{4.7}, and \eqref{4.16},
\begin{equation} \label{4.22}
\pi(x) \in L^\ast \ \text{ for } x\in R_1
\end{equation}
by \eqref{4.4}, \eqref{4.7}, and \eqref{4.16}, and
\begin{equation} \label{4.23}
\pi(x) = x \ \text{ for } x\in R_3,
\end{equation}
by \eqref{4.8} and \eqref{4.16}. Finally record that
\begin{equation} \label{4.24}
|\pi(x)| \leq |x| \ \text{ for } x\in \overline B,
\end{equation}
by \eqref{4.16} and the definition of $\pi$ on $\d B$.

\ms
We are ready to define our final mapping $\varphi = \varphi_1$.
We don't need to restrict to $E$ yet; $\varphi$ will be defined
on $\R^n$. Set
\begin{equation} \label{4.25}
B_j = B(0,r_j), \text{ for } j = 0, 1, 2, \  
A_1 = B_0 \sm B_1, \text{ and } A_2 = B_1 \sm B_2.
\end{equation}
We start slowly and set
\begin{equation} \label{4.26}
\varphi(x) = x \ \text{ for } x\in \R^n\sm B_0.
\end{equation}
Next we set
\begin{equation} \label{4.27}
\varphi(x) = \pi(x) \ \text{ for } x\in \d B(0,r_1),
\end{equation}
and interpolate quietly in the middle. That is, we take
\begin{equation} \label{4.28}
\alpha_1(x) = {|x|-r_1 \over r_0 - r_1}
\ \text{ and } \  \varphi(x) = \alpha_1(x) x + (1-\alpha_1(x)) \pi(x)
\ \text{ for } x\in A_1 = B_0 \sm B_1.
\end{equation}
In the remaining ball $B_1 = B(0,r_1)$, we shall use a function $\hbar$
to contract along radii whenever this is possible. Set
\begin{equation} \label{4.29} 
{\hbar}(x) = \big[ 1 - \tau^{-1} \dist(x,L) 
- \tau^{-1} d(x) \big]_+ \in [0,1]
\ \text{ for } x\in \overline B^\ast
\end{equation}
(where $d$ and $B^\ast$ come from \eqref{4.18} and
$[a]_+$ is a notation for $\max(a,0)$). Thus ${\hbar}\equiv 0$
when $L = \emptyset$. Then define $\varphi$ on $B_2$ by 
\begin{equation} \label{4.30}
\varphi(x) = {\hbar}(x)\pi(x) \ \text{ for } x\in B_2 \sm \{ 0 \}
\end{equation}
and $\varphi(0) = 0$; this makes a continuous function, by \eqref{4.24}
and because $0 \leq {\hbar} \leq 1$ everywhere.
On the remaining annulus $A_2$, we interpolate. That is, we set
\begin{equation} \label{4.31}
\alpha_2(x) = {|x|-r_2 \over r_1 - r_2} \ \text{ for } x\in A_2 = B_1 \sm B_2,
\end{equation}
and take
\begin{equation} \label{4.32}
\varphi(x) = \alpha_2(x) \pi(x) + (1-\alpha_2(x)) {\hbar}(x)\pi(x)
\ \text{ for } x\in A_2.
\end{equation}
Since $d$ has a singularity at the origin, we feel obligated to check that
\begin{equation} \label{4.33}
\text{$\varphi$ is Lipschitz on $\overline B$}.
\end{equation}
Since it is continuous along the boundaries of our different pieces,
it is enough to check that $\varphi$ is Lipschitz on each piece separately.
The various cut-off functions are Lipschitz, so it will be enough to check 
that
\begin{equation} \label{4.34}
{\hbar}(x) \pi(x) \text{ is ${15C_0 r \over \tau}$-Lipschitz on $\overline B$.}
\end{equation}
We know that ${\hbar}\pi$ is locally Lipschitz on $B^\ast$
(by \eqref{4.19} in particular), so we just need to bound the derivative
$D(\hbar\pi)(x)$; but
\begin{eqnarray} \label{4.35}
|D(\hbar \pi)(x)| &\leq& |D\pi(x)| \hbar(x) + |\pi(x)| |D\hbar(x)|
\leq 13 C_0 \hbar(x) + |\pi(x)| \big[ \tau^{-1} + \tau^{-1} |\nabla d(x)|  \big]
\nonumber\\
&\leq& 13 C_0 + \tau^{-1} |x| + \tau^{-1} |x| |\nabla d(x)| 
\leq 13 C_0 +2 \tau^{-1} r 
\leq  {15 C_0 r \over \tau}
\end{eqnarray}
by \eqref{4.17}, \eqref{4.29}, \eqref{4.24}, and \eqref{4.19},
and because $|x| \leq r$ and $\tau \leq r/4$; \eqref{4.34} follows
because ${\hbar} \pi$ is also continuous across $0$.

We now complete the family by taking
\begin{equation} \label{4.36}
\varphi_t(x) = t x + (1-t) \varphi(x)
\ \text{ for $x\in \R^n$ and $t \in [0,1]$,}
\end{equation}
and check that 
\begin{equation} \label{4.37}
\begin{aligned}
&\text{the restriction to $E$ of the $\varphi_t$, $0 \leq t \leq 1$,}
\\&\hskip2cm
\text{forms an acceptable deformation, with $\wh W \i B_0$.} 
\end{aligned}
\end{equation}
By \eqref{4.33} in particular, the mapping $(x,t) \to \varphi_t(x)$
is Lipschitz on $\R^n \times [0,1]$, which takes care of \eqref{1.2}
and \eqref{1.4}; \eqref{1.3} is trivial. All the sets $W_t$, $0 \leq t \leq 1$,
of \eqref{1.5} are contained in $B_0$, by \eqref{4.26}, and  
\begin{equation} \label{4.38}
\varphi_t(W_t) \i \varphi_t(B_0) \i B_0
\end{equation}
by \eqref{4.24} and  because $0 \leq {\hbar}(x) \leq 1$.
Thus $\wh W \i B_0$ and \eqref{1.5} holds.

We are left with \eqref{2.3} (the current replacement for \eqref{1.6}) to check.
So let $1 \leq j \leq j_{max}$ and $x\in E \cap L_j$ be given, and let us
check that $\varphi_t(x) \in L_j$ for $0 \leq t \leq 1$. If
$x \in \R^n\sm B_0$, \eqref{4.26} and \eqref{4.36} say that
$\varphi_t(x)=x$, and we are happy because $x\in L_j$.
Otherwise, notice that $x\in L 
\i L^\ast$, by \eqref{4.1} and
\eqref{4.2}. Then $\pi(x) = x$, by \eqref{4.21},
and ${\hbar}(x) = 1$ by \eqref{4.29}. This yields $\varphi_t(x) = x \in L_j$,
which proves \eqref{2.3}; \eqref{4.37} follows.

\ms
Thus we get one of the formulas \eqref{2.5}-\eqref{2.7}, with $r = r_0$.
Our next task is to estimate the measure of 
$\varphi_1(E \cap B_0) = \varphi(E\cap B_0)$,
which we will cut into many small pieces. 

We start with $B_2$. Notice that by \eqref{4.29} and 
the definitions \eqref{4.18} and \eqref{4.20},
\begin{equation} \label{4.39}
{\hbar}(x) = 0 \ \text{ for } x\in R_2 \cup R_3;
\end{equation}
thus \eqref{4.30} yields $\varphi(x) = 0$ for $x\in B_2 \cap [R_2 \cup R_3]$,
hence
\begin{equation} \label{4.40}
\H^d(\varphi(B_2 \cap [R_2 \cup R_3])) = 0.
\end{equation}
We are left with $B_2 \cap R_1$. Let us show that
\begin{equation} \label{4.41} 
\varphi(B_2 \cap R_1) \i \big\{ z\in L^\ast \, ; \, 
\dist(z,L^\sharp) \leq (C_0+2) \tau \big\}.
\end{equation}
Let $x\in B_2 \cap R_1$ be given. 
If $\dist(x,L^\sharp) \geq \tau$, 
then $\dist(x,L) \geq \tau$ (because $L \i L^\sharp$),
${\hbar}(x) = 0$ by \eqref{4.29}, $\varphi(x) = 0$ by \eqref{4.30},
and we are happy because $R_1$ is not empty, $L^\ast$ and 
$L$ are not empty either, and \eqref{4.3} says that $0 \in L^\sharp$.
So we may assume that
\begin{equation} \label{4.42}
\dist(x,L^\sharp) \leq \tau. 
\end{equation}
Set $x^\circ = {r x \over |x|}$. By \eqref{4.16},
$\pi(x) = {|x| \pi(x^\circ) \over r}$. Since $x\in R_1$,
\eqref{4.20} says that $x^\circ \in R(\tau)$,
then $\pi(x^\circ) = \pi_0(x^\circ) \in L^\ast$ by
\eqref{4.7} and \eqref{4.4}. Since $x^\circ \in R(\tau)$,
we can find $y\in L^\ast \cap \d B$ such that $|y-x^\circ| \leq \tau$
(see \eqref{4.5}), and then
\begin{equation} \label{4.43}
|\pi_0(x^\circ)-x^\circ| \leq |\pi_0(x^\circ)-\pi_0(y)|+|y-x^\circ|
\leq (C_0+1) |y-x^\circ| \leq (C_0+1) \tau
\end{equation}
because $\pi_0(y)=y$ by \eqref{4.6}. Next
\begin{equation} \label{4.44}
\av{\pi(x)-x} = {|x| \over r} \av{\pi(x^\circ)-x^\circ}
= {|x| \over r} \av{\pi_0(x^\circ)-x^\circ}
\leq (C_0+1) \tau
\end{equation}
and so $\dist(\pi(x),L^\sharp) \leq (C_0+2) \tau$ by \eqref{4.42}.
But $\lambda y \in L^\sharp$ when $y\in L^\sharp$ and
$\lambda \in [0,1]$ (by \eqref{4.2}),
so 
\begin{equation} \label{4.45}
\dist(\varphi(x),L^\sharp) = \dist({\hbar}(x)\pi(x),L^\sharp)
\leq \dist(\pi(x),L^\sharp) \leq (C_0+2) \tau
\end{equation}
by \eqref{4.30} and because ${\hbar}(x) \in [0,1]$. By \eqref{4.22},
$\pi(x) \in L^\ast$ and hence also $\varphi(x) = {\hbar}(x) \pi(x) \in L^\ast$;
this completes our proof of \eqref{4.41}.

\ms
Next we consider the (more interesting) interior annulus $A_2$.
The largest piece is $A_2 \cap R_3$. Recall from \eqref{4.23} that
$\pi(x) = x$ for $x\in R_3$. In addition, ${\hbar}(x) = 0$ by \eqref{4.39},
so \eqref{4.32} simplifies and becomes
\begin{equation} \label{4.46}
\varphi(x) = \alpha_2(x) x = {|x|-r_2 \over r_1 - r_2} \, x
\ \text{ for } x\in A_2 \cap R_3,
\end{equation}
by \eqref{4.31}; this is a rather simple dilation that maps $A_2 \cap R_3$ to
$B_1 \cap R_3$. In this region, we have no other option
but to compute $\H^d(\varphi(E \cap A_2 \cap R_3))$ with
the area formula, and later take a limit.

Recall from \eqref{2.11} that $E$ is rectifiable. Then it has an
approximate tangent $d$-plane $P(x)$ at $\H^d$-almost every point $x$. 
In fact, thanks to the local Ahlfors regularity \eqref{2.9}, this approximate tangent 
plane is even a true tangent plane; this is reassuring, but we shall not really
need this remark. The mapping $\varphi$, given by \eqref{4.46}, is smooth,
and we can compute its Jacobian $J(x)$ relative to the plane $P(x)$. 
Again, an approximate differential would be enough, but we are happy that we 
can compute $J(x)$ in terms of $P(x)$. 
Denote by $P'(x)$ the vector space of dimension $d$ parallel to $P(x)$.
The main quantity here is the smallest angle $\theta(x) \in [0,\pi/2]$ between
the line $(0,x)$ and a vector of $P'(x)$. Said in other words,
\begin{equation} \label{4.47}
\cos\theta(x) = \sup\big\{ \langle v, {x \over |x|}\rangle \, ; \, 
v\in P'(x) \text{ and } |v| = 1 \big\}.
\end{equation}
If $E$ were a cone centered at $0$, we would get $\cos\theta(x) = 1$
almost everywhere, for instance.

Denote by $D = D\varphi(x)$ the differential of $\varphi$ at $x$.
Set $e = {x \over |x|}$ and define $\alpha : [r_2,r_1] \to [0,1]$ by 
$\alpha(\rho) = {\rho-r_2 \over r_1 - r_2}$.
From \eqref{4.46} we deduce that for $v\in \R^n$,
\begin{equation} \label{4.48}
D(v) = \alpha_2(x) v + \langle \nabla \alpha_2(x), v\rangle x
= \alpha_2(x) v + \alpha'(|x|) \langle e, v\rangle x
=\alpha_2(x) v + {\langle e, v\rangle \over r_1-r_2} \, x.
\end{equation}
Then we compute the Jacobian $J(x)$. Use \eqref{4.47} to choose a 
first unit vector $v_1 \in P'(x)$, such that $\cos\theta(x) = \langle v, e \rangle$,
and then choose unit vectors $v_2, \ldots, v_d$ such that
$(v_1, \ldots, v_d)$ is an orthonormal basis of $P'(x)$.
By \eqref{4.47} again, 
\begin{equation} \label{4.49}
\langle v_1 + t v_j , e\rangle \leq \cos\theta(x) |v_1 + t v_j|
\end{equation}
for $j \geq 2$ and $t \in \R$. When we take the derivative at $t=0$,
we get that $\langle v_j , e\rangle = 0$.

Now return to \eqref{4.48}. Set $\rho = |x|$; for $j \geq 2$, we get that
\begin{equation} \label{4.50}
D(v_j) =\alpha_2(x) v_j = \alpha(\rho) v_j.
\end{equation}
For $v_1$, we further decompose $v_1$ as
$v_1 = \cos\theta(x) e + \sin\theta(x) w$,
where $w$ is a unit normal vector orthogonal to $e$, and get that
\begin{equation} \label{4.51}
D(v_1) = \alpha_2(x) v_1 + {x \cos\theta(x) \over r_1-r_2}
= \alpha(\rho) [\cos\theta(x) e + \sin\theta(x) w] 
+ {\rho \cos\theta(x) e\over r_1-r_2}.
\end{equation}
Both $e$ and $\sin\theta(x) w =  v_1 - \cos\theta(x) e$ are orthogonal
to the $v_j$, $j \geq 2$, so all the vectors $D(v_j)$, $j \geq 1$, are
orthogonal, and 
\begin{equation} \label{4.52}
J(x) = \prod_{j \geq 1} | D(v_j) |
= \alpha(\rho)^{d-1} | D(v_1) | = \alpha(\rho)^{d-1} \beta(x),
\end{equation}
with
\begin{eqnarray} \label{4.53}
\beta(x) &=& \Big\{ \cos^2\theta(x) \big(\alpha(\rho) 
+ {\rho\over r_1-r_2}\big)^2
+  \sin^2\theta(x) \alpha^2(\rho) \Big\}^{1/2}
\nn\\
&=& \Big\{ \cos^2\theta(x) \big({2\rho -r_2 \over r_1-r_2}\big)^2
+  \sin^2\theta(x) \alpha^2(\rho) \Big\}^{1/2}.
\end{eqnarray}
We now write the area formula for the injective mapping $\varphi$ (see for instance \cite{Federer}):
\begin{equation} \label{4.54}
\H^d(\varphi(E\cap A_2 \cap R_3))
= \int_{E\cap A_2 \cap R_3} J(x) d\H^d(x)
= \int_{E\cap A_2 \cap R_3} \alpha(\rho)^{d-1} \beta(x) d\H^d(x),
\end{equation}
where we set
\begin{equation} \label{4.55}
\rho = |x| \ \text{ and  } \ 
\alpha(\rho) = {\rho-r_2 \over r_1 - r_2}.
\end{equation}
We leave this as it is for the moment, to be evaluated later,
and return to the other pieces of $A_2$, starting with
$A_2 \cap R_2$. In this region, we still have that ${\hbar}(x) = 0$,
by \eqref{4.39}, but we need to keep $\pi(x)$ as it is, and 
\eqref{4.32} only yields
\begin{equation} \label{4.56}
\varphi(x) = \alpha_2(x) \pi(x) = {|x|-r_2 \over r_1 - r_2} \, \pi(x)
\ \text{ for } x\in A_2 \cap R_2.
\end{equation}
Let us again apply the area formula. Let $x \in E$ be such that
$E$ has a tangent plane $P(x)$ at $x$ (as before), but also
$\varphi$ has an approximate differential $D\varphi(x)$ at $x$ in the direction of
$P(x)$; since $\varphi$ is Lipschitz, this happens for $\H^d$-almost
every $x\in E \cap B$; see for instance \cite{Federer},
to which we shall systematically refer concerning the area formula on rectifiable sets.
At such a point $x$, $\pi$ also has an approximate 
differential $D\pi(x)$ in the direction of $P(x)$; we compute as
in \eqref{4.48} and get that for $v\in P'(x)$,
\begin{equation} \label{4.57}
D\varphi(x)(v) = \alpha_2(x) D\pi(x)(v) 
+ \langle \nabla \alpha_2(x),v \rangle \pi(x)
= \alpha_2(x) D\pi(x)(v) 
+ {\langle e,v \rangle \pi(x) \over r_1 - r_2}.
\end{equation}
Let us use the same orthonormal basis $(v_1, \ldots, v_d)$
of $P'(x)$ as before. For $j \geq 2$, we now get that
$D\varphi(x)(v_j) = \alpha_2(x) D\pi(x)(v_j)$ (because
$v_j \perp e$), and we shall remember that $|D\varphi(x)(v_j)| \leq C$ 
for $\H^d$-almost every $x$, where $C$ depends on $C_0$, 
just because \eqref{4.17} says that $\pi$ is $13C_0$-Lipschitz.
For $v_1$ we have an extra term, and we can only say that
\begin{equation} \label{4.58}
|D\varphi(x)(v_1)| \leq C + {|\pi(x)| \over r_1 - r_2}
\leq C + {r \over r_1 - r_2}.
\end{equation}
Then we estimate $J(x)$, the Jacobian of $\varphi$ on $E$ at $x$,
brutally, and get that $J(x) \leq {C r \over r_1 - r_2}$. The
area formula now yields
\begin{equation} \label{4.59}
\H^d(\varphi(E\cap A_2 \cap R_2))
\leq \int_{E\cap A_2 \cap R_3} J(x) d\H^d(x)
\leq {C r \over r_1 - r_2} \, \H^d(E\cap A_2 \cap R_2),
\end{equation}
where we only get an inequality in the first part because do not know
whether $\varphi$ is injective.
This looks large because of ${r \over r_1 - r_2}$, but we hope
that the fact that $A_2$ is quite thin, plus the small angle of
$R_2$, will compensate.

Our next piece is $A_2\cap R_1$, which we decompose again.
Set
\begin{equation} \label{4.60} 
V = \big\{ x\in B \, ; \, \dist(x,L^\sharp) \leq \tau \big\}. 
\end{equation}
On $A_2\cap R_1 \sm V$, we again have that ${\hbar}(x) = 0$
(directly by \eqref{4.29} and because $\dist(x,L) \geq \dist(x,L^\sharp) \geq \tau$), 
so the formula \eqref{4.56} still holds, and we can compute as for $A_2 \cap R_2$. 
We avoid $L^\ast$ for the moment, and we get that
\begin{eqnarray} \label{4.61}
\H^d(\varphi(E\cap A_2 \cap R_1 \sm (V\cup L^\ast)))
&\leq& {C r \over r_1 - r_2} \, \H^d(E\cap A_2 \cap R_1 \sm (V\cup L^\ast))
\nonumber\\
&\leq& {C r \over r_1 - r_2} \, \H^d(E\cap A_2 \cap R_1 \sm L^\ast).
\end{eqnarray} 
Next we claim that
\begin{equation} \label{4.62} 
\varphi(A_2 \cap R_1 \cap V) \i \big\{ z\in L^\ast \, ; \, 
\dist(z,L^\sharp) \leq (C_0+2) \tau \big\}.
\end{equation}
We can repeat the proof that we gave below \eqref{4.42},
up to \eqref{4.45}, which is the first time where we used
the fact that $x\in B_2$ (to compute $\varphi(x)$).
Here \eqref{4.31} and \eqref{4.32} say that 
$\varphi(x) \in [{\hbar}(x)\pi(x),\pi(x)]$,
so we still can write $\varphi(x) = \lambda \pi(x)$ for some
$\lambda \in [0,1]$, and we can replace \eqref{4.45} by
\begin{equation} \label{4.63}
\dist(\varphi(x),L^\sharp) = \dist(\lambda\pi(x),L^\sharp)
\leq \dist(\pi(x),L^\sharp) \leq (C_0+2) \tau.
\end{equation}
Then our claim follows as before.

The last part of $A_2$ is $A_2 \cap L^\ast \sm V$. By \eqref{4.21},
$\pi(x) = x$ on this set, so $\varphi(x) \in [0,x]$, and we just record
that
\begin{equation} \label{4.64}
\varphi(E \cap A_2 \cap L^\ast \sm V)
\i \bigcup_{x \in E\cap A_2 \cap L^\ast} [0,x]
= \big\{ \lambda x \, ; \, x\in E \cap A_2 \cap L^\ast 
\text{ and } \lambda \in [0,1] \big\}.
\end{equation}

\ms
We are left with the contribution of the exterior annulus $A_1$,
where $\varphi$ interpolates between $x$ and $\pi(x)$ (see \eqref{4.28}).
When $x\in R_3$, \eqref{4.23} says that $\pi(x)=x$, so
let us record that
\begin{equation} \label{4.65}
\varphi(x) = x \ \text{ for } x\in A_1 \cap R_3.
\end{equation}
Let us check that 
\begin{equation} \label{4.66}
\varphi \text{ is $28 C_0 \big(1+ {\tau \over r_0 - r_1}\big)$-Lipschitz
on $A_1$.}
\end{equation}
As usual, for $x, y \in A_1$, we write
\begin{eqnarray} \label{4.67}
\varphi(x)-\varphi(y)
&=& \alpha_1(x) x + (1-\alpha_1(x)) \pi(x) - \alpha_1(y) y - (1-\alpha_1(y)) \pi(y)
\nn\\
&=& [\pi(y)-\pi(x)] + \alpha_1(x) [x-\pi(x)] - \alpha_1(y)[y-\pi(y)]
\nn\\
&=& [\pi(y)-\pi(x)] + \alpha_1(x) [x-\pi(x)-y+\pi(y)] 
+ [\alpha_1(x)-\alpha_1(y)][y-\pi(y)].
\end{eqnarray}
Then we observe that $\dist(y,R_3) \leq 2 \tau$ by \eqref{4.20} and
\eqref{4.5}, so we can choose $z\in R_3$ such that $|z-y| \leq 2\tau$,
and
\begin{equation} \label{4.68}
|y-\pi(y)| \leq |y-z| + |\pi(z)-\pi(y)| \leq 14 C_0 |z-y|
\leq 28C_0 \tau
\end{equation}
by \eqref{4.23} and \eqref{4.17}. Thus \eqref{4.67} yields
\begin{equation} \label{4.69}
|\varphi(x)-\varphi(y)| \leq 13C_0 |x-y| + 14 C_0 |x-y| 
+ {28C_0 \tau |x-y| \over r_0-r_1}
\end{equation}
by \eqref{4.28} and \eqref{4.68}; the Lipschitz bound \eqref{4.66}
follows. We deduce from this that
\begin{equation} \label{4.70}
\H^d(\varphi(E \cap A_1 \cap (R_1\cup R_2)))
\leq \big[28 C_0 \big(1+ {\tau \over r_0 - r_1}\big)\big]^d
\H^d(E \cap A_1 \cap (R_1\cup R_2)).
\end{equation}

\ms
Let us summarize our estimates so far. Let us first assume that
$E$ is $A'$-almost minimal (see Definition \ref{t2.1}). 
As a consequence of \eqref{4.37},
we can apply \eqref{2.6} (see above \eqref{4.39}).
We get that
\begin{equation} \label{4.71}
\H^d(E \sm \varphi(E)) \leq \H^d(\varphi(E)\sm E) + h(r_0) r_0^d
\leq \H^d(\varphi(E)\sm E) + h(r) r^d
\end{equation}
because $\varphi_1 = \varphi$ and $h$ is nondecreasing. 
Next observe that $E$ and $\varphi(E)$ coincide on $\R^n \sm B_0$,
because $\varphi(x) = x$ on $\R^n \sm B_0$ and $\varphi(B_0) \i B_0$.
Then add $\H^d(B_0 \cap E \cap \varphi(E))$ to both sides of
\eqref{4.71}. We get that
\begin{equation} \label{4.72}
\H^d(E \cap B_0) \leq \H^d(\varphi(E) \cap B_0) + h(r) r^d
= \H^d(\varphi(E \cap B_0)) + h(r) r^d,
\end{equation}
where the last part holds because $\varphi(x) = x$ on $\R^n \sm B_0$.
Then we add the various pieces from above and get that
\begin{equation} \label{4.73}
\H^d(E \cap B_0) \leq \H^d(L^\sharp) + I_1+ I_2+I_3+I_4+I_5+I_6
+ h(r) r^d,
\end{equation}
where we get no contribution from \eqref{4.40},
\begin{equation} \label{4.74}
I_1 = \H^d(Z_1), \text{ with }
Z_1 = \big\{ z\in L^\ast \sm L^\sharp \, ; \, 
\dist(z,L^\sharp) \leq (C_0+2)\tau \big\}
\end{equation}
comes from \eqref{4.41},
\begin{equation} \label{4.75}
I_2 = \H^d(\varphi(E\cap A_2 \cap R_3))
= \int_{E\cap A_2 \cap R_3} \alpha(\rho)^{d-1} \beta(x) d\H^d(x)
\end{equation}
comes from \eqref{4.54},
\begin{equation} \label{4.76}
I_3 = \H^d(\varphi(E\cap A_2 \cap R_2))
\leq {C r \over r_1 - r_2} \, \H^d(E\cap A_2 \cap R_2)
\end{equation}
comes from \eqref{4.59}, 
\begin{equation} \label{4.77}
I_4 = {C r \over r_1 - r_2} \, \H^d(E\cap A_2 \cap R_1 \sm L^\ast)
\end{equation} 
comes from \eqref{4.61}, the contribution of \eqref{4.62} was already
accounted for in $I_1$, 
\begin{equation} \label{4.78}
I_5 = \H^d(Z_5), \text{ with }
Z_5 = \big\{ \lambda x \, ; \, x\in E \cap A_2 \cap L^\ast
\text{ and } \lambda \in [0,1] \big\} \sm L^\sharp
\end{equation} 
comes from \eqref{4.64}, and we remove $L^\sharp$ because it
was accounted for in \eqref{4.73}, and 
\begin{eqnarray} \label{4.79}
I_6 &=& \H^d(\varphi(E \cap A_1))
\nn\\
&\leq& \H^d(E \cap A_1\cap R_3) + 
\big[28 C_0 \big(1+ {\tau \over r - r_1}\big)\big]^d
\H^d(E \cap A_1 \cap (R_1\cup R_2))
\nn\\
&\leq& [28 C_0 \big(1+ {\tau \over r - r_1}\big)\big]^d \H^d(E \cap A_1)
\end{eqnarray} 
comes from \eqref{4.65} and \eqref{4.70}. We shall estimate all these
terms in the next section.

\ms
If $E$ is $A$-almost minimal, we can of course use 
Proposition 20.9 in \cite{Sliding} to say that 
$E$ is $A'$-almost minimal, with the same gauge function, and use
the estimate above. Since this is a little heavy, we can instead use 
\eqref{2.5} and work a little bit more. 

Since \eqref{2.5} is written in terms of 
$W_1 = \big\{ x\in E \, ; \, \varphi(x) \neq x \big\}$,
we shall need to control $E \sm W_1$. We claim that
\begin{equation} \label{4.80}
E \cap B_0 \sm W_1 \i (E \cap A_1) \cup (E \cap L) \cup \{ 0 \}. 
\end{equation}
Indeed, let $x\in E \cap B_0 \sm W_1$ be given; by definition,
$\varphi(x) = x$. 
We may assume that $x \notin L$, 
and then \eqref{4.29} says that ${\hbar}(x) < 1$.
Also recall from \eqref{4.24} that $|\pi(x)| \leq |x|$.

If $x\in A_1$, we are happy. 
If $x\in A_2$, \eqref{4.31} and \eqref{4.32} say that $\varphi(x)$ lies
(strictly) between ${\hbar}(x) \pi(x)$ and $\pi(x)$; hence
$|\varphi(x)| < |\pi(x)| \leq |x|$, and $\varphi(x) \neq x$
(a contradiction). 
If $x\in B_2 \sm \{ 0 \}$, then again 
$|\varphi(x)| = |{\hbar}(x)| |\pi(x)| < |x|$.
Finally, we are also happy if $x=0$. So \eqref{4.80} holds.

Now we want to estimate $\H^d(E \cap B_0)$. We say that
\begin{eqnarray} \label{4.81}
\H^d(E \cap B_0) &=& \H^d(W_1) + \H^d(E \cap B_0 \sm W_1)
\leq \H^d(\varphi(W_1)) + h(r) r^d + \H^d(E \cap B_0 \sm W_1)
\nn\\
&\leq& \H^d(\varphi(E \cap B_0)) + h(r) r^d + \H^d(E \cap B_0 \sm W_1)
\nn\\
&\leq& \H^d(\varphi(E \cap B_0)) + h(r) r^d + \H^d(E \cap A_1)
\end{eqnarray}
because $W_1 \i E \cap B_0$ (by \eqref{2.4}), by \eqref{2.5},
and then by \eqref{4.80} and because $\H^d(L) = 0$ 
(recall that $L$ is at most $(d-1)$-dimensional).
This almost the same thing as \eqref{4.72}, and then we get
the same thing as \eqref{4.73}, except that we need to add the extra term
\begin{equation} \label{4.82}
I_7 = \H^d(E \cap A_1).
\end{equation}
This term will not bother, as it is dominated by $I_6$.

\ms
Finally assume that $E$ is $A_+$-almost minimal. This time we can only use 
\eqref{2.7}, whose error term is $h(r) \H^d(\varphi(W_1))$ instead of $h(r) r^d$.
Notice that if $\H^d(\varphi(W_1)) \geq \H^d(W_1))$, we have 
\eqref{2.5} with no error term, and otherwise we can replace
the error term $h(r) \H^d(\varphi(W_1))$ with the larger $\H^d(W_1))$. 
Thus we get that
\begin{eqnarray} \label{4.83}
\H^d(E \cap B_0) &\leq& \H^d(L^\sharp) + \sum_{j=1}^7 I_j 
+ h(r) \min( \H^d(\varphi(W_1)), \H^d(W_1)))
\nn\\
&\leq&  \H^d(L^\sharp) + \sum_{j=1}^7 I_j 
+ h(r) \H^d(E \cap B)
\end{eqnarray}
when $E$ is sliding $A_+$-almost minimal.

This may be better than \eqref{4.73}, if the origin lies outside of $E$. 
It is not much worse, because of the local Ahlfors regularity of $E$. 
That is, if we assume (as in \eqref{2.10}) that
\begin{equation} \label{4.84}
B(0,2r) \i U \ \text{ and } \ h(2r) \text{ is small enough,}  
\end{equation}
then we get that 
\begin{equation} \label{4.85}
\H^d(E \cap B) \leq C r^d,
\end{equation}
with a constant $C$ that depend only on $n$, $d$, and the $L_j$
(through the constants in \eqref{2.1}), 
regardless of whether $x\in E$ or not, and \eqref{4.83}
is nearly as good as \eqref{4.73}. If we do not want to assume that
$B(0,2r) \i U$, we still get \eqref{4.85}, but with a constant $C$ 
that depends also on $r^{-1} \dist(B,\R^n \sm U)$
(apply $\eqref{2.8}$ to balls of size $\dist(B,\R^n \sm U)$, and then
count how many you need to get an upper bound for $\H^d(E \cap B)$).

\section{We take a first limit and get an integral estimate}
\label{S5}

In this section we integrate the estimate obtained in the previous
section, and take a first limit. We get a bound on 
$\H^d(E \cap B(0,a))$, for $a < r$, in terms of the restriction
of $E$ to an annulus $B(0,b) \sm B(0,a)$; see Lemma \ref{t5.1} below.
Later on, we will let $a$ tend to $b$.

We continue with the fixed radius $r$, and keep the same assumptions
as in Section \ref{S4}. Let $a, b \in [r/2,r]$ be given, with $a < b \leq r$.
We want to do the following computations. For each small $\tau > 0$
and $t \in [a,b]$, we shall write down the main estimate \eqref{4.73} 
(with the added term $I_7$ if $E$ is $A$-almost minimal, 
or even \eqref{4.83} if $E$ is $A_+$-almost minimal), with
\begin{equation} \label{5.1}
r_0 = t, r_1 = t-\tau, \text{ and } r_2 = t-2\tau, 
\end{equation}
average in $t$, and then take the limit when $\tau$ tends to $0$. 

The main reason why we do this slow and cautious limiting process
is that the author was not able to handle the more natural process 
when one would take $r_0=r$, $r_1=r-\tau$, and $r_2 = r-2\tau$,
and go to the limit directly. It seems harder to control the contribution
of the regions $R_2 \cap A_2$ when we do that.

The average of \eqref{4.73} (adapted to $A$-almost minimal sets)
yields
\begin{equation} \label{5.2}
\H^d(E \cap B(0,a)) \leq \H^d(L^\sharp \cap B(0,b)) 
+ \sum_{j=1}^7 J_j(\tau) + h(r) r^d,
\end{equation}
with
\begin{equation} \label{5.3}
J_j(\tau) = {1 \over b-a} \int_{t\in [a,b]} I_j(t,\tau) dt
\end{equation}
for $1 \leq j \leq 7$, and where $I_j(t,\tau)$ is the value of $I_j$
for the choice of $r_0, r_1, r_2$ of \eqref{5.1}. 
For $A_+$-almost minimal sets, we would replace $h(r) r^d$ with
$h(r) \H^d(E \cap B)$, or $C h(r) r^d$, as in \eqref{4.83} or \eqref{4.85}.

Our next task is to take the $J_j(\tau)$ one after the other, and estimate them.
The advantage of $I_1$ (in \eqref{4.74}) is that it does not depend
on $t$. Recall that $\H^d(L^\ast) < +\infty$ (because we assumed 
\eqref{3.4} and \eqref{3.5} follows),
and since the set $Z_1 = Z_1(\tau)$ is contained in $L^\ast$ and
decreases to the empty set when $\tau$ tends to $0$, we get that
\begin{equation} \label{5.4}
\lim_{\tau \to 0} J_1(\tau) = \lim_{\tau \to 0} H^d(Z_1(\tau)) = 0.
\end{equation}
Next consider the integral
\begin{equation}  \label{5.5}
J_2(\tau) = {1 \over b-a} \int_{t\in [a,b]} 
\int_{E\cap A_2(t,\tau) \cap R_3(\tau)} \alpha_t(\rho)^{d-1} \beta_t(x) d\H^d(x)
\end{equation}
(coming from \eqref{4.75}),
where $\alpha_t$ and $\beta_t$ are as in \eqref{4.53} and \eqref{4.55}, 
and we still work with $\rho = |x|$. Here
\begin{equation} \label{5.6}
A_2(t,\tau) = B_1(t) \sm B_2(t) = \big\{ t-2\tau \leq |x| < t-\tau \big\}
\end{equation}
by \eqref{4.25} and \eqref{5.1}, the set $R_3(\tau)$ depends on $\tau$,
but not on $t$ (see \eqref{4.20}), 
\begin{equation} \label{5.7}
\alpha_t(\rho) = {\rho-r_2 \over r_1 - r_2}
= {\rho-t+2\tau \over \tau}
\end{equation}
by \eqref{4.55}, and 
\begin{eqnarray} \label{5.8}
\beta_t(x) &=& \Big\{ \cos^2\theta(x) \big({2\rho -r_2 \over r_1-r_2}\big)^2
+  \sin^2\theta(x) \alpha^2(\rho) \Big\}^{1/2}
\nn\\
&=& \Big\{ \cos^2\theta(x) \big({2\rho -t + 2\tau \over \tau}\big)^2
+  \sin^2\theta(x) \alpha_t^2(\rho) \Big\}^{1/2}
\end{eqnarray}
by \eqref{4.53}. We apply Fubini and get
\begin{equation}  \label{5.9}
J_2(\tau) = {1 \over b-a} \int_{E} f(x) d\H^d(x),
\end{equation}
where in fact we only integrate on 
\begin{equation} \label{5.10}
A(\tau) = \bigcup_{t \in [a,b]} A_2(t,\tau) 
= \big\{ a-2\tau \leq |x| < b-\tau \big\}
\end{equation}
and
\begin{equation} \label{5.11}
f(x) = \int_{t\in [a,b]} \1_{t-2\tau \leq |x| < t-\tau}(x) \, \1_{R_3(\tau)}(x)
\,\alpha_t(\rho)^{d-1} \beta_t(x) dt.
\end{equation}
We want to estimate $\beta_t$. Notice that
$0 \leq \alpha_t(\rho) \leq 1$ when $x\in A_2(t,\tau)$, so
\begin{equation} \label{5.12}
\beta_t(x)^2 
\leq \cos^2\theta(x) \big({2\rho -t + 2\tau \over \tau}\big)^2+ 1
\leq \cos^2\theta(x) \big({t \over \tau}\big)^2+ 1
\leq {b^2 \cos^2 \theta(x) \over \tau^2} + 1
\end{equation}
because $\rho \leq t-\tau$ and $t \leq b$. Since
$\sqrt{1+u^2} \leq 1+ u$ for $u \geq 0$, we get that
\begin{equation} \label{5.13} 
\beta_t(x) \leq 1+ \tau^{-1} b \cos\theta(x)
\ \text{ for } x\in A_2(t,\tau).
\end{equation}
Thus
\begin{equation} \label{5.14}
f(x) \leq \1_{R_3(\tau)}(x) \, [1+\tau^{-1} b \cos\theta(x)] \, g(x),
\end{equation}
with
\begin{eqnarray} \label{5.15}
g(x) &=& \int \1_{t-2\tau \leq |x| < t-\tau}(x) \,\alpha_t(\rho)^{d-1} dt
= \int_{[\rho+\tau,\rho+2\tau]} \Big({\rho-t+2\tau \over \tau}\Big)^{d-1} dt
\nn\\
&=& \int_{u \in [0,\tau]}  \Big({\tau-u \over \tau}\Big)^{d-1} du
= \tau^{1-d}\int_{v\in [0,\tau]} v^{d-1} dv
= {\tau \over d} \, ,
\end{eqnarray}
where we set $t = u+\rho+\tau$ and then $v= \tau-u$.
We return to \eqref{5.9} and get that
\begin{eqnarray}  \label{5.16}
J_2(\tau) &\leq& {1 \over d (b-a)} \int_{E \cap A(\tau) \cap R_3(\tau)} 
[ b \cos\theta(x) + \tau] d\H^d(x)
\nn\\
&\leq& {1 \over d (b-a)} \int_{E \cap A(\tau) \sm L^\ast} 
[ b \cos\theta(x) + \tau] d\H^d(x)
\end{eqnarray}
by \eqref{5.14} and because the set $R_3(\tau)$ never meets $L^\ast$ 
(see \eqref{4.20}). Set $A(a,b) = B(0,b) \sm B(0,a)$. We claim that
\begin{equation} \label{5.17}
\limsup_{\tau \to 0} J_2(\tau) 
\leq {b \over d (b-a)} 
\int_{E \cap A(a,b)\sm L^\ast} \cos\theta(x) d\H^d(x).
\end{equation}
Indeed, for every $c < a$, $A(\tau) \i A(c,b)$ for $\tau$ small,
so \eqref{5.16} yields \eqref{5.17}, but where we integrate on $A(c,b)$ 
instead of $A(a,b)$.
We easily deduce \eqref{5.17} from this, because 
$\H^d(E\cap B(0,b)) < +\infty$.

Recall from \eqref{4.76} that
\begin{equation} \label{5.18}
I_3(t) \leq {C r \over r_1 - r_2} \, \H^d(E\cap A_2(t,\tau) \cap R_2(\tau))
\leq {2C b \over \tau} \H^d(E\cap A_2(t,\tau) \cap R_2(\tau))
\end{equation}
with the same sort of notation as above. We average over $[a,b]$
and get that
\begin{eqnarray} \label{5.19}
J_3(\tau) &\leq& {C b \over \tau (b-a)} \int_{t\in [a,b]} 
\H^d(E\cap A_2(t,\tau) \cap R_2(\tau)) dt
\nn\\
&=& {C b \over \tau (b-a)} \int_{x\in E \cap A(\tau) \cap R_2(\tau)} 
\int_{t\in [a,b]} \1_{x\in A_2(t,\tau)}(t) dt d\H^d(x).
\end{eqnarray}
But $\int_{t\in [a,b]} \1_{x\in A_2(t,\tau)}(t) dt \leq \tau$
by \eqref{5.6}, so we are left with
\begin{equation} \label{5.20}
J_3(\tau) \leq {C b \over (b-a)} \H^d(E \cap A(\tau) \cap R_2(\tau)).
\end{equation}
Now $E \cap A(\tau) \cap R_2(\tau) \i H(\tau)$, where
$H(\tau) = \big\{ x\in E \cap B(0,b) \, ; \, 0 < d(x) \leq 2\tau  \big\}$
(see \eqref{4.20}). Since $\H^d( E \cap B(0,b)) < +\infty$ and the
$H(\tau)$ decrease to the empty set, we get that
\begin{equation} \label{5.21}
\lim_{\tau \to 0} J_3(\tau) = 0.
\end{equation}
We turn to
\begin{equation} \label{5.22}
I_4(t) = {C r \over r_1 - r_2} \, 
\H^d(E\cap A_2(t,\tau) \cap R_1(\tau) \sm L^\ast)
\leq {2C b \over \tau} \, \H^d(E\cap A_2(t,\tau) \cap R_1(\tau) \sm  L^\ast)
\end{equation}
(see \eqref{4.77}). This term can be treated exactly like $I_3(t)$, and we get that
\begin{equation} \label{5.23}
\lim_{\tau \to 0} J_4(\tau) = 0.
\end{equation}
Next we study $I_5(t,\tau) = \H^d(Z_5(t,\tau))$, where
$Z_5(t,\tau)$ is as in \eqref{4.78}. For each $c \in (0,a]$, set
\begin{equation} \label{5.24}
Z(c,b) = \big\{\lambda x \, ; \, x\in E \cap  L^\ast \cap B(0,b) \sm B(0,c)
\text{ and } \lambda \in [0,1] \big\} \sm L^\sharp.
\end{equation}  
If $c < a$, then for $\tau$ small, $Z_5(t,\tau) \i Z(c,b)$ for every $t\in [a,b]$.
We take the average and get that $J_5(\tau) \leq \H^d(Z(c,b))$.
Then we let $c$ tend to $a$, use the fact that $\H^d(L^\ast) < +\infty$,
and get that
\begin{equation} \label{5.25}
\limsup_{t \to 0} J_5(\tau) \leq \H^d(Z(a,b)).
\end{equation}
Recall from \eqref{5.1} that $r_0 - r_1 = \tau$; then by \eqref{4.79}
\begin{equation} \label{5.26}
I_6(t) \leq [56C_0]^d \H^d(E \cap A_1(t,\tau)),
\end{equation}
where by \eqref{5.1}
\begin{equation} \label{5.27}
A_1(t,\tau) = B(0,r_0) \sm B(0,r_1) = B(0,t) \sm B(0,t-\tau).
\end{equation}
Set $A_1(\tau) = \bigcup_{t\in [a,b]} A_1(t,\tau) = 
B(0,b) \sm B(0,a-\tau)$.
We proceed as for $J_3(\tau)$ and use Fubini to estimate
\begin{equation} \label{5.28}
J_6(\tau) = {1 \over b-a} \int_{t\in [a,b]} I_6(t) dt
\leq [56C_0]^d \, {1 \over b-a} \int_{x\in E \cap A_1(\tau)} 
\1_{x\in A_1(t,\tau)}(t) dt d\H^d(x).
\end{equation}
As before, $\int_{t\in [a,b]} \1_{x\in A_1(t,\tau)}(t) dt \leq \tau$,
so 
\begin{equation} \label{5.29}
J_6(\tau) \leq [56C_0]^d \, {\tau \over b-a} \, \H^d(E \cap A_1(\tau))
\leq [56C_0]^d \, {\tau \over b-a}\, \H^d(E\cap B(0,b)).
\end{equation}
Recall from \eqref{4.82} that $I_7(t) = H^d(E \cap A_1(t))$;
this term is smaller than the right-hand side of \eqref{5.26},
so \eqref{5.29} also holds for $J_7(\tau)$. Then of course
\begin{equation} \label{5.30}
\lim_{\tau \to 0} (J_6(\tau) + J_7(\tau)) = 0.
\end{equation}
Let us summarize the estimates from this section as a lemma.

\begin{lem} \label{t5.1}
Let $U$, $E$, and $r$ satisfy the assumptions of Section \ref{S4}.
If $E$ is of type $A$ or $A'$ (as in \eqref{2.5} or \eqref{2.6}),
then for all choices of $r/2 \leq a < b \leq r$,
\begin{equation} \label{5.31}
\begin{aligned}
\H^d(E \cap B(0,a)) \leq \H^d(L^\sharp \cap B(0,b)) 
&+ {1 \over d} {b \over (b-a)} 
\int_{E \cap A(a,b)\sm L^\ast} \cos\theta(x) d\H^d(x)
\\
& + \H^d(Z(a,b)) + h(r) r^d,
\end{aligned}
\end{equation}
with $Z(a,b)$ as in \eqref{5.24}. If $E$ is of type $A_+$
(as in \eqref{2.7}), replace $h(r) r^d$ by
$h(r) \H^d(E \cap B)$, or $C h(r) r^d$, 
as in \eqref{4.83} or \eqref{4.85}. 
\end{lem}

\ms
This follows from \eqref{5.2}, \eqref{5.4}, \eqref{5.17}, \eqref{5.21},
\eqref{5.23}, \eqref{5.25}, and \eqref{5.30}.
\qed

\section{The second limit and a differential inequality}
\label{S6}

In this section we still work with $U$, $E$, and a fixed $r > 0$,
with the same assumptions as in the previous sections, and we 
try to see what happens to the estimates of Lemma \ref{t5.1}
when we fix $b$ (in a suitable Lebesgue set) and let $a$ tend to $b$.

We start with $\H^d(Z(a,b))$, for which no special caution needs
to be taken. Set
\begin{equation} \label{6.1}
Z(b) = \big\{\lambda x \, ; \, x\in E \cap  L^\ast \cap \d B(0,b) 
\text{ and } \lambda \in [0,1] \big\} \sm L^\sharp.
\end{equation}
We claim that
\begin{equation} \label{6.2}
\limsup_{a \to b^-} \H^d(Z(a,b)) \leq \H^d(Z(b)).
\end{equation}
Indeed, the sets $Z(a,b) \cup Z(b)$ all trivially contain $Z(b)$,
are contained in a set $L^\ast$ such that $\H^d(Z^\ast) < +\infty$,
and their monotone intersection is the set $Z(b)$ (see \eqref{5.24});
\eqref{6.2} follows.

We want to evaluate $\H^d(Z(b))$, and for this we shall use the coarea
formula. First we check that
\begin{equation} \label{6.3}
\text{$L^\ast$ is a rectifiable set of dimension $d$.}
\end{equation}
For $k \geq 0$, set $L^k = L \sm B(0,2^{-k})$. By 
the definition \eqref{4.1} and our assumptions
\eqref{2.1} and \eqref{2.2}, $L^k$ is a rectifiable set of dimension
$d-1$, with finite measure. Then the set 
$L_k^\ast = \overline B \cap \big\{ \lambda x \, ; \, \lambda \geq 0
\text{ and } x\in L^k \big\}$ is a rectifiable set of dimension
$d$, with finite measure. Since $L^\ast$ is the union of these sets,
it is rectifiable as well. We already know from \eqref{3.4} and \eqref{3.5}
that $\H^d(L^\ast) < +\infty$, so \eqref{6.3} holds.

Since $Z(b) \i L^\ast$, it is rectifiable as well, with finite measure, and
we can write the coarea formula, applied to $Z(b)$ and the radial projection
$\pi$ defined by $\pi(x) = |x|$. By \cite{Federer}, Theorem~3.2.22,
we have the following identity between measures:
\begin{equation} \label{6.4}
J d\H^d_{|Z(b)} 
= \int_0^b d\H^{d-1}_{| \pi^{-1}(t) \cap Z(b)} dt
= \int_0^b d\H^{d-1}_{|\d B(0,t) \cap Z(b)} dt,
\end{equation}
where $J$ is a Jacobian function that we shall discuss soon, and
\eqref{6.4} means that we can take any positive Borel function, 
integrate both sides of \eqref{6.4} against this function, and get the same
result. Let us say that, if needed, we normalized the Hausdorff measures
so that they coincide with the Lebesgue measures of the same 
dimensions; then we do not need a normalization constant in
\eqref{6.4} (i.e, we can take $J=1$ when $Z(b)$ is a 
$d$-plane through the origin).

Now the Jacobian $J$ is the same as if we computed 
it for $L^\ast$ (either go to the definitions, or observe that \eqref{6.4}
is the restriction of the coarea formula on $L^\ast$). But in $B(0,r)$,
$L^\ast$ coincides with a cone, so its approximate tangent planes, wherever
they exist, contain the radial direction. In these direction, the derivative
of $\pi$ is $\pm 1$, so $J(x) \geq 1$ almost everywhere on $L^\ast$
and $Z(b)$. Also, $\pi$ is $1$-Lipschitz, so $J\leq 1$. Altogether,
$J=1$.

We apply \eqref{6.4} to the function $1$ and get that
\begin{equation} \label{6.5}
\H^d(Z(b)) = \int_0^b d\H^{d-1}(\d B(0,t) \cap Z(b)) dt.
\end{equation}
Set $X = E \cap  L^\ast \cap \d B(0,b)$.
For our main estimate, we can forget about removing $L^\sharp$ in \eqref{6.1},
say that $Z(b)$ is contained in the cone over $X$, and get that
\begin{equation} \label{6.6}
\H^{d-1}(\d B(0,t) \cap Z(b)) \leq \H^{d-1}((t/b) X)
= (t/b)^{d-1} \H^{d-1}(X),
\end{equation}
and hence
\begin{equation} \label{6.7}
\H^d(Z(b)) \leq \int_0^b (t/b)^{d-1} \H^{d-1}(X) dt
= {b \over d} \, \H^{d-1}(X).
\end{equation}
But we want to prepare the case of equality, so we also evaluate 
the intersection with $L^\sharp$. Define a function
$g$ on $X$ by
\begin{equation} \label{6.8}
g(x) = \sup\big\{ t \in [0,1] ; t x\in L \big\}.
\end{equation}
Since $L$ is closed, this a Borel (even semicontinuous) function.
Also (again because $L$ is closed), $g(x) x \in L$, which implies
that $\lambda x \in L^\sharp$ for $0 \leq \lambda \leq g(x)$, 
by \eqref{4.3}.

Let us check that for $0 < t \leq b$,
\begin{equation} \label{6.9}
\d B(0,t) \cap Z(b) = b^{-1} t \big\{ x \in X \, ; \, g(x) < b^{-1} t \big\}.
\end{equation}
If $z\in \d B(0,t) \cap Z(b)$, then by \eqref{6.1}
we can find $x\in X$ and $\lambda \in [0,1]$ such that $z = \lambda x$.
Clearly, $\lambda = b^{-1} t$, and also $\lambda > g(x)$ because
otherwise $z = \lambda x \in L^\sharp$. That is, $g(x) < b^{-1} t$.
Conversely, if $x\in X$ and $g(x) < b^{-1} t$, then $z = b^{-1} t x \in Z(b)$
because otherwise $z=\lambda y$ for some $y\in L$ and $\lambda \in [0,1]$
(by \eqref{4.3}), and then $g(x) \geq b^{-1} t$.

Now \eqref{6.5} yields
\begin{eqnarray} \label{6.10}
\H^d(Z(b)) &=& 
\int_0^b (b^{-1} t)^{d-1} \int_X \1_{g(x) < b^{-1} t} \, d\H^{d-1}(x) dt 
= b^{1-d} \int_X \Big\{\int_{bg(x)}^b  t^{d-1} dt \Big\}d\H^{d-1}(x)
\nn\\
&=&  b^{1-d} \int_X { b^d \over d} (1-g(x)^d) d\H^{d-1}(x)
= {b \over d}\, \H^{d-1}(X) - {b \over d} \int_X g(x)^d d\H^{d-1}(x).
\end{eqnarray}
For the moment, just set
\begin{equation} \label{6.11}
\Delta = {b \over d} \int_X g(x)^d d\H^{d-1}(x),
\end{equation}
and remember that $\Delta \geq 0$ and 
\begin{equation} \label{6.12}
\H^d(Z(b)) = {b \over d}\, \H^{d-1}(X) - \Delta.
\end{equation}

\ms
Next we evaluate integrals on the annulus $A(a,b)$.
We start with integrals on $E \cap L^\ast$. Denote by
$\mu$ the restriction of $\H^d$ to $E \cap L^\ast$, and by
$\nu$ its pushforward by the radial projection $\pi$.
Thus
\begin{equation} \label{6.13}
\nu(K) = \mu(\pi^{-1}(K)) = \H^d(E \cap L^\ast \cap \pi^{-1}(K))
\end{equation}
for Borel subsets $K$ of $[0,r]$. Let us use again the coarea formula
\eqref{6.4}, but now on the set $E \cap L^\ast$. The Jacobian is still $J = 1$,
for the same reason. We apply the formula to the function
$F = \1_{\pi^{-1}(K)}$, and we get that
\begin{eqnarray} \label{6.14}
\nu(K) &=& \H^d(E \cap L^\ast \cap \pi^{-1}(K))
= \int_{E \cap L^\ast} F d\H^d
= \int_{t=0}^r \int_{E \cap L^\ast \cap \d B(0,t)} F d\H^{d-1} dt
\nn\\
&=&  \int_{t=0}^r \1_K(t)  H^{d-1}(E \cap L^\ast \cap \d B(0,t)) dt.
\end{eqnarray}
This proves that $\nu$ is absolutely continuous with respect
to the Lebesgue measure on $[0,r]$, with the density
$f(t) = H^{d-1}(E \cap L^\ast \cap \d B(0,t))$
(the measurability of $f$ is included in the formula).
We shall restrict to points $b\in [r/2,r]$ that are Lebesgue points
for $f$, because for such points $b$,
\begin{equation} \label{6.15}
\lim_{a \to b, \, a < b} \,{1 \over b-a} \int_{[a,b)} f(t) dt = f(b).
\end{equation}
Since
\begin{equation} \label{6.16}
\int_{[a,b)} f(t) dt = \nu([a,b) 
= H^d(E \cap L^\ast \cap \pi^{-1}([a,b))
= H^d(E \cap L^\ast \cap A(a,b))
\end{equation}
by \eqref{6.13} and with the notation introduced above \eqref{5.17},
we get that
\begin{eqnarray} \label{6.17}
\lim_{a \to b, \, a < b} \,{1 \over b-a} H^d(E \cap L^\ast \cap A(a,b))
&=& f(b) = H^{d-1}(E \cap L^\ast \cap \d B(0,b))
\nn\\
&=& H^{d-1}(X) = {d \over b} \H^d(Z(b)) + {d \over b} \Delta  
\geq {d \over b} \H^d(Z(b))
\end{eqnarray}
by various definitions and \eqref{6.12}.

Let us now consider the measures 
\begin{equation} \label{6.18}
\mu_0 = \H^d_{| E \cap \overline B(0,r)}, \, 
\mu_1 = \H^d_{| E \cap \overline B(0,r)\sm L^\ast} = \mu_0 - \mu
\text{ and } \mu_2 = \cos \theta(x) \mu_1, 
\end{equation}
where $\theta(x)$ is the same angle as in \eqref{5.31}, for instance.
For $j= 0,1, 2$, define the pushforward measure $\nu_j$ of $\mu_j$
by $\pi$, as we did for $\mu$ in \eqref{6.13}, and then decompose
$\nu_j$ into its absolutely continuous part $\nu_{j,a}$ and its 
singular part $\nu_{j,s}$. Finally, let $f_j$ denote the density of
$\nu_{j,a}$ with respect to the Lebesgue measure $d\lambda$
on $[0,r]$. 

Observe that $f_j$ can be computed from density ratios, i.e.,
\begin{equation} \label{6.19}
f_j(t) = \lim_{\tau \to 0} \tau^{-1} \mu_j([t-\tau,t))
\end{equation}
for (Lebesgue)-almost every $t\in [0,r]$.
See for instance Theorem 2.12 in \cite{Mattila}, 
in a much more general context. 

Again we shall assume that $b$ is such a point. 
Then the definitions yield
\begin{eqnarray} \label{6.20}
\lim_{a \to b, \, a < b}  \,{1 \over b-a}
\int_{E \cap A(a,b)\sm L^\ast} \cos\theta(x) d\H^d(x)
&=& \lim_{a \to b, \, a < b}  \,{1 \over b-a}\, \mu_2(A(a,b))
\nonumber\\
&=& \lim_{a \to b, \, a < b}  \,{1 \over b-a}\, \nu_2([a,b))
= f_2(b).
\end{eqnarray}
We now let $a$ tend to $b$ in \eqref{5.31}, and get that
\begin{eqnarray} \label{6.21}
\H^d(E \cap B(0,b)) &\leq& \H^d(L^\sharp \cap B(0,b))
+ {b \over d} f_2(b) + \H^d(Z(b)) + h(r) r^d
\nn\\
&=& \H^d(L^\sharp \cap B(0,b))
+ {b \over d} f_2(b) + {b\over d} f(b) -  \Delta  + h(r) r^d
\nn\\
&=& \H^d(L^\sharp \cap B(0,b))
+ {b \over d} f_2(b) +  {b \over d}(f_0(b)-f_1(b)) - \Delta  + h(r) r^d
\\
&\leq& \H^d(L^\sharp \cap B(0,b)) + {b \over d} f_0(b) + h(r) r^d
\nn
\end{eqnarray}
by \eqref{6.20} and \eqref{6.2}, then the second part of \eqref{6.17}, 
then \eqref{6.18}, and \eqref{6.19} and its analogue
\eqref{6.15} for $f$. This holds for Lebesgue-almost every $b \in [r/2,r]$
and when $E$ is of type $A$ or $A'$; when $E$ is of type $A_+$, we modify
the last term $h(r) r^d$ as usual. This comment about $A_+$ will remain valid,
but we shall not always repeat it.

\ms
We are now ready to take a third limit, and let $b$ tend to $r$,
but since additional constraints on $r$ will arise, let us review some
of the notation. Set 
\begin{equation} \label{6.22}
R = \dist(0,\R^n \sm U) = \sup\big\{ r > 0 \, ; \, B(0,r) \subset U \big\},
\end{equation}
define the measure $\overline \mu_0 = \H^d_{| E \cap B(0,R)}$,
and let $\overline \nu_0$ denote the pushforward measure of 
$\overline \mu_0$ by $\pi$, as usual. Thus our measure $\nu_0$ is 
the restriction of $\overline \nu_0$ to $[0,r]$.

Write $\overline \nu_0 = \overline \nu_{0,a} + \overline \nu_{0,s}$,
and denote by $f_0$ the density of the absolutely continuous
part $\overline \nu_{0,a}$. There is no confusion here, the previous
$f_0$ was just the restriction of the new one to $[0,r]$.
In addition to the assumptions from the beginning of Section \ref{S4},
let us assume that $r$ is a Lebesgue point of $f_0$, so that in particular
\begin{equation} \label{6.23}
f_0(r) = \lim_{\tau \to 0} \, {1 \over \tau} \int_{r-\tau}^r f_0(t) dt.
\end{equation}
Then we can take a limit in \eqref{6.21} (which is valid for almost 
every $b\in [r/2,r]$), and get that
\begin{equation} \label{6.24}
\H^d(E \cap B(0,r)) \leq 
\H^d(L^\sharp \cap B(0,r)) + {r \over d} f_0(r) + h(r) r^d.
\end{equation}

\section{The almost monotonicity formula}
\label{S7}

We start with the assumptions for the next theorem.
We are given a domain $U$, which contains $0$, and 
a finite collection of boundary sets $L_j$, 
$0 \leq j \leq j_{max}$. We assume that 
\begin{equation} \label{7.1}
\text{the $L_j$ satisfy \eqref{2.1}, \eqref{2.2}, and \eqref{3.4}.}
\end{equation}
We are also given an interval $(R_0,R_1)$, such that
\begin{equation} \label{7.2}
B(0, R_1) \i U
\end{equation}
and 
\begin{equation} \label{7.3}
\text{ almost every $r \in (R_0,R_1)$ admits a local retraction
(as in Definition \ref{t3.1}).}
\end{equation}
Then we consider a coral sliding almost minimal set $E$ in $U$, with boundary
conditions defined by the $L_j$, and with a gauge function $h$,
as in \eqref{2.8}. 

Finally we assume that $h$ satisfies the Dini condition \eqref{1.17}, and we set,
as in \eqref{1.20},
\begin{equation} \label{7.4}
A(r) = \int_{0}^{r} h(t) {dt \over t}
\ \text{ for } 0 < r  <  R_1 \, .
\end{equation}

\ms
Next we introduce more notation and define our functional $F$. 
First set
\begin{equation} \label{7.5}
L' = \bigcup_{1 \leq j \leq j_{max}} L_j.
\end{equation}
If, as in the introduction, there is only one boundary set $L$,
then $L' = L$. For $0 < r < R_1$, set
\begin{equation} \label{7.6}
L^\sharp(r) =  
\big\{ \lambda z, \, ; \, z\in L' \cap \overline B(0,r) 
\text{ and } \lambda \in [0,1]\big\}
\end{equation}
(this is the same thing as $L^\sharp$ in the previous sections), then
\begin{equation} \label{7.7}
m(r) = \H^d(L^\sharp(r))
\ \text{ and } 
H(r) = d r^d \int_0^r {m(t) dt \over t^{d+1}}.
\end{equation}
With our current assumptions we are not sure that 
$H(r) < \infty$ for $r$ small. For instance, when $d=2$, $L'$ could be a 
Logarithmic spiral in the plane; then $m(r) = a r^2$ for $r$ small,
and the integral in \eqref{7.7} diverges. But if $L'$ is a $C^{1+\varepsilon}$
curve through the origin, $m(r) \leq C r^{2+\varepsilon}$ (only a small sector
is seen), and the integral converge. Of course $H(r)=0$ for $r$ small when 
$L'$ lies at positive distance from the origin.

If $H \equiv +\infty$, the theorem below is true, but useless, 
so we may as well assume that the integral in \eqref{7.7} converges.

The reason why we choose this function $H$ is that it is a solution of
a differential equation. Namely, \eqref{7.7} yields
\begin{equation} \label{7.8}
H'(r) = {d \over r}\, H(r) + d r^d \, {m(r) \over r^{d+1}} 
= {d \over r}\, H(r) + {d \over r} \, m(r)
\end{equation}
for $0 < r < R_1$. 

We shall use the functions $G$ and $F$ defined by
\begin{equation} \label{7.9}
G(r) = \H^d(E \cap B(0,r)) + H(r)
\ \text{ and } \ 
F(r) = r^{-d} G(r)
\end{equation}
for $r \in (0,R_1)$. In our mind, $H$ is a correcting term which we add 
to $\H^d(E \cap B(0,r))$ so that $F$ becomes nondecreasing for minimal sets, 
and almost nondecreasing for almost minimal sets. 
Notice that although it seems complicated, 
it depends only on the geometry of $L'$. We will check later that
in the special case of the introduction where $L'$ is an affine subspace
of dimension $d-1$, $H(r) = \H^d(S \cap B(0,r))$, where the shade $S$ is as in
\eqref{1.9}. See Remark \ref{t7.3}.

\ms
\begin{thm} \label{t7.1}
Let $U$, the $L_j$, $E$, and $h$ satisfy the assumptions \eqref{7.1}-\eqref{7.3}
and \eqref{1.17}. Then there exist constants $a > 0$ and $\tau > 0$, that depend
only on $n$, $d$, and the constants that show up in \eqref{2.1}, with
the following property. Suppose in addition that
\begin{equation} \label{7.10}
0 \in E \ \text{ and } \ h(R_1) \leq \tau.
\end{equation}
Then 
\begin{equation} \label{7.11}
F(r) e^{aA(r)} \text{ is nondecreasing on the interval } (R_0,R_1).
\end{equation}
\end{thm}

\ms
We start with a few remarks, and then we shall prove the theorem.

When $E$ is a minimal set, i.e., $h=0$, then $A=0$ and \eqref{7.11}
says that $F$ is nondecreasing.

In the special case when all the $L_j$ are composed of cones
of dimensions at most $d-1$ centered at the origin, 
$m \equiv 0$, $F$ is the same as $\theta_0$ in \eqref{1.8}, 
and we recover a special case of Theorem 28.7 in \cite{Sliding}. 

When the $L_j$ are almost cones (again, of dimensions
smaller than $d$), $r^{-d} m(r)$ is rather small near $0$, 
and we get some form of near monotonicity for $\theta_0$ as well. The
author did not try to compare this to Theorem 18.15 in \cite{Sliding},
but bets that Theorem 18.15 in \cite{Sliding} is at least as good because the
competitor used in the proof looks more efficient.

When $E$ is an $A_+$-almost minimal set, we do not even need 
\eqref{7.10}, and and we can take $a=d$. 
See Remark \ref{t7.2} below. 

In the special case when $L'$ is an affine subspace, or more generally
when $L'$ has at most one point on each ray starting from the origin,
$H(r)=H^d(S \cap B(0,r))$, where $S$ is the shade
\begin{equation} \label{7.12}
S = \big\{ x\in \R^n \, ; \lambda x \in L' \text{ for some } \lambda \in (0,1] \big\}.
\end{equation}
See Remark \ref{t7.3}.
We thus recover Theorems \ref{t1.2} and \ref{t1.5} as special cases of 
Theorem \ref{t7.1}.

Even though Theorem \ref{t7.1} looks quite general, it is not clear to the
author that all this generality will be useful (in fact, the author did not know
exactly where to stop and ended up not taking tough decisions). 
In the proof, we spend some time
making sure that it works in many complicated situations, but this does not
mean that it is efficient there. For instance, if $n=3$, $d=2$, and 
$L$ consists in two parallel lines at a small distance from each other,
our main competitor uses two half planes bounded by the two lines, and it
would be more efficient to use one of these half planes, plus a small
thin stripe that connects the two lines.

Many of our assumptions (for instance \eqref{2.2}, \eqref{3.4}, or \eqref{7.3})
are used in the limiting process, but do not show up in the final estimate. 
This explains why we do not need uniform bounds for $C_0$ in \eqref{7.3}.

\ms
We shall now complete the proof of the theorem.
The main point will be a differential inequality, that we shall derive from
\eqref{6.24}. Recall from Sections \ref{S4}-\ref{S6}
that \eqref{6.24} holds for almost every $r \in (0,R_1)$
that admits a local retraction, as in \eqref{4.4}.
Because of our assumption \eqref{7.3}, this means almost
every $r \in (R_0,R_1)$.

The set $L^\sharp$ in \eqref{6.24} is the same as
our $L^\sharp(r)$; see \eqref{7.6}, \eqref{4.3}, and \eqref{4.1}. 
Then we get that for almost every $r \in (R_0,R_1)$,
\begin{eqnarray} \label{7.13}
G(r) &=& \H^d(E \cap B(0,r)) + H(r)
\leq \H^d(L^\sharp(r) \cap B(0,r)) + {r \over d} f_0(r) + h(r) r^d + H(r)
\nn\\
&\leq& m(r) + {r \over d} f_0(r) + h(r) r^d + H(r)
= {r \over d} (f_0(r)+H'(r)) + h(r) r^d
\end{eqnarray}
by \eqref{7.9}, \eqref{6.24}, \eqref{7.7}, and \eqref{7.8}.

\ms
This is our main differential inequality. We now need to integrate it to obtain
the desired conclusion \eqref{7.11}.
Let $a$ and $A$ be as in \eqref{7.11}, and set 
\begin{equation} \label{7.14}
g(r) = r^{-d} e^{aA(r)} \ \text{ for } R_0 < r < R_1.
\end{equation}
Since $F(r) = r^{-d}G(r)$ by \eqref{7.9}, 
\eqref{7.11} amounts to checking that 
\begin{equation} \label{7.15}
\text{$gG$ is nondecreasing on $(R_0, R_1)$.}
\end{equation}
Recall from the definitions below \eqref{6.22} that
for $0 < r < s \leq R_1$, 
\begin{eqnarray} \label{7.16}
\H^d(E\cap B(0,s)) - \H^d(E\cap B(0,r))
&=& \overline \nu_0([0,s)) - \overline \nu_0([0,r)) = \overline \nu_0([r,s))
\nn\\
&\geq&  \overline \nu_{0,a}([r,s))
= \int_r^s f_0(u) du.
\end{eqnarray}
By \eqref{7.7}, 
\begin{equation} \label{7.17}
H(s)-H(r) = \int_r^s H'(t) dt.
\end{equation}
Set $d\nu = d\overline \nu_0 + H'(t) dt$ for the moment.
We sum \eqref{7.16} and \eqref{7.17} and get that
\begin{equation} \label{7.18}
G(s) - G(r) = \nu([r,s)),
\end{equation}
by \eqref{7.9}.
Next we check that for $0 < r < t < R_1$,
\begin{equation} \label{7.19}
g(t)G(t) - g(r)G(r)  =  \int_r^t g'(s) G(s) ds + \int_{[r,t)} g(s) d\nu(s).
\end{equation}
To see this without integrating by parts, we use Fubini's theorem to compute 
the integral $I = \int\int_{r \leq s \leq u < t} g'(s) d\nu(u)$ in two different ways.
When we integrate in $s$ first, we get that
\begin{equation} \label{7.20}
I = \int_{r \leq u < t} (g(u)-g(r)) d\nu(u) 
= \int_{r \leq u < t} g(u) d\nu(u) - g(r) [G(t)-G(r)].
\end{equation}
When we integrate in $u$ first, we get
\begin{equation} \label{7.21}
I = \int_{r \leq s < t} g'(s)[G(t)-G(s)] ds = G(t) (g(t)-g(r)) 
- \int_{r \leq s < t} g'(s)G(s)ds.
\end{equation}
We compare the two and get \eqref{7.19}.

The positive measure $\nu$ is at least as large as its absolutely
continuous part $\nu_a$, whose density is $f_0+H'$ (see 
\eqref{7.16}). Thus \eqref{7.19} yields
\begin{equation} \label{7.22}
g(t)G(t) - g(r)G(r)  \geq \int_r^t g'(s) G(s) ds + \int_r^t g(s)(f_0+H')(s)ds.
\end{equation}
But $g'(s) = g(s) \Big[{-d \over s} + a A'(s) \Big]
= g(s) \Big[{-d \over s} + {a h(s) \over s} \Big]$
by \eqref{7.14} and \eqref{7.4}.
So \eqref{7.15} and \eqref{7.11} will follow if we prove that
\begin{equation} \label{7.23}
(f_0+H')(s) \geq \Big[{d \over s} - {a h(s) \over s} \Big] G(s)
\end{equation}
for almost every $s\in (R_0,R_1)$.
But \eqref{7.13} says that
\begin{equation} \label{7.24}
{d \over s} G(s) \leq (f_0+H')(s) + {d \over s} h(s) s^d,
\end{equation}
so it is enough to prove that
\begin{equation} \label{7.25}
{d \over s} h(s) s^d \leq {a h(s) \over s} G(s).
\end{equation}
This is the place where we use our assumption \eqref{7.10}:
the (lower) local Ahlfors regularity \eqref{2.9} yields
\begin{equation} \label{7.26}
G(s) \geq \H^d(E \cap B(0,r/2)) \geq C^{-1} 2^{-d} r^d
\end{equation}
(we used $r/2$ to make sure that \eqref{2.10} holds).
We now choose $a$ sufficiently large, depending on $C$,
and \eqref{7.25} follows from \eqref{7.26}.
This completes the proof of Theorem \ref{t7.1}.
\qed

\ms
\begin{rem} \label{t7.2}
When $E$ is a sliding $A_+$-almost minimal set (as in \eqref{2.7}),
we may drop \eqref{7.10} from the assumptions of Theorem \ref{t7.1}.
Indeed, our initial error term in \eqref{4.83} could be taken to be
$h(r) \H^d(E\cap B)$. When we follow the computations, we see that
we can replace $s^d$ with $\H^d(E\cap B(0,s))$ in 
\eqref{7.24} and \eqref{7.25}. Then we just need to observe that
$G(s) \geq \H^d(E\cap B(0,s))$, and we get \eqref{7.25} if $a \geq d$, 
without using \eqref{7.10} or the local Ahlfors regularity.
Some constraint on $h$, namely the fact that it tends to $0$, is needed
to obtain the qualitative properties of $E$ (mainly the rectifiability; the
local Ahlfors regularity was only comfort), but not the specific bound in
\eqref{7.10}.
\end{rem}

\begin{rem} \label{t7.3}
Suppose that for some $R \in (0,R_1)$, 
\begin{equation} \label{7.27}
\text{the set $L' \cap B(0,R)$ never meets a radius $(0,x]$ more than once.}
\end{equation}
Then
\begin{equation} \label{7.28}
H(r) = \H^d(S \cap B(0,r)) \ \text{ for } 0 < r < R,
\end{equation}
where $S$ is the shade set of \eqref{7.12}.
\end{rem}

\ms
Since we also want to prepare the next remark, we shall
make a slightly more general computation than needed for \eqref{7.28}.
First we define a function $N$.

Set $N(0) = 0$ and, for $x \neq 0$, denote by $N(x)$ the number
of points of $L'$ that lie on the line segment $(0,x]$. In short,
\begin{equation} \label{7.29}
N(x) = \sharp(L' \cap (0,x]) \in [0,+\infty].
\end{equation}
It is not hard to check that $N$ is a Borel function. 
Notice that
\begin{equation} \label{7.30}
S \sm\{ 0 \}= \big\{ x\in \R^n \, ; \, N(x) > 0 \big\}.
\end{equation}
Then set
\begin{equation} \label{7.31}
H_1(r) = \int_{S \cap B(0,r)} N(x) dH^d(x);
\end{equation}
we will later show that $H_1$ could have been used instead of $H$,
but let us not discuss this yet.
When we assume \eqref{7.27}, we have that 
\begin{equation} \label{7.32}
N(x) \leq 1 \ \text{ for } x\in B(0,R)
\end{equation}
and we shall see that $H_1 = H$, but let us not use this assumption for the moment.

Since we prefer to work with finite measures, let us first 
replace $L'$ with $L'_\rho = L' \sm B(0,\rho)$, where the small $\rho >0$
will soon tend to $0$.
Define $N_\rho$, $S_\rho$, and $H_{1,\rho}$ as above,
but with the smaller set $L'_\rho$. 
As for \eqref{7.30} and because $0 \notin S_\rho$, 
\begin{equation} \label{7.33}
S_\rho = \big\{ x\in \R^n \, ; \, N_\rho(x) > 0 \big\}.
\end{equation}

For $0 < r  \leq R_1$,  set
\begin{equation} \label{7.34}
L_\rho(r) = L'_\rho \cap \overline B(0,r) 
= L' \cap \overline B(0,r) \sm B(0,\rho) 
\end{equation}
and   
\begin{equation} \label{7.35}
L_\rho^\ast(r) = \overline B(0,r) \cap 
\big\{ \lambda z, \, ; \, z\in L_\rho(r) \text{ and } \lambda \geq 0 \big\}.
\end{equation}
Recall that by the description of $L'$ by \eqref{2.1} and \eqref{2.2},
$L'_\rho$ is contained in a finite union of biLipschitz images of
dyadic cubes of dimensions at most $d-1$. Since it also stays far from
the origin, we get that $\H^d(L_\rho^\ast(R_1)) < +\infty$.

Set $\mu_\rho = N_\rho \H^d_{|B(0,R_1)} 
= N_\rho \H^d_{|S_\rho \cap B(0,R_1)}$ (by \eqref{7.33}).
Our nice description of  $L'_\rho$ also implies, with just a little more work, that
$\mu_\rho$ is a finite measure. 
Finally denote by $\nu_\rho$ the pushforward 
measure of $\mu_\rho$ by $\pi$; we want to show that $\nu_\rho$
is absolutely continuous and control its density, and we shall use 
the coarea formula again.

By \eqref{7.12}, $S_\rho \cap B(0,R_1) \i L_\rho^\ast(R_1)$,
which is a rectifiable truncated cone of finite $\H^d$ measure. 
This allows us to apply the coarea formula \eqref{6.4}, 
with $Z(b)$ replaced by $S_\rho$ and $J=1$, as we did near 
\eqref{6.13}-\eqref{6.14} with the set $E \cap L^\ast$. 
That is, for every nonnegative Borel function $h$ on $B(0,R_1)$, 
\begin{equation} \label{7.36}
\int_{S_\rho \cap B(0,R_1)} h(x) d\H^d(x)
= \int_{t=0}^{R_1} \int_{S \cap \d B(0,t)} h(x) d\H^{d-1}(x) dt.
\end{equation}
We apply this with $h = \1_{\pi^{-1}(K)} N_\rho $, where $K$ is a 
Borel subset of $[0,R_1]$, and get that
\begin{eqnarray}  \label{7.37}
\nu_\rho(K)  &=& \int_{S_\rho \cap B(0,R_1)} \1_{\pi^{-1}(K)}(x) d\mu_\rho(x)
= \int_{S_\rho \cap B(0,R_1)} \1_{\pi^{-1}(K)}(x) N_\rho(x)  d \H^d(x) 
\nn\\
&=& \int_{t=0}^{R_1} \int_{S \cap \d B(0,t)} 
\1_{\pi^{-1}(K)}(x) N_\rho(x)  d\H^{d-1}(x) dt
\nn\\
&=& \int_{t=0}^{R_1} \1_{K}(t) \int_{S \cap \d B(0,t)} N_\rho(x) d\H^{d-1}(x) dt;
\end{eqnarray}
this proves that $\nu_\rho$ is absolutely
continuous with respect to the Lebesgue measure, and its density is
\begin{equation} \label{7.38}
f_\rho(t) = \int_{S_\rho \cap \d B(0,t)} N_\rho(x) d\H^{d-1}(x).
\end{equation}
The function $f_\rho$ is locally integrable, but we can say a bit more:
since $N_\rho$ is radially nondecreasing, $f_{\rho}$ is also nondecreasing,
so its restriction to any $[0,T]$, $T < R_1$, is bounded.

Let us record that by the analogue of \eqref{7.31} for $H_{1,\rho}$, 
\begin{equation} \label{7.39}
H_{1,\rho}(r) = \int_{S_\rho \cap B(0,r)} N_\rho(x) dH^d(x) = \nu_\rho([0,r))
= \int_0^r f_\rho(t) dt.
\end{equation}

For $r > 0$, define a function 
$g_r$ on $B(0,r)$ by 
\begin{equation} \label{7.40}
g_r(0) = 0 \ \text{ and }\ g_r(x) = N_\rho(xr/|x|) \text{ for } x\neq 0.
\end{equation}
Notice that
\begin{equation} \label{7.41}
g_r(x) \geq N_\rho(x) \ \text{ for } x\in B(0,r),
\end{equation}
just because $N_\rho$ is radially nondecreasing (see the definition \eqref{7.29}).
In addition, we claim that
\begin{equation} \label{7.42}
g_r(x) \geq N_\rho(x) +1 \ \text{ for } x\in B(0,r) \cap L_\rho^\sharp(r) 
\sm L_\rho(r),
\end{equation}
where 
\begin{equation} \label{7.43}
L_\rho^\sharp(r) =  
\big\{ \lambda z, \, ; \, z\in L_\rho(r) 
\text{ and } \lambda \in [0,1]\big\}
\end{equation}
is the analogue of $L^\sharp(r)$ for $L'_\rho$ (see \eqref{7.6}
and \eqref{7.34}).
Indeed, if $x\in B(0,r) \cap L_\rho^\sharp(r) \sm L_\rho(r)$,
\eqref{7.43} says that there exists $z\in L_\rho(r)$ and $\lambda \in [0,1]$
such that $x= \lambda z$. In addition, $\lambda \neq 1$ because 
$x \notin L_\rho(r)$. Thus $z \in (x,rx/|x|]$, and 
by \eqref{7.40} and \eqref{7.29}, 
$g_r(x) = N_\rho(x,rx/|x|) > N_\rho(x)$, as claimed.
This yields
\begin{eqnarray}  \label{7.44}
H_{1,\rho}(r) + \H^d(L_\rho^\sharp(r) \cap B(0,r)) &=& 
\int_{S_\rho \cap B(0,r)} N_\rho(x) dH^d(x) 
+ \int_{L_\rho^\sharp(r)\cap B(0,r)} dH^d(x)
\nn\\
&\leq& \int_{B(0,r)} g_r(x) dH^d(x)
\end{eqnarray}
by \eqref{7.39}, \eqref{7.41}, and \eqref{7.42}, and because 
$\H^d(L_\rho(r)) = 0$.

In the special case when \eqref{7.27} holds and for $r < R$, we claim that in fact
\begin{equation} \label{7.45}
H_{1,\rho}(r) + \H^d(L_\rho^\sharp(r) \cap B(0,r)) =  \int_{B(0,r)} g_r(x) dH^d(x).
\end{equation}
By the first line of \eqref{7.44}, we just need to check that 
$N_\rho(x) + \1_{L_\rho^\sharp(r)}(x) = g_r(x)$
for $H^d$-almost every $x\in B(0,r)$. So let $x\in B(0,r)$ be given.
If $g_r(x) = 0$, then $N_\rho(x) = 0$ by \eqref{7.41}, and $x$ cannot lie in 
$L_\rho^\sharp(r)\sm L_\rho(r)$, by \eqref{7.42}. 
Since $\H^d(L_\rho(r))=0$, we get that 
$N_\rho(x) + \1_{L_\rho^\sharp(r)}(x) = 0$ almost surely.
If instead $g_r(x) \neq 0$, then $(0,xr/|x|]$ meets $L_\rho(r)$ 
by \eqref{7.40} and \eqref{7.29}. 
By \eqref{7.27}, the intersection is unique. Call $z$ the only point of
$(0,xr/|x|] \cap L_\rho(r)$; if $|z| > |x|$, then $N_\rho(x) = 0$,
but $x\in L^\sharp(r) \sm L_\rho(r)$. If $|z| \leq |x|$, then $N_\rho(x) = 1$,
but $x\notin L^\sharp(r) \sm L_\rho(r)$. In both cases 
$N_\rho(x) + \1_{L_\rho^\sharp(r)}(x) = 1$. The claim follows.

We return to the general case. The coarea formula \eqref{7.36}, applied
with $h = g_r \1_{B(0,r)}$, yields
\begin{eqnarray} \label{7.46}
\int_{S_\rho \cap B(0,r)} g_r(x) dH^d(x)
&=& \int_{t=0}^{r} \int_{S_\rho \cap \d B(0,t)} g_r(x) d\H^{d-1}(x) dt
\nn\\
&=& \int_{t=0}^{r} \int_{S_\rho \cap \d B(0,t)} N_\rho(xr/t) d\H^{d-1}(x) dt
\nn\\
&=& \int_{t=0}^{r}  (t/r)^{d-1} \int_{S_\rho \cap \d B(0,r)} 
N_\rho(y) d\H^{d-1}(y) dt = {r \over d} \, f_\rho(r)
\end{eqnarray}
by \eqref{7.40} and \eqref{7.38}. Thus by \eqref{7.44} and \eqref{7.45}
\begin{equation} \label{7.47}
H_{1,\rho}(r) + \H^d(L_\rho^\sharp(r)\cap B(0,r)) \leq {r \over d} \, f_\rho(r)
\end{equation}
for almost every $r\in [0,R_1)$, with equality when 
\eqref{7.27} holds and $0 < r < R$. 

Let us now complete the proof of Remark \ref{t7.3}. 
Set $\varphi = H_{1,\rho}$; by \eqref{7.39} and the paragraph above it,
$\varphi$ is differentiable almost everywhere on $[0,R)$, with a derivative
$\varphi' = f_\rho$ which is bounded on compact subintervals, and in
addition $\varphi(r) = \int_0^r \varphi'(t) dt$ for $0 < r < R$. 
Finally, set $m_\rho(r) = \H^d(L_\rho^\sharp(r))\cap B(0,r))$,
and recall that 
\begin{equation} \label{7.48}
{r \over d}  \varphi'(r) = \varphi(r) + m_\rho(r)
\end{equation}
almost everywhere on $[0,R)$, by \eqref{7.47} and the line that follows it.

Notice that thanks to the fact that we removed $B(0,\rho)$
from $L'$, $m_\rho(r) = 0$ for $r < \rho$
(see \eqref{7.43} and \eqref{7.34}). Similarly,
$\varphi(r) = H_{1,\rho}(r) = 0$ for $r < \rho$, by the analogue for 
$L'_\rho = L' \sm B(0,\rho)$ of \eqref{7.31} and \eqref{7.29}.
On any compact subinterval of $(0,R)$ our equation \eqref{7.48}
is very nice, and it is easy to see that
\begin{equation} \label{7.49}
H_{1,\rho}(r) = \varphi(r) = d r^{d} \int_0^r {m_\rho(t) dt \over t^{d+1}}
\ \text{ for } 0 < r < R;
\end{equation}
for instance we may call $\psi$ the right-hand side of \eqref{7.49}, 
observe that $\eta = \varphi - \psi$ vanishes near $0$, is the integral of its
derivative, and satisfies $\eta'(r) = {d \over r}\eta(r)$ almost-everywhere,
and then  apply Gr\"onwall's inequality to show that it vanishes on $(0,R)$.

We may now let $\rho$ tend to $0$ in \eqref{7.49}, notice that
both $H_{1,\rho}$ and $m_\rho$ are nondecreasing functions of $\rho$
(see \eqref{7.43} and \eqref{7.7} for $m_\rho$, and \eqref{7.29}, \eqref{7.31},
\eqref{7.34}, and \eqref{7.39} for $H_{1,\rho}$).

Thus by Beppo-Levi we get that $H_1(r) = H(r)$ for $0 < r < R$. Since
$N(x) = \1_S(x)$ by \eqref{7.32} and \eqref{7.33}, \eqref{7.31}
yields $H_1(r) = \H^d(S \cap B(0,r))$, as needed for 
Remark \ref{t7.3}.
\qed

\ms
Return to the general case. 
Notice that $H_{1,\rho}$, $L_\rho^\sharp(r)$, and $f_\rho$, 
are nondecreasing functions of $\rho$; then by \eqref{7.47} and Beppo-Levi,
\begin{equation} \label{7.50}
H_1(r) + \H^d(L^\sharp(r)\cap B(0,r)) \leq {r \over d} \, f(r)
\end{equation}
for almost every $r \in (0,R_1)$, with 
$f(r) = \lim_{\rho \to 0} f_\rho(r) = \int_{S \cap \d B(0,t)} N(x) d\H^{d-1}(x)$.
Let us assume that 
\begin{equation} \label{7.51}
H_1(r) < +\infty \ \text{ for some } r>0.
\end{equation}
Since the formulas \eqref{7.29} and \eqref{7.31} are additive in terms
of $L'$, and we checked earlier that the function $\H_{1,\rho}$ associated
to $L' \sm B(0,\rho)$ is bounded on $[0,R_1]$
(see \eqref{7.39} and recall that $\nu_\rho$ is a finite measure),
we get that $H_1(r) \leq H_1(R_1) < +\infty$ for $0 < r \leq R_1$
(because $H_1$ is clearly nondecreasing). We are ready for 
the following.

\begin{rem} \label{t7.4}
In the statement of Theorem \ref{t7.1}, we may replace the function $H$
of \eqref{7.7} with the function $H_1$ of \eqref{7.31}.
\end{rem}

Of course the statement is only useful when the functional $F$ that we build
with $H_1$ is finite somewhere, which forces \eqref{7.51}. When this happens,
$H_1$ is bounded by $H_1(R_1)$, $H_1$ is the integral of its derivative
$f$ (start from \eqref{7.39} and apply Beppo-Levi), and we have the differential inequality \eqref{7.50}. We can then reproduce
the proof of Theorem \ref{t7.1}, starting at \eqref{7.17}, and conclude as before.
\qed

\section{$E$ is contained in a cone when $F$ is constant.}
\label{S8}

The main result of this section is the following theorem, which gives some 
information on the case of equality in Theorem \ref{t7.1}.

\begin{thm} \label{t8.1}
Let $U$ and the $L_j$ satisfy the assumptions \eqref{7.1}-\eqref{7.3}
of Theorem \ref{t7.1}, and let $E$ be a coral sliding minimal set in $U$, with
boundary conditions given by the $L_j$. Suppose in addition that the functional
$F$ defined by \eqref{7.6}, \eqref{7.7}, and \eqref{7.9} is finite and constant
on $(R_0,R_1)$ (the same interval as in \eqref{7.3}).
Set $A = B(R_1) \sm \overline B(0,R_0)$. Then
\begin{equation} \label{8.1}
\H^d(E \cap S \cap A) = 0,
\end{equation}
where $S$ is the shade set defined by \eqref{7.12},
and, if 
\begin{equation} \label{8.2}
X = \big\{ \lambda x \, ; \, \lambda \geq 0 \text{ and } x\in A\cap E \big\}
\end{equation}
denotes the cone over $A\cap E$, then
\begin{equation} \label{8.3}
A \cap X\sm S \i E.
\end{equation}
If in addition $R_0 < \dist(0,L')$, then $X$ is also a coral minimal set
of dimension $d$ in $\R^n$ (with no boundary condition), and 
\begin{equation} \label{8.4}
\H^d(S \cap B(0,R_1)\sm X) = 0.
\end{equation}
\end{thm}

\ms
By sliding minimal, we mean sliding almost minimal with the gauge function 
$h = 0$, as in Definition \ref{t2.1}; then the three ways to measure
minimality (that is, $A$, $A'$, and $A_+$) are equivalent. Recall also
that ``coral'' is defined near \eqref{1.11}. 
When $R_0 = 0$, we will
prove the result with $A = B(0,R_1) \sm \{ 0 \}$, but it immediately
implies the result with $A = B(0,R_1)$.

In the special case of the introduction, we recover Theorem \ref{t1.3}.
Theorem \ref{t8.1} is also a generalization of 
Theorem 6.2 in \cite{Holder} and a partial generalization of 
Theorem 29.1 in \cite{Sliding}.

Of course we can feel very good when we can apply Theorem \ref{t8.1},
because this means that the choice of functional $F$ was locally optimal.
We are lucky that there is at least one example (the case when $L$ is an affine
space of dimension $d-1$) where this happens.

\ms
Most of the proof of Theorem \ref{t8.1}, which we shall start now,
consists in checking the proof of Theorem \ref{t7.1} for places where
we could have had strict inequalities, and our main target is the string
of inequalities that lead to the differential inequality \eqref{6.24}.

Let $E$ be as in the statement. This means that on $(R_0,R_1)$,
$F$ is finite and constant. In the present situation, \eqref{7.14}
just says that $g(r) = r^{-d}$, and then $gG = F$ (see \eqref{7.9}).
By \eqref{7.19},
\begin{equation} \label{8.5}
\int_{[r,t)} g(s) d\nu(s) = - \int_r^t g'(s) G(s) ds = d \int_r^t  s^{-d-1} G(s) ds
\end{equation}
for $R_0 < r < t < R_1$. This forces $\nu$, and then $\overline \nu_0$
to be absolutely continuous on $(R_0,R_1)$ (recall from below \eqref{7.17}
that $d\nu = d\overline\nu_0 + H'(t)dt$),
and since the density of $\overline \nu_0$ is $f_0$ (see below \eqref{6.22}),
we get that
\begin{equation} \label{8.6}
s^{-d} (f_0+H')(s) = g(s) (f_0+H')(s)
= d s^{-d-1} G(s) \ \text{ for  almost every } s\in (R_0,R_1).
\end{equation}
That is, \eqref{7.13} is an identity almost everywhere on $(R_0,R_1)$
(recall that $h \equiv 0$),
which means (because \eqref{7.13} was derived from \eqref{6.24})
that \eqref{6.24} is an identity almost everywhere on $(R_0,R_1)$.
Let us record this:
\begin{equation} \label{8.7}
\H^d(E \cap B(0,r)) = \H^d(L^\sharp(r) \cap B(0,r)) + {r \over d} f_0(r)
\ \text{ for almost every } r \in (R_0,R_1);
\end{equation}
we noted above \eqref{7.13} that $L^\sharp$ in \eqref{6.24} is the same
as $L^\sharp(r)$ in \eqref{7.6}.

We first prove \eqref{8.1}. We proceed by contradiction, and suppose
that $\H^d(E \cap S \cap A) > 0$. First define a function $\overline g$
by
\begin{equation} \label{8.8}
\overline g(x) = \sup\big\{ t\in [0,1] ; tx \in L' \big\}
\ \text{ for } x \in E \cap S.
\end{equation}
Notice that $\overline g(x) > 0$ by the definition \eqref{7.12} of $S$.
It is easy to check that $g$ is a Borel function, and we can use our contradiction 
assumption that $\H^d(E \cap S \cap A) > 0$ to find $x_0 \in E \cap S \cap A$ 
such that 
\begin{equation} \label{8.9}
\liminf_{\rho \to 0} \rho^{-d} \H^d(E \cap S \cap B(x_0,\rho)) = \omega_d
\end{equation}
(this density requirement holds $\H^d$-almost everywhere on 
$E \cap S \cap A$, because this is a subset of the rectifiable set 
$L^\ast(R_1)$, which also has a finite Hausdorff measure
(see the proof of \eqref{6.3}); in fact bounds on the lower and upper densities 
would also be enough),
\begin{equation} \label{8.10}
\liminf_{\rho \to 0} \rho^{-d} \int_{E \cap S \cap B(x_0,\rho)}
|\overline g(x)^d-\overline g(x_0)^d| d\H^d(x)  = 0
\end{equation}
(i.e., $x_0$ is a Lebesgue point for $\overline g^d$ on $E \cap S$),
and finally $x_0 \notin L'$, which we may add because $L'$ is at most
$(d-1)$-dimensional. We shall restrict our attention to radii $\rho$ so small
that
\begin{equation} \label{8.11}
B(x_0,4\rho) \i A \sm L'.
\end{equation}
For $r \in (R_0,R_1)$, set 
\begin{equation} \label{8.12}
L^\ast(r) = \overline B(0,r) \cap \big\{  \lambda z \, ; \,
z\in \overline B(0,r) \cap L' \text{ and } \lambda \geq 0\big\};
\end{equation}
this is the same set that was called $L^\ast$ in \eqref{4.2}
(also see \eqref{7.5} for $L'$) but since we shall let $r$ vary, 
we include it in the notation. Set $B = B(x_0,\rho)$ and notice that
\begin{equation} \label{8.13}
S \cap B \i L^\ast(|x_0|+\rho)
\end{equation}
because if $x\in S \cap B$, \eqref{7.12} says that we can find 
$\lambda \in (0,1]$ such that $\lambda x \in L'$; by \eqref{8.11},
$|\lambda x - x_0| \geq 4 \rho$, which implies that $|\lambda x - x| \geq 3 \rho$
and forces $\lambda x \in B(0,|x|-3\rho) \subset B(0,|x_0|-2\rho)$. Hence 
$x\in L^\ast(|x_0|+\rho)$ by \eqref{8.12}. Let us also check that
\begin{equation} \label{8.14}
B \cap L^\ast(r) = B \cap \overline B(0,r) \cap L^\ast(|x_0|+\rho)
\ \text{ for } |x_0|-\rho \leq r \leq |x_0|+\rho.
\end{equation}
The direct inclusion is clear from the definitions; conversely, if 
$x\in B \cap \overline B(0,r) \cap L^\ast(|x_0|+\rho)$, then
$x = \lambda z$ for some $z \in L' \cap \overline B(0,|x_0|+\rho)$ and
$\lambda \geq 0$. By definition of $B$ and then \eqref{8.11}, 
$|z-x| \geq |z-x_0| - \rho \geq 3\rho$, and since $z$ is 
collinear with $x$, this forces $|z| \leq |x|-3\rho$ or 
$|z| \geq |x|+3\rho$. The second one is impossible because 
$z \in \overline B(0,|x_0|+\rho)$, and so
$|z| \leq |x|-3\rho < |x_0|-\rho \leq r$. Thus $z\in L' \cap \overline B(0,r)$
and $x\in L^\ast(r)$, which proves \eqref{8.14}.

Now suppose that
\begin{equation} \label{8.15}
|x_0|-\rho \leq b < r \leq |x_0|+\rho,
\end{equation}
and recall from the intermediate estimate in \eqref{6.21} that
\begin{eqnarray} \label{8.16}
\H^d(E \cap B(0,b)) &\leq& \H^d(L^\sharp(r) \cap B(0,b))
+ {b \over d} f_0(b) -  {b \over d}(f_1(b)-f_2(b)) - \Delta(b,r)
\nn\\
&\leq& \H^d(L^\sharp(r) \cap B(0,b)) + {b \over d} f_0(b) - \Delta(b,r),
\end{eqnarray}
where we give a more explicit notation for $\Delta$ because we will let $b$ and
$r$ vary, and where the last line follows because $f_1(b) \geq f_2(b)$, 
by \eqref{6.18}). Also recall that 
\begin{equation} \label{8.17}
\Delta(b,r) = {b \over d} \int_X g(x)^d \H^{d-1}(x)
\end{equation}
(by \eqref{6.11}), with $X = E \cap L^\ast(r) \cap \d B(0,b)$ (see below \eqref{6.5})
and 
\begin{equation} \label{8.18}
g(x) = \sup\big\{ t\in [0,1] ; tx \in L' \cap \overline B(0,r)\big\}
\end{equation}
(see \eqref{6.8}, and recall the definitions \eqref{4.1} of $L$ and \eqref{7.5} of $L'$).
We claim that 
\begin{equation} \label{8.19}
g(x) = \overline g(x) \ \text{ for } x\in X \cap B
\end{equation}
when $r$ and $b$ as in \eqref{8.15}. Indeed, if $t\in [0,1]$ is such that
$tx \in L'$, then $|tx-x| \geq |tx - x_0| - \rho \geq 3\rho$ by \eqref{8.11},
hence $tx \in \overline B(0,|x|-3\rho) \subset B(0,r)$; hence the supremums
in \eqref{8.8} and \eqref{8.17} are the same, and \eqref{8.19} follows.
Thus
\begin{equation} \label{8.20}
\Delta(b,r) \geq {b \over d} \int_{X\cap B} \overline g(x)^d \H^{d-1}(x)
\geq {b \over d} \int_{B \cap E \cap S\cap \d B(0,b)} 
\overline g(x)^d \H^{d-1}(x)
\end{equation}
because $B \cap E \cap S\cap \d B(0,b) \i B \cap E \cap L^\ast(|x_0|+\rho) \cap \d B(0,b)
= B \cap X$ by \eqref{8.13} and \eqref{8.14}.
Thus \eqref{8.16} implies that
\begin{equation} \label{8.21}
\H^d(E \cap B(0,b)) \leq \H^d(L^\sharp(r) \cap B(0,b)) + {b \over d} f_0(b) 
- {b \over d} \int_{B \cap E \cap S\cap \d B(0,b)} \overline g(x)^d \H^{d-1}(x).
\end{equation}
Fix $b$ and let $r > b$ tend to $b$. Notice that by \eqref{7.6},
$L^\sharp(r)$ is a nondecreasing set function that tends to $L^\sharp(b)$;
since all these sets have finite $\H^d$-measure by \eqref{3.5}, we can take a
limit in \eqref{8.21} and get that
\begin{equation} \label{8.22}
\H^d(E \cap B(0,b)) \leq \H^d(L^\sharp(b) \cap B(0,b)) + {b \over d} f_0(b) 
- {b \over d} \int_{B \cap E \cap S\cap \d B(0,b)} \overline g(x)^d \H^{d-1}(x).
\end{equation}

We compare this to \eqref{8.7} and get that
\begin{equation} \label{8.23}
\int_{B \cap E \cap S\cap \d B(0,b)} \overline g(x)^d \H^{d-1}(x)
= 0 \ \text{ for almost every } b\in (|x_0|-\rho,|x_0|+\rho).
\end{equation}

We are going to use the coarea formula again. By \eqref{8.13},
$B \cap E \cap S$ is contained in the truncated rectifiable cone 
$L^\ast(|x_0|+\rho)$, which has a finite $\H^d$-measure; as in \eqref{6.4}
above, the coarea formula yields
\begin{equation} \label{8.24}
J d\H^d_{|B \cap E \cap S} 
= \int_{|x_0|-\rho}^{|x_0|+\rho} d\H^{d-1}_{| \pi^{-1}(t) \cap B \cap E \cap S} dt
= \int_{|x_0|-\rho}^{|x_0|+\rho} d\H^{d-1}_{|\d B(0,t) \cap B \cap E \cap S} dt,
\end{equation}
where in addition $J=1$ because $L^\ast(|x_0|+\rho)$ is a truncated cone.
We apply this to the function $\overline g$ and get that
\begin{equation} \label{8.25}
\int_{B \cap E \cap S} \overline g(x) d\H^d(x)
= \int_{|x_0|-\rho}^{|x_0|+\rho} \int_{\d B(0,t) \cap B \cap E \cap S} 
\overline g(x) d\H^{d-1}(x) dt = 0
\end{equation}
by \eqref{8.23}. On the other hand, recall that $B = B(x_0,\rho)$,
where we still may choose $\rho$ as small as we want. Then by
\eqref{8.9} and \eqref{8.10},
\begin{eqnarray} \label{8.26}
\int_{B \cap E \cap S} \overline g(x) d\H^d(x)
&\geq& g(x_0) \, \H^d(B \cap E \cap S) 
- \int_{B \cap E \cap S} |\overline g(x)- \overline g(x_0)| d\H^d(x)
\nn\\
&\geq& g(x_0) \, {\omega_d \rho^d \over 2} 
- \int_{B \cap E \cap S} |\overline g(x)- \overline g(x_0)| d\H^d(x)
> 0
\end{eqnarray}
if $\rho$ is small enough. This contradiction with \eqref{8.25} completes our proof
of \eqref{8.1}.

\ms\ms
Next we worry about the angle $\theta(x)$.
Recall that $\theta(x) \in [0,\pi/2]$ was defined, for $H^d$-almost every point
$x\in E$, to be the smallest angle between the radial direction
$[0,x)$ and the (approximate) tangent plane $P(x)$ to $E$ at $x$;
see near \eqref{4.47}. We want to check that
\begin{equation} \label{8.27}
\theta(x) = 0 \ \text{ for $\H^d$-almost every point } x\in E \cap A. 
\end{equation}
Suppose that \eqref{8.27} fails. Then we can find $\eta < 1$ and
a set $E_0 \i E \cap A$ such that $\H^d(E_0) > 0$ and 
$\cos\theta(x) \leq \eta$ for $\H^d$-almost every $y\in E_0$.

Then choose $x_0 \in E_0$ such that (as in \eqref{8.9})
\begin{equation} \label{8.28}
\liminf_{\rho \to 0} \rho^{-d} \H^d(E_0 \cap B(x_0,\rho)) > 0.
\end{equation}
As before, we can take $x_0 \in A \sm L'$, because 
$\H^d(L') = 0$. Then pick $\rho > 0$ so small that \eqref{8.11}
holds and set $B = B(x_0,\rho)$. 
Notice that \eqref{8.14} holds for the same reason as before.
In fact, we just need the (trivial) first inclusion. Observe that since 
$L^\ast(|x_0|+\rho)$ is a truncated rectifiable cone with finite $\H^d$-measure,
the tangent measure to any of its points, when it exists, passes through the
origin. Since this is almost never the case on the set $E_0$ (and by
the uniqueness of the tangent plane almost everywhere on 
$E_0 \cap L^\ast(|x_0|+\rho))$, we get that
\begin{equation} \label{8.29}
\H^d(E_0 \cap L^\ast(r)) \leq \H^d(E_0 \cap L^\ast(|x_0|+\rho)) = 0
\end{equation}
for $r \leq |x_0|+\rho$.

Denote by $\mu_B$ the restriction of $\H^d$ to $E_0 \cap B$, 
and let $\nu_B$ be direct image of $\mu_B$ by $\pi$. Observe that since
$\mu_B \leq \overline\mu_0$ (the restriction of $\H^d$ to 
$E \cap B(0,R) \supset E \cap A$;
see \eqref{6.22} and below), we get that $\nu_B \leq \overline \nu_0$ is 
absolutely continuous as well (see below \eqref{8.5}).
Denote by $f_B$ the density of $\nu_B$.

Let $r \in (|x_0|-\rho,|x_0|+\rho)$ be given; for $b < r$, we defined in \eqref{6.18}
the two measures $\mu_1 = \H^d_{E \cap \overline B(0,r) \sm L^\ast(r)}$
and $\mu_2 = \cos\theta(x) \mu_1$, and then defined the pushed measures
$\nu_1$ and $\nu_2$ as usual. Notice that
$\1_{B(0,r)} \mu_B \leq \1_{E_0 \cap B} \,\mu_1$ because of \eqref{8.29}.
 
Next consider $b$ such that $|x_0|-\rho <b < r$. For $\tau > 0$ (small),
\begin{eqnarray} \label{8.30}
\nu_1([b-\tau,b)) - \nu_2([b-\tau,b)) 
&=& \int_{B(0,b) \sm B(0,b-\tau)} (1-\cos\theta(x)) d\mu_1(x)
\nn\\
&\geq& \int_{B(0,b) \sm B(0,b-\tau)} \1_{E_0 \cap B}(x) 
(1-\cos\theta(x)) d\mu_1(x)
\nn\\
&\geq& (1-\eta) \int_{B(0,b) \sm B(0,b-\tau)} d\mu_B(x) 
= (1-\eta) \nu_B([b-\tau,b)) 
\end{eqnarray}
because $\cos\theta(x) \leq \eta$ almost everywhere on $E_0$,
then by definition of $\mu_B$ and $\nu_B$. 

If in addition $b$ is a Lebesgue point for the absolutely continuous parts
of $\nu_1$, $\nu_2$, and of $\nu_B$, we can divide \eqref{8.30}
by $\tau$, take a limit, and get that
\begin{equation} \label{8.31}
f_1(b) - f_2(b) \geq (1-\eta) f_B(b).
\end{equation}
Then \eqref{6.21} says that
\begin{eqnarray} \label{8.32}
\H^d(E \cap B(0,b)) &\leq& \H^d(L^\sharp(r) \cap B(0,b))
+ {b \over d} f_2(b) +  {b \over d}(f_0(b)-f_1(b))
\nn\\
&\leq& \H^d(L^\sharp(r) \cap B(0,b))
+ {b \over d} f_0(b) -  {b \over d} (1-\eta) f_B(b).
\end{eqnarray}
With $b$ fixed, we may let $r > b$ tend to $b$, observe that
$\H^d(L^\sharp(r) \cap B(0,b))$ tends to $\H^d(L^\sharp(b) \cap B(0,b))$
(see above \eqref{8.22}), and get as in \eqref{8.22} that
\begin{equation} \label{8.33}
\H^d(E \cap B(0,b)) \leq \H^d(L^\sharp(b) \cap B(0,b))
+ {b \over d} f_0(b) -  {b \over d} (1-\eta) f_B(b).
\end{equation}
Then we compare with \eqref{8.7} and get that $f_B(b) = 0$
for almost every $b\in (|x_0|-\rho, |x_0|+\rho)$.
Since $\nu_B$ is absolutely continuous and $f_B$ is its density,
we get that $\nu_B = 0$, and hence $\H^d(E_0 \cap B) = \mu_B(E_0 \cap B) = 0$.
If $\rho$ was chosen small enough, this contradicts \eqref{8.28}, and proves
\eqref{8.27}.

\ms\ms
Now we shall worry about cones. For $x\in A$, set
\begin{equation} \label{8.34}
\ell(x) = \big\{ \lambda x \, ; \, \lambda > 0 \big\}.
\end{equation}
The next stage is to prove the following: for $\H^d$-almost
every $x\in E \cap A \sm S$ (in fact, every $x\in E \cap A \sm S$
such that the tangent plane to $E$ at $x$ goes through the origin), 
\begin{equation} \label{8.35}
\text{the connected component of $x$ in 
$A \cap \ell(x) \sm L'$ is contained in $E$}.
\end{equation}
Notice that $x\in \ell(x) \sm L'$ because $L' \i S$ by \eqref{7.12},
so \eqref{8.35} makes sense.
The proof of this is long and painful (especially since this should
morally be easy once we know \eqref{8.27}), but fortunately it
was already done in \cite{Holder}, Proposition 6.11. 
Of course the statement in \cite{Holder} does not mention sliding
boundary conditions, but the argument applies for the following
reason. The proof goes by contradiction: if \eqref{8.35} fails,
we first find a point $y\in \ell'(x)$, the component of $x$ in 
$A \cap \ell(x) \sm L'$, which does not lie in $E$. Then we 
deform $E$ into a set $E'$ with smaller measure and get a contradiction.
The deformation takes place in a neighborhood of the segment
$[x,y]$ which is as small as we want, and in particular can be taken
not to meet $L'$ (because $L'$ is closed and does not meet 
$[x,y]$). Then the boundary constraints \eqref{2.3} are
automatically satisfied, the competitor constructed in \cite{Holder}
is also valid here, and we can use the estimates from \cite{Holder}
to get the conclusion. So \eqref{8.35} holds for $\H^d$-almost every 
$x\in E \cap A \sm S$.

We are ready to prove \eqref{8.3}. Let $z\in A \cap X$ be given,
and let $x\in E \cap A$ and $\lambda \geq 0$ be such that
$z = \lambda x$.
 
Since $E$ is coral,
$\H^d(E\cap B(x,r)) > 0$ for every $r > 0$, so we can find a sequence
$\{ x_k \}$ in $E$, that converges to $x$, such that for each $k \geq 0$,
$x_k \in A \sm S$ (possible by \eqref{8.1}) and \eqref{8.35} holds.

For each $k$, the segment $(0,x_k]$ lies in $\R^n \sm L'$
(see the definition \eqref{7.12}), so $A \cap (0,x_k] \i E$
by \eqref{8.35}, and now $A \cap (0,x] \i E$ because $E$ is closed.
Thus $z\in E$ if $\lambda \leq 1$.

Suppose that  $\lambda > 1$ instead, and assume in addition that 
$z \in A \cap X \sm S$.
Then the segment $[x,z] = [x,\lambda x]$ does not meet $L'$ 
(because $z \notin S$ and by the definition \eqref{7.12}), 
and is contained in $A$ (because $x\in A$). For $k$ large, 
$[x_k, \lambda x_k]$ is also contained in $A$, and does not meet $L'$ either
(because $L'$ is closed). By \eqref{8.35}, $[x_k,\lambda x_k] \i E$. Hence
$[x,\lambda x] \i E$, and in particular $z\in E$. This completes our proof of \eqref{8.3}.

\ms
Let us now assume that $R_0 < \dist(0,L')$ and get additional
information on $X$. First we want to show that
\begin{equation} \label{8.36}
\text{$X$ is a minimal set in $\R^n$,}
\end{equation}
with no boundary condition.

Set $\delta = \dist(0,L')$ and observe 
that for $0 < r < \delta$, $H(r) = 0$ and so
$F(r) = r^{-d} \H^d(E \cap B(0,r))$. See \eqref{7.6}, \eqref{7.7}, 
and \eqref{7.9}. Similarly, $S \cap B(0,\delta) = \emptyset$
(see \eqref{7.12}).

By \eqref{8.2} and \eqref{8.3}, $E$ and $X$ coincide
on $B(0,\delta) \sm \overline B(0,R_0)$. In particular, $X$ is closed. 
If $D$ denotes the constant value of $F$ on $(R_0,R_1)$; then for
$R_0 < r < \delta$, 
\begin{eqnarray} \label{8.37}
\H^d(X\cap B(0,r)) &=& {r^d \over \delta^d - R_0^d} \,
\H^d(X \cap B(0,\delta) \sm B(0,R_0))
\nn\\
&=& {r^d \over \delta^d - R_0^d} \,\H^d(E \cap B(0,\delta) \sm B(0,R_0)) 
\nn\\
&=& r^d D = \H^d(E \cap B(0,r))
\end{eqnarray}
because $X$ is a cone and by the discussion above.
Notice also that $E$ is a minimal set in $B(0,\delta)$,
with no boundary condition (because $L' \cap B(0,\delta) = \emptyset$).
Now if $X$ was not a minimal in $\R^n$,
it would be possible to find a strictly better competitor $\wt X$ of $X$,
with a deformation inside  $B(0,\delta)$ (we may reduce to this because
$X$ is a cone). We could then deform $E$ inside $B(0,\delta)$, first to $X$ 
and then to $\wt X$, and contradict the minimality of $X$. See the argument
on page 1225-126 of \cite{Holder} (near (6.98)) for detail.
This proves \eqref{8.36}.

Since $E$ is coral and $X = E$ on $B(0,\delta) \sm \overline B(0,R_0)$,
we easily get that $X$ is coral too.
Then we prove \eqref{8.4}. Observe that
\begin{equation} \label{8.38}
R_1^d D =  H^d(X \cap B(0,R_1))
= \H^d(E \cap (B(0,R_1)) + H^d(X \cap S \cap B(0,R_1))
\end{equation}
by \eqref{8.37} and because $X$ is a cone, and because \eqref{8.2} and \eqref{8.3} 
say that $A \cap E = A \cap X \sm S$ modulo a $\H^d$-negligible set.
Also,
\begin{equation} \label{8.39}
R_1^d D = R_1^d F(R_1) = G(R_1) = \H^d(E \cap (B(0,R_1)) + H(R_1)
\end{equation}
by \eqref{7.9}. Thus $H^d(X \cap S \cap B(0,R_1)) = H(R_1)$.
We shall now prove that
\begin{equation} \label{8.40}
H(r) \geq \H^d(S \cap B(0,r))  \ \text{ for } 0 < r \leq R_1,
\end{equation}
and \eqref{8.4} will follow at once. Let us proceed as in Remark \ref{t7.3},
and compare $H_2(r) = \H^d(S \cap B(0,r))$ with $H(r)$ by means of
the differential equalities that they satisfy.
We claim that for $0 < r < R_1$,
\begin{equation} \label{8.41}
H_2(r) = \int_0^r f_2(t) dt,
\text{ with } f_2(t) = \H^{d-1}(S \cap \d B(0,t)).
\end{equation}
Indeed, $S \cap B(0,R_1) \i  L^\ast(R_1)$, where
$L^\ast(R_1)$ is the same as in \eqref{7.35}. We observed above \eqref{7.36} 
that $L^\ast(R_1)$ is a rectifiable truncated cone with 
$\H^d(L^\ast(R_1)) < +\infty$; then we can use the coarea formula on
$S \cap B(0,R_1) \i  L^\ast(R_1)$, and we get \eqref{8.41}, as in 
\eqref{7.37}-\eqref{7.38}, where we use $\1_S$ instead of $N_\rho$.
If we prove that for almost all $r < R_1$, we have the differential inequality
\begin{equation} \label{8.42}
H_2'(r) = f_2(r) \leq {d \over r} H_2(r) + {d \over r} m(r),
\end{equation}
where $m(r) = \H^d(L^\sharp(r))$ is as in \eqref{7.7}, then the comparison
with $H$, which satisfied the corresponding identity (by \eqref{7.8})
and has the same vanishing values on $(0,\delta)]$, we will deduce 
\eqref{8.40} from Gr\"onwall's inequality, as we did near \eqref{7.49}.
But by the coarea formula again,
${r \over d} f_2(r) = {r \over d}  \H^{d-1}(S \cap \d B(0,r))$ is 
the $\H^d$-measure of the cone 
$\Gamma = \big\{ \lambda z \, ; \, z\in S \cap \d B(0,r) \text{ and } 
\lambda \in [0,1) \big\}$. Let us check that
\begin{equation} \label{8.43}
\Gamma \i (S\cap B(0,r)) \cup L^\sharp(r).
\end{equation}
Let $x\in \Gamma$ be given, and write $x = \lambda z$ as in the definition of $\Gamma$. 
Since $z\in S$, the segment $(0,z]$ meets $L'$ at some
point $w$ (see \eqref{7.12}). If $|w| \leq |x|$, then $(0,x]$ contains
$w\in L'$, and $x\in S$. Otherwise, $x\in [0,w]$ and $x\in L^\sharp(r)$
(see \eqref{7.6}). This proves \eqref{8.43}, the differential inequality \eqref{8.42},
and finally \eqref{8.40} and \eqref{8.4}.
This also completes the proof of Theorem 8.1.
\qed

\section{Nearly constant $F$ and approximation by a cone}
\label{S9}

The main result of this section is a simple consequence of the theorems on limits
of sliding almost minimal sets that are at the center of \cite{Sliding}.
Roughly speaking, we start from a sliding almost minimal set $E$ with sufficiently small
gauge function $h$, assume that its function $F$ is almost constant on
an interval, and get that in a slightly smaller annulus, $E$ is close to a sliding minimal 
set $E_0$ for which $F$ is constant.

Then we should be able to apply Theorem~\ref{t8.1} to the set $E_0$,
prove that it is close to a truncated cone, and get the same thing for $E$,
but we shall only do this here in very special cases; see Sections \ref{S11}
and \ref{S12}.

We start with a statement where the interval is of the form $(0,r_1)$.

\ms
\begin{thm} \label{t9.1}
Let $U$ and the $L_j$ be as in Section \ref{S7}.
In particular, assume \eqref{2.1}, \eqref{2.2}, and \eqref{3.4}.
Let $r_1 > 0$ be such that $B(0,r_1) \i U$ and almost every
$r \in (0,r_1)$ admits a local retraction (as in Definition \ref{t3.1}).
For each small $\tau > 0$ we can find $\varepsilon > 0$,
which depends only on $\tau$, $n$, $d$, $U$, the $L_j$, and $r_1$,
with the following property. 
Let $E \in SA^\ast M(U,L_j,h)$ be a coral sliding almost minimal set
in $U$,  with sliding condition defined by $L$ and some nondecreasing gauge 
function $h$ (see Definition \ref{t2.1}). Suppose that
\begin{equation} \label{9.1}
\text{$B(0,r_1) \i U$ and $h(r_1) < \varepsilon$,}
\end{equation}
and 
\begin{equation} \label{9.2}
F(r_1) \leq \varepsilon + \inf_{0 < r < 10^{-3} r_1} F(r) < +\infty.
\end{equation}
Then there is a coral minimal set $E_0$ in $B(0,r_1)$, with 
sliding condition defined by $L$, such that
\begin{eqnarray} \label{9.3}
\text{ the analogue of $F$ for the set $E_0$ is constant on }(0,r_1),
\end{eqnarray}
\begin{equation} \label{9.4}
E_0 \text{ satisfies the conclusion of Theorem \ref{t8.1},
with the radii $R_0 = 0$ and $R_1 =r_1$},
\end{equation}
\begin{equation} \label{9.5}
\dist(y, E_0) \leq \tau r_1 \ \text{ for } y \in E\cap B(0,(1-\tau)r_1),
\end{equation} 
\begin{equation} \label{9.6}
\dist(y,E) \leq \tau r_1 \ \text{ for } y \in E_0 \cap B(0,(1-\tau)r_1),
\end{equation}
and
\begin{eqnarray} \label{9.7}
\av{\H^d(E \cap B(y,t)) - \H^d(E_0 \cap B(y,t))} \leq \tau r_1^d &&
\nonumber \\
&&\hskip-7cm
\ \text{ for all $y\in \R^n$ and $t>0$ such that }
B(y,t) \i B(0,(1-\tau)r_1).
\end{eqnarray}
\end{thm}

\ms 
Clearly this is not a perfect statement, because we would prefer
$\varepsilon$ not to depend on $r_1$ or the $L_j$. The difficulty 
is that we should then understand how the part $H$ of the functional
depends on $L'$, and in particular its values near the origin.

Instead we decided to make a simple statement, show how the proof
goes, and maybe revise it later if a specific need arises.

See Corollary \ref{t9.3} for a statement that concerns 
a single boundary condition that comes from a set $L$ which is very close
to an affine subspace.

We now prove the theorem by contradiction and compactness, following
the arguments for Proposition 7.24 in \cite{Holder} and Proposition 30.19 in \cite{Sliding}.
Let $U$, the $L_j$, $r_1$, and $\tau$ be given, and suppose we cannot
find $\varepsilon$ as in the statement. In particular, $\varepsilon = 2^{-k}$
does not work, which means that for $k \geq 1$ we can find $E_k$ and
$h_k$ as in the statement (and in particular that satisfy 
\eqref{9.1} and \eqref{9.2} with $\varepsilon = 2^{-k}$), 
and for which no minimal set $E_0$ satisfies the conclusion.

First we want to replace $\{ E_k \}$ with a converging subsequence.
Let us say what we mean by that. For each ball $B(x,r)$, let $d_{x,r}$ 
be the normalized local variant of the Hausdorff distance between closed sets 
that was defined by \eqref{1.29}. 
Standard compactness results on the Hausdorff distance imply 
that we can replace $\{ E_k \}$ with a subsequence
for which there is a closed set $E_0$ such that
\begin{equation} \label{9.8}
\lim_{k \to +\infty} d_{x,r}(E_k,E_0) = 0
\end{equation}
for every ball $B(x,r)$ such that $\overline B(x,r) \i U$.
In short, we will just say that the $E_k$ converge to $E_0$
locally in $U$. We want to show that 
\begin{equation} \label{9.9}
\text{$E_0$ is a sliding $A$-almost minimal set in $U$,}
\end{equation}
where it is understood that the boundary conditions are still
given by the $L_j$, for the gauge function $h_0$ defined by
\begin{equation} \label{9.10}
h_0(r) = 0 \text{ for } 0 < r < r_1 \text{ and } 
h_0(r) = +\infty \text{ for } r \geq r_1.
\end{equation}
For this we shall apply Theorem 10.8 in \cite{Sliding}. 

We may replace $\{ E_k \}$ with a subsequence so that all the 
$E_k$ are almost minimal of the same type (i.e, $A$, $A'$, or $A_+$;
see Definition \ref{t2.1}).
Let us first suppose that the $E_k$ are almost minimal of type 
$A$, i.e., satisfy \eqref{2.5}. Then let us check the assumptions
with $M = 1$, $\delta = r_1$, and any small number $h>0$. 

The Lipschitz assumption ((10.1) in \cite{Sliding}) is satisfied, by
\eqref{2.1}. For $k$ large enough, $h_k(r_1) \leq 2^{-k}< h$, and then 
the main assumption that $E_k$ lies in the class $GSAQ(U,M,\delta,h)$ 
follows from \eqref{2.5} (compare with Definition 2.3 in \cite{Sliding}). 
This takes care of (10.3) in \cite{Sliding},  (10.4) holds 
because the $E_k$ are coral, and (10.5) holds (with the limit $E_0$),
precisely by \eqref{9.8}. Finally the technical assumption (10.7)
holds, because we said in \eqref{2.2} that the faces of the $L_j$
are at most $(d-1)$-dimensional. So Theorem 10.8 in \cite{Sliding} applies, 
and we get that $E_0$ lies in $GSAQ(U,1,r_1,h)$ for any $h > 0$. 
Looking again at Definition 2.3 in \cite{Sliding}, we see that 
$E_0 \in GSAQ(U,1,r_1,0)$, and this means that \eqref{9.9} holds
with the function $h_0$.

If the $E_k$ are almost minimal of type $A'$, i.e., satisfy \eqref{2.6},
the easy part of Proposition~20.9 in \cite{Sliding} says that they are
also of type $A$, with the same function $h_k$, and we can conclude as before.
Finally, if the $E_k$ are of type $A_+$, as in \eqref{2.7}, we apply 
Theorem 10.8 in \cite{Sliding} with $M = 1+a$, where $a > 0$ is any
small number, $\delta = r_1$, and $h=0$. The main assumption that
$E_k \in GSAQ(U,1+a,r_1,0)$ is satisfied as soon as $h_k(r_1) < a$
(compare \eqref{2.7} with Definition 2.3 in \cite{Sliding}). This time we get
that $E_0 \in GSAQ(U,1+a,r_1,0)$ for any $a> 0$, hence
$E_0 \in GSAQ(U,1,r_1,0)$, and $E$ is $A$-almost minimal with the
gauge $h_0$, as before. So we get \eqref{9.9}.

Next we worry about Hausdorff measure.
By Theorem 10.97 in \cite{Sliding}, which has the same hypotheses
as Theorem 10.8 there, we get that for every open set $V \i U$,
\begin{equation} \label{9.11}
\H^d(E_0 \cap V) \leq \liminf_{k \to +\infty} \H^d(E_k \cap V).
\end{equation}
In addition, Lemma 22.3 in \cite{Sliding}, which has also the same assumptions
as Theorem 10.8, can be applied with $M$ as close to $1$ and $h$ as
small as we want; thus we get that for every compact set $K \i U$,
\begin{equation} \label{9.12}
\H^d(E_0 \cap K) 
\geq \limsup_{k \to +\infty} \H^d(E_k \cap K).
\end{equation}
Because of \eqref{2.9}, for each $\rho < r_1$ there is a constant
$C(\rho)$ such that $\H^d(E_k \cap B(0,\rho)) \leq C(\rho)$
for all $k \geq 0$. Then by \eqref{9.11}, 
$\H^d(E_0 \cap B(0,\rho)) \leq C(\rho) < +\infty$ for each $\rho < r_1$, 
and $\H^d(E_0 \cap \d B(0,r)) = 0$ for almost every
$r \in (0,r_1)$, because the $\d B(0,r)$ are disjoint.
For such $r$,
\begin{eqnarray}  \label{9.13}
\limsup_{k \to +\infty} \H^d(E_k \cap B(0,r))
&\leq& \H^d(E_0 \cap \overline B(0,r)) = \H^d(E_0 \cap B(0,r))
\nonumber \\
&\leq& \liminf_{k \to +\infty} \H^d(E_k \cap B(0,r)),
\end{eqnarray}
by \eqref{9.11} and \eqref{9.12}, so
\begin{equation} \label{9.14}
\lim_{k \to +\infty} \H^d(E_k \cap B(0,r)) = \H^d(E_0 \cap B(0,r)).
\end{equation}

Denote by $F_k$ the analogue of $F$ for $E_k$; that is, set
\begin{equation} \label{9.15}
F_k(r) = r^{-d} \H^d(E_k \cap B(0,r)) + r^{-d} H(r)
\ \text{ for } 0 < r \leq r_1,
\end{equation}
where $H(r)$ is as in \eqref{7.7} (see \eqref{7.9} for the definition
of $F$). Notice that $H(r)$ stays the same (it depends on
the geometry of $L'$, not on $E_k$). Similarly, set
\begin{equation} \label{9.16}
F_\infty (r) = r^{-d} \H^d(E_0 \cap B(0,r)) + r^{-d} H(r)
\ \text{ for } 0 < r \leq r_1.
\end{equation}
We want to show that
\begin{equation} \label{9.17}
F_\infty \text{ is finite and constant on } (0,r_1).
\end{equation}
We start with the finiteness. We checked below \eqref{9.12} that for $0 < r < r_1$,
there is a constant $C(r)$ such that $\H^d(E_0 \cap B(0,r)) \leq C(r) < +\infty$,
so we just need to show that $H(r) < +\infty$. This last follows from our assumption 
\eqref{9.2} (for any $k$).

Let us apply Theorem \ref{t7.1} to $E_0$, with the function $h_0$.
First observe that $E_0$ is also $A_+$-almost minimal,
with the same gauge function $h_0$ (just compare
\eqref{2.5} and \eqref{2.7} when $h(r) = 0$).
By Remark \ref{t7.2}, we do not need to check \eqref{7.10}.
Also observe that the function $A$ defined by \eqref{7.4} with the
gauge function $h_0$ vanishes on $(0,r_1)$. Then Theorem \ref{t7.1} says that
$F_\infty$ is nondecreasing on $(0,r_1)$. 

Next $r$ and $s$ be such that $0 < r < 10^{-3} r_1 < s < r_1$ and 
\eqref{9.14} holds for $r$ and $s$. Then 
\begin{eqnarray} \label{9.18}
F_\infty (s) &=& \lim_{k \to +\infty} F_k(s)
\leq (r_1/s)^d \limsup_{k \to +\infty} F_k(r_1)
\nn\\
&\leq& (r_1/s)^d  \lim_{k \to +\infty} (2^{-k} + F_k(r)))
= (r_1/s)^d \lim_{k \to +\infty} F_k(r) = (r_1/s)^d F_\infty (r)
\end{eqnarray}
by \eqref{9.14}, the definition of $F_k$ in \eqref{9.15} (twice), 
because $E_k$ satisfies \eqref{9.2} 
with $\varepsilon = 2^{-k}$, and by \eqref{9.14} again.
We know that $F_\infty$ is nondecreasing; then $F_\infty (s)$ has a limit when
$s$ tends to $r_1$, and by \eqref{9.18} this limit is at most $F_\infty (r)$; 
thus $F_\infty$ is constant on $(r,r_1)$ and since this holds for all $r\in (0,10^{-3} r_1)$ we get \eqref{9.17}.

Now we claim that \eqref{9.4} holds.
This is not an immediate consequence of the statement, because
\eqref{9.9} does not exactly say that $E_0$ is minimal in $U$,
but the proof of Theorem \ref{t8.1} only uses deformations in
compact subsets of $B(0,r_1)$, for which the values of $h_0(r)$,
$r \geq r_1$, do not matter. So \eqref{9.4} holds.

The set $E_0$ also satisfies the constraints \eqref{9.5} and \eqref{9.6}
for $k$ large, by \eqref{9.8}. We want to check that it satisfies \eqref{9.7}
as well, and as soon as we do this, we will get the desired contradiction with
the definition of $E_k$. Since we shall use this sort of argument a few times,
let us state a lemma.

\begin{lem} \label{t9.2}
Let the open set $U \i \R^n$ and the closed sets $L_j$ satisfy the Lipschitz
assumption (as in \eqref{2.1}), and let $B_0 = B(0,r_0) \subset U$ be given.
Then let $\{ E_k \}$ be a sequence of coral sliding almost minimal sets, such that 
\begin{equation} \label{9.19}
E_k \in SA^\ast M(U,L_j,h_k) \ \text{ for } k \geq 0,
\end{equation}
for a sequence $\{ h_k \}$ of nondecreasing gauge functions $h_k$ such that
\begin{equation} \label{9.20}
\lim_{k \to +\infty} h_k(r_0) = 0.
\end{equation}
Also suppose that there is a closed set $E_0$ such that
\begin{equation} \label{9.21}
\lim_{k \to +\infty} d_{x,r}(E_k,E_0) \ \text{ for every ball $B(x,r)$ such that } 
\overline B(x,r) \i U.
\end{equation}
Finally assume that $E_0 \cap B(x_0,r_0)$ is contained in a cone $X$, 
which is a countable union of closed rectifiable cones $X_m$ such that
$\H^d(X_m \cap B(0,R)) < +\infty$ for every $R > 0$.
Then for each $\tau > 0$, the following measure estimate holds for $k$ large:
\begin{eqnarray} \label{9.22}
\av{\H^d(E \cap B(y,t)) - \H^d(E_0 \cap B(y,t))} \leq \tau r_0^d &&
\nonumber \\
&&\hskip-7cm
\ \text{ for all $y\in \R^n$ and $t>0$ such that }
B(y,t) \i B(x_0,(1-\tau)r_0).
\end{eqnarray}
\end{lem}

\ms\begin{proof}
Recall that the distance $d_{x_0,r_0}$ is defined in \eqref{1.29}.
We start as in the argument above. We first check that \eqref{9.9} holds
with the gauge function $h_0$ of \eqref{9.10}; the proof relies on
Theorem 10.8 in \cite{Sliding} and is the same as above.
We also have the two semicontinuity estimates \eqref{9.11} and \eqref{9.12},
with the same proof, and we deduce from this that 
\begin{equation} \label{9.23}
\H^d(E_0 \cap B(y,t)) = \lim_{k \to +\infty} \H^d(E_k \cap B(y,t))
\end{equation}
for every ball $B(y,t)$ such that $\overline B(y,t) \subset U$ and 
\begin{equation} \label{9.24}
\H^d(E_0 \cap \d B(y,t)) = 0.
\end{equation}
The proof is the same as for \eqref{9.14}.

Now this almost the same thing as \eqref{9.22}; the difference is that maybe 
some balls do not satisfy \eqref{9.24}, but also that in the lemma, we announced 
some uniformity in $B(y,t)$. This is where we use our extra assumption about $X$.
Again we follow the proof of Proposition~7.1 in \cite{Holder}. First we observe that
if $X_m$ is a rectifiable cone of locally finite $\H^d$-measure (as in the statement),
not necessarily minimal, Lemma 7.34 in \cite{Holder} says that
$H^d(X\cap \d B) = 0$ for every ball $B$. This stays true for a countable union $X$
of such cones, and then also for $E_0 \cap B(x_0,r_0)$.

So \eqref{9.24}, and then \eqref{9.23} hold for all balls $B(y,t)$ such that 
$\overline B(y,t) \subset B(x_0,r_0)$.

Finally we need to deduce from this the conclusion of the lemma. 
This can be done with a small uniformity argument. We do not repeat it here,
and instead send the reader either to \cite{Holder}, 
starting from (7.15) on page 128 and ending below (7.19) on the next page,
where the proof is done in almost the same context, or to 
Lemma \ref{t9.4} and the argument that follows it, where we shall do it in a slightly 
more complicated situation.
Lemma \ref{t9.2} follows.
\end{proof}

We may now return to the proof of Theorem \ref{t9.1}. 
We observed just before Lemma \ref{t9.2} that all we have to do is
prove that the sets $E_k$ that we constructed satisfy the condition
\eqref{9.7} for $k$ large.

For this, it is enough to show that we can apply
Lemma \ref{t9.2} to the sets $E_k$, in the ball $B(0,r_1)$.
We already checked all the assumptions except one, the existence of the
closed rectifiable cone $X$.

Recall from \eqref{9.4} that $E_0$ satisfies the conclusions
of Theorem \ref{t8.1}. Then let $X$ be, as in \eqref{8.2}, the cone over $A \cap E$, 
where here $A = B(0,r_1)$ (see the remark below the statement of Theorem \ref{t8.1}).
By definition, $X$ contains $E \cap A = E \cap B(0,r_1)$.
By \eqref{8.3}, and for any $r < r_1$,
\begin{eqnarray} \label{9.25}
\H^d(X\cap B(0,r)) &\leq& \H^d(E\cap B(0,r)) + \H^d(S\cap B(0,r))
\nn\\
&\leq& \H^d(E\cap B(0,r)) + H(r)
\leq r^d F_{\infty}(r) < +\infty
\end{eqnarray}
where the bound on $\H^d(S\cap B(0,r))$ follows from \eqref{8.40}.
So $\H^d(X\cap B) < +\infty$ for every ball $B$. Next,
$X$ is (locally) rectifiable by \eqref{8.3}, because $E$ is locally rectifiable 
(see \eqref{2.11}), and $S$ is also locally rectifiable.

Maybe $X$ is not closed because of bad things that may happen near the origin, 
but it is a countable union of closed cones $X_m \i X$; for instance, we may take for 
$X_m$ the cone over $E \cap \overline B(0,r_1 - 2^{-m}) \sm B(0,2^{-m})$.
So Lemma \ref{t9.2} applies, and we get \eqref{9.7} and the desired contradiction.

This completes our proof of Theorem \ref{t9.1} by contradiction.
\qed

\ms
In the special case of the introduction when $L'$ is composed of a unique boundary 
piece $L$ which is an affine subspace of $\R^n$, we can try to choose an 
$\varepsilon$ in Theorem \ref{t9.1} that does not depend on the position of
$L$. We shall only do this when $0$ is not too close to the origin (see the
condition \eqref{9.26} below, but we shall include the possibility that
$L$ be very close (in a bilipschitz way) to an affine subspace, without 
necessarily being one. We leave open the case when $\dist(0,L)$ is very
small, but in this case we claim that it is probably just as convenient to center
the balls at some $x_0 \in L$, use the simple density 
$r^{-d}\H^d(E\cap B(x_0,r))$, and apply
Proposition~30.3 in \cite{Sliding}. See Remark \ref{t9.6}.

\ms
\begin{cor} \label{t9.3}
For each small $\tau > 0$ we can find $\varepsilon > 0$,
which depends only on $\tau$, $n$, and $d$, with the following property. 
Let $E \in SA^\ast M(U,L_j,h)$ be a coral sliding almost minimal set
(as in Definition \ref{t2.1}), associated to a single boundary set $L$,
a gauge function $h$, and an open set $U$ that contains $B(0,1)$.
Suppose that 
\begin{equation} \label{9.26}
\tau \leq \dist(0,L) \leq {9\over 10}, 
\end{equation}
and let $y_0 \in L$ be such that $\dist(0,L) = |y_0|$. 
Also suppose that there exists a vector subspace $L_0$, whose 
dimension $m$ is at most $d-1$, and a bilipschitz mapping 
$\xi : \R^n \to \R^n$, such that $\xi(L_0) = L$ and 
\begin{equation} \label{9.27}
(1-\varepsilon)|y-z| \leq |\xi(y)-\xi(z)| \leq (1-\varepsilon)|y-z|
\ \text{ for } y,z\in \R^n
\end{equation}
Finally suppose that
\begin{equation} \label{9.28}
h(1) < \varepsilon \ \text{ and } \ 
F(1) \leq \varepsilon + \inf_{0 < r < 10^{-3}} F(r).
\end{equation}
Then we can find a coral minimal cone $X_0$ in $\R^n$ (with no sliding boundary condition), and an affine subspace $L'$ of dimension $m$ through $y_0$, such that
\begin{equation}\label{9.29}
L' \cap B(0,99/100) \i X_0 \ \text{ if $m=d-1$,} 
\end{equation}
and, if \begin{equation}\label{9.30}
S = \big\{ z \in \R^n \, ; \, \lambda z \in L' \text{ for some } \lambda \in (0,1] \big\}
\end{equation}
denotes the shade of $L'$, 
\begin{eqnarray}\label{9.31}
&&E_0 = \overline{X_0 \sm S} \text{ is a coral minimal set in $B(0,1-\tau)$,} 
\nn\\
&&\hskip3.6cm
\text{ with sliding boundary condition defined by $L'$,}
\end{eqnarray}
\begin{equation} \label{9.32}
\dist(y, E_0) \leq \tau \ \text{ for } y \in E\cap B(0,1-\tau),
\end{equation} 
\begin{equation} \label{9.33}
\dist(y,E) \leq \tau \ \text{ for } y \in E_0 \cap B(0,1-\tau),
\end{equation}
and
\begin{eqnarray} \label{9.34}
\av{\H^d(E \cap B(y,t)) - \H^d(E_0 \cap B(y,t))} \leq \tau &&
\nonumber \\
&&\hskip-7cm
\ \text{ for all $y\in \R^n$ and $t>0$ such that } B(y,t) \i B(0,1-\tau).
\end{eqnarray}
\end{cor}

\ms
We are mostly interested in the case when $L$ is an affine
subspace, in which case Definition~\ref{t2.1} could be replaced
by the simpler Definition \ref{t1.4} (or its $A$- and $A_+$ variants).
In this case we take $L_0$ to be the vector space parallel to $L$,
and the proof will show that the conclusion holds with $L' = L$.

For our main application, $X_0$ will turn out to be a plane through $0$
or a $\bY$-set centered at $0$,  \eqref{9.29} will say that $X_0$
contains (a piece of) $L$, and this will determine $X_0$ and $E_0$.

Of course we don't really need to have a bilipschitz mapping defined
on the whole $\R^n$; a ball of radius $3$ centered at $\xi^{-1}(y_0)$
would be more than enough to cover $B(0,1)$ and prove the 
theorem.

When $L$ lies very close to the origin, we can still get an approximation
result, but only by a sliding minimal cone, and we claim that it should be as
simple to use results concerning ball that are centered on the boundary.
See Remark \ref{t9.6}. 
When $\dist(x,L) \geq 9/10$, the simplest is to restrict to
$B(0,9/10)$ and apply Proposition 7.24 in \cite{Holder}
to the plain almost minimal set $E$, with no boundary condition.
We would still get something like the conclusion above, maybe in a smaller ball
(notice that then $E_0 = X_0$ in $B(0,9/10)$, and 
\eqref{9.29} and \eqref{9.31} are void anyway).

\ms
We start the proof of the corollary as for Theorem \ref{t9.1} above. 
We may assume that the dimension of $L$ is a given integer $m \leq d-1$, 
and that $L_0$ is a fixed $m$-dimensional vector space.

We assume that we can find $\tau > 0$ such that the corollary fails for 
all $\varepsilon$, and we let $E_k$ provide a counterexample for the statement, 
with $\varepsilon = 2^{-k}$. This time, $E_k$ is a sliding almost minimal set
associated to a (changing) boundary set $L_k$, of fixed dimension $m$, 
and which is bilipschitz equivalent, through some mapping $\xi_k$ whose 
bilipschitz constants that tend to $1$, to the set $L_0$. 

By precomposing each $\xi_0$ with a translation in $L_0$ if needed, 
we may assume that $\xi_k(0) = y_{0,k}$, where 
$y_{0,k}$ denotes a point of $L_k$ such that $|y_{0,k}| = \dist(0,L_k)$.
Then we can use the uniform Lipschitz bounds on the $\xi_k$ to choose our 
subsequence so that the $\xi_k$ converge, uniformly in $\overline B(0,1)$, 
to some Lipschitz mapping $\xi_\infty$.
Since the $\xi_k$ satisfy \eqref{9.27} with biLipschitz constants that tend to
$1$, $\xi_\infty$ is an isometry from $\overline B(0,1)$ to its image, 
and it is known that there is an affine isometry of $\R^n$ that coincides with
$\xi_\infty$ on $\overline B(0,1)$. We subtract $\xi_\infty(0)$
and we get a linear isometry $\xi$ of $\R^n$ such that
\begin{equation} \label{9.35}
\xi(y) = \xi_\infty(y) - \xi_\infty(0) = 
\lim_{k \to +\infty} [\xi_k(y) - \xi_k(0)]
= \lim_{k \to +\infty} [\xi_k(y) - y_{0,k}]
\ \text{ for } y\in \overline B(0,1).
\end{equation}

By rotation invariance of our problem, the sets $\xi^{-1}(E_k)$ still provide 
a counterexample, so we may replace the $E_k$ by $\xi^{-1}(E_k)$;
thus we may assume that $\xi$ is the identity, i.e., that 
\begin{equation} \label{9.36}
\lim_{k \to +\infty} [\xi_k(y) - y_{0,k}] = y \ \text{ for } y\in \overline B(0,1).
\end{equation}
We want to use the $m$-planes $L_0 + y_{0,k}$ and their limit
\begin{equation} \label{9.37}
L_\infty = L_0+y_{0,\infty}, \ \text{ where } \,
y_{0,\infty} = \xi_\infty(0) = \lim_{k \to +\infty} y_{0,k},
\end{equation}
as boundaries. Maybe we should mention that in the simpler case of 
Corollary \ref{t9.3} when $L$ is an affine space, we may decide in advance
that $L_0$ was the vector space parallel to $L$, and then $\xi_\infty$ is a translation.
When we follow the argument above, we see that we do not need to modify
the $E_k$, and that we get $L_0 + y_{0,k} = L_k$. 

We replace $\{E_k\}$ with a subsequence for which 
$\{E_k\}$ converges, locally in $B(0,1)$, to a closed set $E_\infty$. 
This just means that $d_{0,\rho}(E,E_\infty)$ tends to $0$ for every 
$\rho \in (0,1)$, and the existence of a subsequence like this is classical. 

We claim that
\begin{eqnarray}\label{9.38}
&\,&\text{$E_\infty$ is a coral almost minimal set in $B(0,1-\tau/2)$,}
\nn\\
&\,&\hskip3cm
\text{with boundary condition coming from $L_\infty$,}
\end{eqnarray}
and with the special gauge function $h_0$ given by \eqref{9.10} with $r_1=1$.

Compared to what we did for Theorem \ref{t9.1}, the situation is slightly different,
because we have variable boundary conditions, so we shall have to use a
limiting result of Section~23 in \cite{Sliding} rather than its Section 10.
We want to apply Theorem 23.8, so let us check the assumptions.

We use the domain $U = B(0,1-\tau/2)$ and the single boundary piece $L_\infty$, 
and then the Lipschitz assumption (23.1) is satisfied.
We use the bilipschitz mappings $\wt\xi_k$ defined by
\begin{equation} \label{9.39}
\wt\xi_k(y) = \xi_k(y-y_{0,\infty}).
\end{equation}
Notice that if we set $U_k = \wt\xi_k(U)$, then 
\begin{equation} \label{9.40}
U_k = \wt\xi_k(B(0,1-{\tau \over 2})) = \xi_k(B(-y_{0,\infty},1-{\tau\over2}))
\subset y_{0,k} + B(-y_{0,\infty},1-{\tau\over 3})
= B(0,1-{\tau\over 4})
\end{equation}
by \eqref{9.36} and for $k$ large; in addition,
\begin{equation} \label{9.41}
\wt\xi_k(L_\infty) = \xi_k(L_0) = L_k,
\end{equation}
so the condition (23.2) of \cite{Sliding} is satisfied, the asymptotically
optimal Lipschitz bound (23.3) comes from \eqref{9.27}, and 
(23.4) (the pointwise convergence of the $\wt \xi_k$ to the identity)
follows from \eqref{9.36}, \eqref{9.39}, and \eqref{9.37}.

With the current notation, $E_k$ is sliding minimal in a domain that contains
$U_k$, with a boundary condition given by $\wt\xi_k(L_\infty)$, and the 
gauge function $h_k$. This implies that (23.5) holds with constants
$M$ that are arbitrarily close to $1$, $\delta$ arbitrarily close to $1$,
and $h$ as small as we want because of \eqref{9.28}
(see the discussion below \eqref{9.9}).

The assumption (23.6) comes from the convergence of the $E_k$ to
$E_\infty$, we don't need to check the technical assumptions (10.7) or (19.36)
because the $L_k$ are $m$-dimensional, and so Theorem 23.8
in \cite{Sliding} applies. We get that $E_\infty$ is sliding quasiminimal,
relative to $L_\infty$, and with constants $M$ arbitrarily close to $1$, $\delta$ 
arbitrarily close to $1$, and $h$ arbitrarily small; \eqref{9.38} follows.

Next we estimate Hausdorff measures. We start with the analogue of \eqref{9.11}.
By Remark~23.23 in \cite{Sliding}, we can apply Theorem 10.97 in \cite{Sliding}
as soon as the assumptions of Theorem 23.8 there are satisfied. We checked this to
prove \eqref{9.38}, so we get that if $V$ is an open set such that
$V \subset B(0,\rho)$ for some $\rho < 1$,  
\begin{equation}\label{9.42}
\H^d(E_\infty \cap V) \leq \liminf_{k \to +\infty} \H^d(E_k \cap V), 
\end{equation}
as in (23.23) in \cite{Sliding}. This stays true for any open set $V \i B(0,1)$
(just apply \eqref{9.42} to $V \cap B(0,\rho)$ and take a limit).
Similarly, Lemma 23.31 in \cite{Sliding} (applied with constants $M > 1$ arbitrarily
close to $1$ and $h > 0$ arbitrarily small) says that if $K$ is a compact subset of $B(0,1)$,
\begin{equation}\label{9.43}
\limsup_{k \to +\infty} \H^d(E_k \cap K) \leq \H^d(E_\infty \cap K).
\end{equation}
Alternatively we could use Lemma 23.36 in \cite{Sliding}, or copy its proof.
As in \eqref{9.14}, we deduce from \eqref{9.42} and \eqref{9.43} that 
\begin{equation}\label{9.44}
\lim_{k \to +\infty} \H^d(E_k \cap B(0,r)) = \H^d(E_\infty \cap B(0,r))
\end{equation}
for almost every $r \in (0,1)$. 

Denote by $F_k$ the functional associated to $E_k$ and the boundary $L_k$.
By \eqref{7.9},
\begin{equation}\label{9.45}
F_k(r) = r^{-d} \H^d(E_k \cap B(0,r)) + r^{-d} H_k(r),
\end{equation}
where $H_k(r)$ is given by \eqref{7.6} and \eqref{7.7} in terms of $L' = L_k$.
That is, 
\begin{equation}\label{9.46}
H_k(r) = d r^d \int_0^r {m_k(t) dt \over t^{d+1}},
\end{equation}
with $m_k(t) = \H^d(L_k^\sharp(t) \cap B(x,t))$ and 
$L_k^\sharp(t) = \big\{\lambda z \, ; \, \lambda \in [0,1] \text{ and } 
z\in L_k \cap \overline B(0,t)\big\}$. 

We also define $F_\infty$, $H_\infty$, $m_\infty$, and the $L_\infty^\sharp(t)$
similarly, but in terms of $E_\infty$ and $L_\infty = L_0+ y_{0,\infty}$ (see \eqref{9.37}). 
We want to check that
\begin{equation}\label{9.47}
F_\infty(r) = \lim_{k \to +\infty} F_k(r)
\end{equation}
for almost every $r\in (0,1)$ and, because of \eqref{9.44}, it will follow from the next lemma.

\begin{lem}\label{t9.4} 
We have that
\begin{equation}\label{9.48}
H_\infty(r) = \lim_{k \to +\infty} H_k(r) \ \text{ for } 0 < r < 1.
\end{equation}
\end{lem}

\begin{proof}
First notice that the lemma is particularly simple when the sets $L_k$ are affine
subspaces of dimension $m$ (and in this case it is also slightly 
simpler to use Remark \ref{t7.3} and the more direct formula \eqref{7.28} 
to compute $H_\infty$ and the $H_k$ in terms of shades),
and also when $m< d-1$ (because then $H_k = H_\infty = 0$). When the $L_k$ are bilipschitz
images of $(d-1)$-planes, we need to be a little careful about how the $H_k$ tend to a limit,
but hopefully the reader will not be surprised by the result.

So let us assume that $m=d-1$. The main step will be to check that for $0 < t < 1$,
\begin{equation}\label{9.49}
m_\infty(t) = \lim_{k \to +\infty} m_k(t).
\end{equation}
Recall from \eqref{9.41} that $L_k = \wt\xi_k(L_\infty)$,
so we get a parameterization of $L_k^\sharp(t)$ as 
\begin{equation}\label{9.50}
L_k^\sharp(t) = \big\{ \lambda \wt\xi_k(y) \, ; \, (\lambda,y) \in [0,1] \times Z_k(t) \big\},
\end{equation}
where $Z_k(t) = \big\{ y \in L_\infty \, ; \, \wt\xi_k(y) \in \overline B(0,t) \big\}$.
We use the area formula and get that
\begin{eqnarray}\label{9.51}
\H^d(L_k^\sharp(t)) &\leq& \int_{y\in Z_k(t)} \int_{t\in [0,1]} 
\lambda^{m-1} \Big| {\d \wt\xi_k(y) \over \d x_1} \wedge \ldots\wedge  {\d \wt\xi_k(y) \over \d x_m}
\wedge \wt\xi_k(y) \Big| dy d\lambda
\nn\\
&=& {1\over d} \int_{y\in Z_k(t)} 
\Big| {\d \wt\xi_k(y) \over \d x_1} \wedge \ldots\wedge  {\d \wt\xi_k(y) \over \d x_m}
\wedge \wt\xi_k(y) \Big| dy
\end{eqnarray}
where we took coordinates $(x_1, \ldots, x_m)$ in $L_\infty$ to compute 
the Jacobian of the parameterization. Unfortunately for the converse computation,
we only have an inequality
because the parameterization may fail to be injective. 

Recall from \eqref{9.39}, \eqref{9.36} and \eqref{9.37} that 
the $\wt\xi_k$ converge uniformly to the identity; then for each $\rho \in (t,1)$, 
$Z_k(t) \i  Z_\rho = : L_\infty \cap \overline B(0,\rho)$ for $k$ large, so
\begin{equation}\label{9.52}
\limsup_{k \to +\infty} \H^d(L_k^\sharp(t))
\leq {1\over d} \limsup_{k \to +\infty} \int_{y\in Z_\rho} 
\Big| {\d \wt\xi_k(y) \over \d x_1} \wedge \ldots\wedge  {\d \wt\xi_k(y) \over \d x_m}
\wedge \wt\xi_k(y) \Big| dy.
\end{equation}
Next we claim that the wedge product
${\d \wt\xi_k(y) \over \d x_1} \wedge \ldots {\d \wt\xi_k(y) \over \d x_m}$
converges weakly to the constant $m$-vector $e_1  \wedge \ldots\wedge e_m$,
where the $e_j$ are the corresponding basis vectors of $L_0$ 
(the vector space parallel to $L_\infty$), 
in the sense that for each small product $Q$ of intervals 
(with faces parallel to the axes) in $L_\infty$,
\begin{equation}\label{9.53}
\lim_{k\to +\infty} \int_Q 
{\d \wt\xi_k(y) \over \d x_1} \wedge \ldots \wedge {\d \wt\xi_k(y) \over \d x_m}
\, dy
= |Q| e_1  \wedge \ldots\wedge e_m.
\end{equation}
This last follows from integrating coordinate by coordinate,
and using the fact that the $\wt\xi_k(y)$ converge uniformly to $y$.
Also, we have uniform bounds on $\nabla \wt\xi_k$,
by \eqref{9.27}, which allows one to pass from the dense class of linear combinations
of characteristic functions $\1_{Q}$ to the vector-valued function 
$y \, \1_{Z_\rho}$, so
\begin{equation}\label{9.54}
\lim_{k \to +\infty} {1\over d} \int_{y\in Z_\rho} 
\Big| {\d \wt\xi_k(y) \over \d x_1} \wedge \ldots\wedge  
{\d \wt\xi_k(y) \over \d x_m} \wedge y \Big| 
= {1\over d} \int_{y\in Z_\rho}  \Big| e_1  \wedge \ldots\wedge e_m 
\wedge y \Big| .
\end{equation}
We can replace $y$ by $\wt\xi_k(y)$ in the right-hand side of \eqref{9.54}
because the $\wt\xi_k$ are uniformly Lipschitz, and converge uniformly 
to the identity; we get that
\begin{equation}\label{9.55}
\limsup_{k \to +\infty} \H^d(L_k^\sharp(t))
\leq {1\over d} \int_{y\in Z_\rho}  \Big| e_1  \wedge \ldots\wedge e_m 
\wedge y \Big| 
= \H^d(L_\infty^\sharp(\rho)),
\end{equation}
where the last part comes from the same computations as above,
i.e., applying the area formula to 
\begin{equation}\label{9.56}
L_\infty^\sharp(\rho) 
= \big\{ \lambda z \, ; \, \lambda \in [0,1] \text{ and } y \in L_\infty \cap B(0,\rho)\big\}
\end{equation}
(this time, our parameterization is injective). We let $\rho$ tend to $t$ from above and get that
\begin{equation}\label{9.57}
\limsup_{k \to +\infty} m_k(t)
=\limsup_{k \to +\infty} \H^d(L_k^\sharp(t)) 
\leq \H^d(L_\infty^\sharp(t)) = m_{\infty}(t)
\end{equation}
by \eqref{7.7}. We may now concentrate on the other inequality in \eqref{9.49}, which we shall obtain 
by topology.

For $\rho < 1$, still set $Z_\rho = L_\infty \cap \overline B(0,\rho)$,
and set $T_\rho = \big\{ \lambda y \, ; \, \lambda \in [0,1]
\text{ and } y\in Z_\rho \big\}$. The $T_\rho$ are homothetic cones in a
fixed space $V$ of dimension $d$; let $\pi_V$ denote the orthogonal projection
on that space. 

Let $t\in (0,1)$ be given, and pick $\rho < t$. For $k$ large, 
$Z_\rho \i Z_k(t)$, so 
\begin{equation} \label{9.58}
L_k^\sharp(t) \supset 
\big\{ \lambda \wt\xi_k(y) \, ; \, (\lambda,y) \in [0,1] \times Z_\rho \big\},
\end{equation}
by \eqref{9.50}. We can even define a continuous mapping 
$h_k : T_\rho \to L_k^\sharp(t)$, by setting $h(z) = \lambda \wt\xi_k(y)$
when $z = \lambda y$ for some $\lambda \in [0,1]$ and $y\in Z_\rho$
(notice that $y$ can be computed from $z$, since it is its radial projection on
$L_\infty$).

Next let $\rho_1 < \rho$ and $\xi\in T_{\rho_1}$ be given. Notice that for
$k$ large and $z\in \d T_\rho$ (the boundary of $T_\rho$ in the space $V$),
\begin{equation} \label{9.59}
\dist(\pi_V(h_k(z)),\xi) \geq \dist(\pi_V(h_k(z)),T_{\rho_1})
\geq {1 \over 2} \dist(\d T_\rho,T_{\rho_1}) > 0
\end{equation}
because $\pi_V \circ h_k(z)$ tends to $z$ uniformly on $T_\rho$
(since $\wt h_k(y)$ tends to $y$ uniformly on $Z_\rho$).
The same argument also shows that
\begin{equation} \label{9.60}
\dist(w,\xi) \geq {1 \over 2}\dist(\d T_\rho,T_{\rho_1}) > 0
\end{equation}
for $w\in [z,\pi_V(h_k(z))]$ (and $z$ as above). This means that
there is a homotopy, from the identity on $\d T_\rho$ to the mapping
$\pi_V \circ h_k$ on $\d T_\rho$, among continuous mappings 
with values in $V$ and that do not take the value $\xi$. 

Denote by $\pi_\xi$ the radial projection, centered at $\xi$, that
maps any point $v\in V\sm \{ \xi \}$ to the point $\pi_\xi \in \d T_\rho$ 
that lies on the half line from $\xi$ through $v$. The mapping
$\pi_\xi \circ \pi_V \circ h_k$ is continuous, from $\d T_\rho$ to itself,
and had degree $1$ (when we identify $\d T_\rho$ with a $(d-1)$-sphere)
because it is homotopic to the identity.
This implies that it cannot be extended to a continuous mapping from
$T_\rho$ to $\d T_\rho$, and in turns this implies that
$\xi \in \pi_V \circ h_k(T_\rho)$
(otherwise, $\pi_\xi \circ \pi_V \circ h_k$ would be a continuous extension).

We proved that for $k$ large, $\pi_V \circ h_k(T_\rho)$ contains $T_{\rho_1}$, 
hence, by \eqref{9.58} $\pi_V(L_k^\sharp(t))$ contains $T_{\rho_1}$. Then
\begin{equation} \label{9.61}
m_k(t) = \H^d(L_k^\sharp(t)) \geq \H^d(\pi_V(L_k^\sharp(t)))
\geq \H^d(T_{\rho_1}) 
\end{equation}
by \eqref{7.7} and because $\pi_V$ is $1$-Lipschitz.
This is true for all choices of $\rho_1 < \rho < t$; when $\rho_1$
tends to $t$, the right-hand side of \eqref{9.61} tends to
$\H^d(T_{t}) = m_\infty(t)$ (by \eqref{7.7} again). Thus
\begin{equation} \label{9.62}
\liminf_{k \to +\infty} m_k(t) \geq m_\infty(t),
\end{equation}
which gives the second half of \eqref{9.49}.

Notice that (by the proof of \eqref{9.57}) the $m_k$ are uniformly 
bounded on $[0,r]$, $r < 1$. They also vanish on $[0,\tau)$, because 
since the $L_k$ satisfy \eqref{9.26}, $B(0,\tau)$ does not meet $L_k$
and $L_k^\sharp(t) = \emptyset$ for $t<\tau$; see below \eqref{9.46}.
Our conclusion \eqref{9.48} now follows from
the dominated convergence theorem, the definition \eqref{9.46}, 
\eqref{9.49}, and \eqref{7.7} (for $H_\infty$).
\end{proof}

By Lemma \ref{9.4}, \eqref{9.44}, and the definitions \eqref{7.9}
and its analogue for $F_\infty$, we get \eqref{9.47}. We then proceed 
as in the proof of \eqref{9.17}. We first apply Theorem \ref{t7.1} to $E_\infty$ 
with the gauge function $h_0$ (of course $L_\infty$ satisfies the conditions of 
Section \ref{S7}, and Remark \ref{t7.2} still allows us not to check \eqref{7.10});
then Theorem \ref{t7.1} says that $F_\infty$ is nondecreasing on $(0,1)$
and the argument of \eqref{9.18} yields that $F_\infty$ is constant on $(0,1)$. 

Now let us apply Theorem \ref{t1.3}, with $R_1 = 1$ and $R_0$ arbitrarily small. 
The main assumption is satisfied because $F_\infty$ is constant on $(0,1)$.
Notice that 
\begin{equation} \label{9.63}
\dist(0, L_\infty) \geq \tau
\end{equation}
because $\dist(0,L_k) \geq \tau$ by \eqref{9.26}, 
$\wt \xi_k(L_\infty) = L_k$ by \eqref{9.41}, and the $\wt \xi_k$ converge 
pointwise to the identity (see below \eqref{9.41}).

Set $A = B(0,1) \sm B(0,R_0)$, and denote by $X_\infty$ the cone over 
$A \cap E_\infty$ (as in \eqref{1.13}); we take $R_0 < \tau$, and then 
we know that $X_\infty$ is a coral minimal cone (with no boundary condition).
Set $A' = B(0,1) \sm \overline B(0,R_0)$, and let us check that
\begin{equation}\label{9.64}
A' \cap E_\infty = A' \cap \overline{X_\infty \sm S_\infty},
\end{equation}
where $S_\infty$ is the shade of $L_\infty$ seen from the origin 
(as in \eqref{7.12}).
By \eqref{1.12}, $H^d(A' \cap E_\infty \cap S_\infty) = 0$;
then each $x\in E_\infty \cap A'$ is the limit of a sequence $\{ x_j \}$ in
$E_\infty \sm S_\infty$ (recall that $E_\infty$ is coral). Obviously, $x_j \in A'$
for $j$ large, hence, by definition of $X_\infty$, $x_j \in X_\infty$; it follows that
$x\in A' \cap \overline{X_\infty \sm S_\infty}$. Conversely, if
$x\in A' \cap X_\infty \sm S_\infty$, \eqref{1.14} says that $x\in E_\infty$;
\eqref{9.64} follows.

Since \eqref{9.64} and \eqref{9.63} say that $E_\infty = X_\infty$ on 
$B(0,\tau) \sm \overline B(0,R_0)$,
and $X_\infty$ is a cone, we see that it does not depend on $R_0$ 
(provided that we take $R_0 < \tau$), and then (letting $R_0$ tend to $0$),
\begin{equation}\label{9.65}
B(0,1) \cap E_\infty = B(0,1) \cap \overline{X_\infty \sm S_\infty}.
\end{equation}
Here we took the intersection with $B(0,1)$, but we do not really need to: all our sets
are initially defined as subsets of $B(0,1)$. We still need to modify the sets 
$X_\infty$ and $E_\infty = \overline{X_\infty \sm S_\infty}$
a little to get the desired $X_0$ and $E_0$, because 
$E_\infty$ is minimal with a sliding boundary condition defined by $L_\infty$, 
while we promised in \eqref{9.31} that $E_0 = E_{0,k}$ would be sliding minimal 
with respect to an affine $d$-plane $L' = L'_k$ through $y_{0,k}$.

So we replace $X_\infty$ and $E_\infty$ with slightly different sets.
Recall from \eqref{9.37} that $L_\infty = L_0 + y_{0,\infty}$ and that
$y_{0,\infty}$ is the limit of the $y_{0,k}$; also, 
$|y_{0,\infty}| = \dist(0,L_\infty) \geq \tau$
because $|y_{0,k}| = \dist(0,L_k)$, by \eqref{9.41}, and by \eqref{9.63}. 
Then we can find numbers $\rho_k$, that tend to $1$, 
and linear isometries $I_k$, that converge to
the identity, so that $y_{0,k} = \rho_k I_k(y_{0,\infty})$. We set 
\begin{equation} \label{9.66}
X_0 = X_{0,k} = I_k(X_\infty), \, 
E_0 = E_{0,k} =  \rho_k I_k(E_\infty)
\ \text{ and } 
L' = L'_k = \rho_k I_k(L_\infty).
\end{equation}
Notice that $L'_k$ is a $m$-dimensional affine subspace that contains 
$y_{0,k}$, because $L_\infty$ contains $y_{0,\infty}$ (by \eqref{9.37}). 
We want to show that for $k$ large, these set satisfy the properties
announced in Corollary \ref{t9.3}.

For the special case of Corollary \ref{t9.3} where $L$ is an affine
space, we can assume that the $L_k$ are affine subspaces, 
all of the same dimension and parallel to the same vector space $L_0$.
For this case, we announced that we would take $L'_k=L_k$, so let us check 
that we can choose $\rho_k$ and $I_k$ so that this is the case. 
Recall that $y_{0,k}$ is the orthogonal projection of $0$ on $L_k$,
and hence, since the $L_k$ converge to $L_\infty$, $y_{0,\infty}$
is the projection of $0$ on $L_\infty$. We need to set 
$\rho_k = |y_{0,k}|/|y_{0,\infty}|$. If $y_{0,k}$ is collinear to $y_{0,\infty}$,
we just take $I_k$ to be the identity. Otherwise, let us first define $I_k$
on the $2$-plane $P_k$ that contains $y_{0,k}$ and $y_{0,\infty}$, so that it
preserves $P_k$ and maps $y_{0,\infty}$ to $y_{0,k} |y_{0,\infty}|/|y_{0,k}|$.
Then set $I_k(z) = z$ on $P_k^\perp$. It is easy to see that $I_k$
is an isometry; since it preserves $L_0 \i P_k^\perp$, the mapping
$\rho_k I_k$ sends $L_\infty$ to $L_k$, as needed.

The fact that $X_\infty$ and  $X_0$ are coral minimal cones (no boundary condition)
comes from our application of Theorem \ref{t1.3} (see below \eqref{9.63}).

For \eqref{9.29}, we need to check that if $m=d-1$,
$L'_k \cap B(0,99/100) \i X_0$, or equivalently (by \eqref{9.66}),
\begin{equation} \label{9.67}
L_\infty \cap B(0, 99\rho_k^{-1}/100) \i I_k^{-1}(X_0) = X_\infty.
\end{equation}
Recall from above \eqref{9.63} that we were able to apply Theorem \ref{t1.3}
to $E_\infty$, with $R_1=1$ and $R_0$ arbitrarily small; then \eqref{1.15} holds,
which says that $\H^d(S_\infty \cap B(0,1) \sm X_\infty) = 0$; 
since $S_\infty$ is a $d$-dimensional set 
(because $L_\infty$ does not contain the origin), we deduce from this that
$X_0$ contains $L_\infty \cap B(0,1)$ (recall that $X_\infty$ is closed too);
\eqref{9.67} (for $k$ large) follows. So \eqref{9.29} holds.

The next condition \eqref{9.31} is equivalent to the fact that
$E_\infty$ is a coral minimal set in $B(0,(1-\tau) \rho_k^{-1})$,
with boundary condition given by $L_\infty = \rho_k^{-1}I_k^{-1}(L')$,
and this follows from \eqref{9.38} as soon as $k$ is so large that
$(1-\tau) \rho_k^{-1} < 1-\tau/2$.

The next conditions \eqref{9.32} and \eqref{9.33} come from
the fact that $E_\infty$ is the limit, locally in $B(0,1)$, of the
$E_k$, and that $\rho_k I_k$ tends to the identity.
We are thus left with \eqref{9.34} to check. That is, we need to show that 
for $k$ large,
\begin{equation}\label{9.68}
|\H^d(B \cap E_k) - \H^d(B \cap E_{0,k})| \leq \tau
\end{equation}
for every ball $B = B(y,t)$ such that $B \i B(0,(1-\tau))$. We intend to proceed 
as in Lemma~\ref{t9.2}, and the main step is the following small lemma.

\begin{lem} \label{t9.5}
Denote by $H$ the set of pairs $(y,t)$, with $y\in \overline B(0,(1-\tau/2))$ 
and $t \geq 0$, such that $\overline B(y,t) \i \overline B(0,(1-\tau/2))$. 
We include $t=0$ (and set $B(x,0) = \emptyset$) to make sure that $H$ is compact.
Set
\begin{equation} \label{9.69}
h(y,t) = \H^d(E_\infty \cap B(y,t)) \ \text{ for } (y,t) \in H.
\end{equation}
Then $h$ is a continuous function on $H$.
\end{lem}

\begin{proof}
We prove the continuity of $h$ at the point $(y,t) \in H$. 
Suppose $\{ (y_j, t_j) \}$ is a sequence in $H$ that tends to $(y,t)$,
and denote by $\Delta_j$ the symmetric difference
\begin{equation} \label{9.70}
\Delta_j = [B(y_j,t_j) \sm B(y,t)] \cup [B(y,t) \sm B(y_j,t_j)];
\end{equation}
then
\begin{equation}  \label{9.71}
\limsup_{j \to \infty} |h(y,t)-h(y_j,t_j)|
\leq \limsup_{j \to \infty} \H^d(E_\infty \cap \Delta_j)
\leq \H^d(E_\infty \cap \d B(y,t))
\end{equation}
because if $V$ is any open neighborhood of $\d B(y,t)$, 
$\Delta_j$ is contained in $V$ for $j$ large (when $t=0$, we replace 
$\d B(y,t)$ by $\{ y \}$ to make this work).

But Lemma 7.34 in \cite{Holder} implies, as for \eqref{9.24} above, that
\begin{equation}\label{9.72}
\H^d(E_\infty \cap \d B(y,t)) = 0;
\end{equation}
thus $h(y,t)-h(y_j,t_j)$ tends to $0$ and the lemma follows.
\end{proof}

\ms
Let us return to the proof of \eqref{9.34}, fix a point $(y,t) \in H$, and prove that
\begin{equation} \label{9.73}
\lim_{k \to +\infty} \H^d(B(y,t) \cap E_k) = \H^d(B(y,t) \cap E_\infty)
= \lim_{k \to +\infty} \H^d(B(y,t) \cap E_{0,k}).
\end{equation}
The first part follows from \eqref{9.42}, \eqref{9.43},
and the fact that $\H^d(E_\infty \cap \d B(y,t)) = 0$ exactly as for
\eqref{9.44} or \eqref{9.14}. For the second part, we use the definition
\eqref{9.66} and get that
\begin{equation} \label{9.74}
B(y,t) \cap E_{0,k} = B(y,t) \cap \rho_k I_k(E_\infty)
= \rho_k I_k(B(y_k,t_k) \cap E_\infty),
\end{equation}
with $y_k = \rho_k^{-1} I_k^{-1}(y)$ and $t_k = \rho_k^{-1} t$.
Then 
\begin{equation} \label{9.75}
\H^d(B(y,t) \cap E_{0,k}) = \rho_k^d \H^d(B(y_k,t_k) \cap E_\infty),
\end{equation}
which tends to $\H^d(B(y,t) \cap E_{\infty})$ by Lemma \ref{t9.5}.
So \eqref{9.73} holds.

Next let $\tau > 0$ be given (as in the statement).
Since $H$ is compact and $h$ is continuous on $H$, there is a constant
$\eta > 0$ such that
\begin{equation} \label{9.76}
|h(y,t) - h(y',t')| \leq \tau/10
\ \text{ for $(y,t), (y',t') \in H$ such that } 
|y-y'|+|t-t'| \leq 5\eta.
\end{equation}
We may assume that $10 \eta < \tau$. Then let $Y$ be a finite
subset of $B(0,1)$ which is $\eta$-dense in $B(0,1)$, and
$T$ a finite subset of $[0,1]$ which is $\eta$ dense. Let us include
$t=0$ in $T$. Then let $H_0$ denote the set of pairs $(y,t)\in H$
such that $y\in Y$ and $t\in T$. By \eqref{9.73}, we get that for $k$ large
\begin{equation} \label{9.77}
|\H^d(B(y,t) \cap E_k) - \H^d(B(y,t) \cap E_\infty)|
+|\H^d(B(y,t) \cap E_\infty)-\H^d(B(y,t) \cap E_{0,k})| \leq \tau/10
\end{equation}
for $(y,t) \in H_0$. For \eqref{9.68}, we want a similar estimate for every
pair $(y,t)$ such that $B(y,t) \i B(0,1-\tau)$. Let us fix such a pair
and first try to choose pairs $(y',t_1), (y',t_2) \in H_0$, so that if we set
$B = B(y,t)$, $B_1 = B(y',t_1)$, and $B_2 = B(y',t_2)$, then
\begin{equation} \label{9.78}
B_1 \subset B \subset B_2 \ \text{ and } \ t_2-t_1 \leq 4 \eta.
\end{equation}
Let us take $y'\in Y$ such that $|y'-y| \leq \eta$. If 
$t > \eta$, take $t_1 \in T \cap [t-2\eta, t-\eta]$; otherwise,
take $t_1=0$.  In both cases, take $t_2 \in T \cap [t+\eta,t+2\eta]$.
With these choices, we get \eqref{9.78}, and also both $(y',t_j)$ lie in $H_0$.
Then for instance
\begin{equation} \label{9.79}
\H^d(B_1 \cap E_k) \leq \H^d(B \cap E_k) \leq \H^d(B_2\cap E_k),
\end{equation}
which by \eqref{9.77} yields
\begin{equation} \label{9.80}
\H^d(B_1 \cap E_\infty) - \tau/10
\leq \H^d(B \cap E_k) \leq \H^d(B_2\cap E_\infty) + \tau/10
\end{equation}
for $k$ large, and since
\begin{equation} \label{9.81}
\H^d(B_1 \cap E_\infty) \leq \H^d(B \cap E_\infty) 
\leq \H^d(B_2\cap E_\infty) \leq \H^d(B_1 \cap E_\infty) + \tau/10
\end{equation}
by \eqref{9.78} and \eqref{9.76}, we get that
\begin{equation} \label{9.82}
|\H^d(B \cap E_k) - \H^d(B \cap E_\infty)| \leq \tau/5
\end{equation}
for $k$ large. The same proof also yields
\begin{equation} \label{9.83}
|\H^d(B \cap E_{0,k}) - \H^d(B \cap E_\infty)| \leq \tau/5,
\end{equation}
and \eqref{9.68} follows. This completes our proof that for $k$ large,
the sets $X_{0,k}$ and $E_{0,k}$ satisfy all the properties \eqref{9.29}-\eqref{9.34} 
(relative to $E_k$ and $L_k$) that were required in the statement of 
Corollary \ref{t9.3}. This contradicts the initial definitions, and completes our proof
of Corollary~\ref{t9.3}. The additional statement below the corollary (concerning
the case of affine subspaces) was checked below \eqref{9.66}.
\qed

\ms
\begin{rem} \label{t9.6}
We could try to prove an analogue of Corollary~\ref{t9.3} when the origin
lies very close to the boundary set $L$, but if $L$ is $(d-1)$-dimensional
and without more precise assumptions on $L$ 
(for instance, uniform $C^1$ estimates, rather than bilipschitz, that say
that $L$ is close to an affine subspace), we will not get a good convergence of 
the functions $H_k$ to their analogue for $L_\infty$, as in Lemma \ref{t9.4} 
above, and then we shall not be able to show that $F_\infty$ is constant and
apply Theorem~\ref{t1.3} as above.

Even if we do (for instance, if $L$ is assumed to be an affine subspace),
we do not get the same conclusion as before, because we may get that
$y_{0,\infty} = 0$ and $L_\infty = L_0$, and then we cannot take 
$R_0 < \dist(0,L_\infty)$ to prove that $X_0$ is a minimal cone.
Instead, we only get the approximation of $E_k$ by a sliding minimal
cone, with boundary condition given by $L_0$.

We shall not try to pursue this here, because it seems as convenient,
when $0$ is very close to $L$, to consider balls centered at a point $x_0 \in L$,
try to deduce some interesting information on the density $\theta_{x_0}(r)$
from the assumption \eqref{9.28}, and then apply Proposition 30.3 in
\cite{Sliding} to show that $E$ is well approximated by a sliding minimal cone,
with boundary condition given by an affine subspace of the same dimension
as  the $L_0$ from \eqref{9.27}.
See the argument below \eqref{11.42} and the proof of Proposition \ref{t12.7}
for illustrations of this scheme.

Generally speaking, if we we have sets $L_j$ that are not necessarily a single
$m$-plane, and we still want an analogue of Theorem \ref{t9.1} 
where $\varepsilon$ depends only on $n$ and $d$, and not on the specific
choices of $L_j$ or $r_1$, we can always try to mimic the proof of
Theorem \ref{t9.1} or Corollary~\ref{t9.3}, but we will need to understand how the function $H$ depends on the specific $L_j$, and this seems easier to do 
on a case by case basis. 
Notice that when each $L_k$ in the proof above is composed of two planes 
that both tend to the same plane $L_\infty$, 
Lemma \ref{t9.4} fails in general.
\end{rem}

\ms
We relax a little and end this section with the variant of Theorem \ref{t9.1} 
that corresponds to an annulus.

\ms 
\begin{thm} \label{t9.7}
Let $U$ and the $L_j$ be as in Section \ref{S7},
and in particular, assume \eqref{2.1}, \eqref{2.2}, and \eqref{3.4}.
Let $r_0, r_1$ be such that $0 < r_0 < r_1$ and $B(0,r_1) \i U$,
and assume that almost every $r \in (r_0,r_1)$ admits a local retraction 
(as in Definition \ref{t3.1}).
For each small $\tau > 0$ we can find $\varepsilon > 0$,
which depends only on $\tau$, $n$, $d$, $U$, the $L_j$, 
$r_0$, and $r_1$, with the following property. 
Let $E$ be a coral sliding almost minimal set in $U$,  with sliding condition defined 
by $L$ and some nondecreasing gauge function $h$. Suppose that
\begin{equation} \label{9.84}
\text{ $B(0,r_1) \i U$ and $h(r_1) < \varepsilon$,}
\end{equation}
and  
\begin{equation} \label{9.85}
F(r_1) \leq F(r_0) + \varepsilon  < +\infty.
\end{equation}
Then there is a coral minimal set $E_0$ in $B(0,r_1)$, with 
sliding condition defined by $L$, such that
\begin{eqnarray} \label{9.86}
\text{ the analogue of $F$ for the set $E_0$ is constant on }(r_0,r_1),
\end{eqnarray}
\begin{equation} \label{9.87}
\text{the conclusions of Theorem \ref{t8.1} hold for $E_0$,
with the radii $R_0 = r_0$ and $R_1 = r_1$,}
\end{equation}
\begin{equation} \label{9.88}
\dist(y, E_0) \leq \tau r_1 \ \text{ for } y \in E\cap B(0,r_1),
\end{equation} 
\begin{equation} \label{9.89}
\dist(y,E) \leq \tau r_1 \ \text{ for } y \in E_0 \cap B(0,r_1),
\end{equation}
and
\begin{eqnarray} \label{9.90}
\av{\H^d(E \cap B(y,t)) - \H^d(E_0 \cap B(y,t))} \leq \tau r_1^d
\text{ for all} &&
\nonumber \\
&&\hskip-9cm
\ \text{ $y\in \R^n$ and $t>0$ such that }
B(y,t) \i B(0,(1-\tau)r_1) \sm B(0,(1+\tau)r_0). 
\end{eqnarray}
\end{thm}

\ms
There would probably be a way to state Theorem \ref{t9.7} so that
$\varepsilon$ does not depend on $r_0$, but we shall not do it.
Also, we shall not try to generalize Corollary \ref{t9.3} to 
the case of an annulus.

\ms
The proof is almost the same as for  Theorem \ref{t9.1}.
We change nothing up to \eqref{9.18}, which we replace with
\begin{eqnarray} \label{9.91}
F_\infty (r_1) &=& r_1^{-d} \H^d(E_0 \cap B(0,r_1)) + r_1^{-d} H(r_1)
\leq \liminf_{k \to +\infty} \H^d(E_k \cap B(0,r_1)) + r_1^{-d} H(r_1)
\nn\\
&=& \liminf_{k \to +\infty} F_k(r_1)
\leq \liminf_{k \to +\infty} (F_k(r_0)+2^{-k})
= \liminf_{k \to +\infty} F_k(r_0)
\nn\\
&=& r_0^{-d} H(r_0) + r_0^{-d} \liminf_{k \to +\infty} \H^d(E_k \cap B(0,r_0))
\end{eqnarray}
by \eqref{9.11} and \eqref{9.85}. But for $r_0 < r$,
$H(r_0) \leq H(r)$ and 
\begin{eqnarray} \label{9.92}
\liminf_{k \to +\infty} \H^d(E_k \cap B(0,r_0))
&\leq& \limsup_{k \to +\infty} \H^d(E_k \cap \overline B(0,r_0))
\leq \H^d(E_0 \cap \overline B(0,r_0)) 
\nn\\
&\leq& \H^d(E_0 \cap B(0,r)),
\end{eqnarray}
by \eqref{9.12}, so $F_{\infty}(r_1) \leq F_\infty(r)$ for
$r_0 < r < r_1$. The fact that $F_\infty$ is nondecreasing on $(r_0,r_1)$
is proved as before, so we get that $F_\infty$ is constant on 
$(r_0,r_1)$, as in \eqref{9.17}. The rest of the argument is as before;
we only prove \eqref{9.90} for balls that do not meet 
$B(0,(1+\tau)r_0)$, because we use the fact that in the annulus $A$, 
$E$ is contained in the cone $X$. 
\qed

\section{Simple properties of minimal cones}
\label{S10}
Before we start proving the application mentioned in the introduction, we shall
introduce some properties of minimal cones (even, without sliding boundary
condition), that will be used in the proofs. 

Let us denote by $MC = MC(n,d)$ the set of coral minimal sets of dimension
$d$ in $\R^n$, which are also cones centered at the origin. 

The set $MC(n,d)$ is only known explicitly when $d=1$
(and then $MC$ is composed of lines and sets $Y \in \bY_0(n,1)$; see 
the definition above \eqref{1.28}), and when $d=2$ and $n=3$,
where $MC(3,2)$ is composed of $2$-planes, sets $Y \in \bY_0(3,2)$,
and cones of type $\bT$ (our name for a cone over the union of the edges of a regular
tetrahedron centered at the origin); see for instance
\cite{Mo}. Even $MC(4,2)$ is not known explicitly, but at least we have a rough 
description of the minimal cones of $MC(n,2)$ for all $n > 3$.

For $X\in MC(n,d)$, the density of $X$ is
\begin{equation} \label{10.1}
d(X) = \H^d(X\cap B(0,1)).
\end{equation}
Notice that by the monotonicity of density (see near \eqref{1.8}),
we get that for $X\in  MC(n,d)$, $x\in \R^n$ and $r > 0$,
\begin{eqnarray} \label{10.2}
r^{-d} \H^d(X\cap B(x,r)) &\leq& \lim_{\rho \to +\infty} r^{-d} \H^d(X\cap B(x,r))
\leq \lim_{\rho \to +\infty} r^{-d} \H^d(X\cap B(0,r+|x|))
\nn\\
&=& \lim_{\rho \to +\infty} r^{-d} (r+|x|)^d d(X) = d(X).
\end{eqnarray}
Denote by 
\begin{equation} \label{10.3}
\omega_d = \H^d(\R^d \cap B(0,1))
\end{equation}
the $\H^d$-measure of the unit ball in $\R^d$.
Since we know that the minimal sets are rectifiable, and also that
\begin{equation} \label{10.4}
\lim_{r \to 0} r^{-d} \H^d(X\cap B(x,r)) = \omega_d
\end{equation}
for $\H^d$-almost every point $x$ of a rectifiable set $X$
(see for instance the easy part of Theorem~16.2 in \cite{Mattila}), 
we get that $d(X) \geq \omega_d$
for $X\in MC(n,d)$. We can then try to classify the minimal cones by their density.
The beginning is easy. To prove that
\begin{equation} \label{10.5}
\text{if $X \in MC(n,d)$ and $d(X) \leq \omega_d$, then $X$ is a $d$-plane,}
\end{equation}
one observes that if $d(X) = \omega_d$, then \eqref{10.2} is an identity
for almost-every point $x\in X$. A close look at the proof of the monotonicity
of density then shows (with some effort but no surprise) that $X$ is also a cone
centered at $x$, and it is then easy to conclude. Here is a quantitative
(but not explicit) version of this.

\begin{lem} \label{t10.1}
For each choice of integers $0 < d < n$, there is a constant $d(n,d) > \omega_d$
such that $d(X) \geq d(n,d)$ for $X \in MC(n,d) \sm \bP_0(n,d)$.
\end{lem}

\begin{proof}
We prove this by contradiction and compactness. 
Suppose that for each integer $k \geq 0$ we can find 
$X_k \in MC(n,d) \sm \bP_0(n,d)$
such that $d(X) \leq \omega_d + 2^{-k}$. We may replace $\{ X_k \}$
with a subsequence for which the $X_k$ converge to a limit $X_\infty$,
i.e., that $d_{0,N}(X_k, X_\infty)$ tends to $0$ for each integer $N$
(see the definition \eqref{1.29}).
Since all the $X_k$ are cones, $X_\infty$ is a cone too. By Theorem 4.1
in \cite{limits} (with $\Omega = \R^n$, $M=1$ and $\delta = +\infty$), 
$X_\infty$ is a coral minimal set in $\R^n$. By the lowersemicontinuity of
$\H^d$ along that sequence (for instance Theorem~3.4 in \cite{limits}),
$d(X_\infty) \leq \liminf_{k \to +\infty} d(X_k) = \omega_d$. Since
$d(X_\infty) \geq \omega_d$ anyway, \eqref{10.5} says that 
$X_\infty \in \bP_0(n,d)$. Since $d_{0,2}(X_k, X_\infty)$ tends to $0$,
$X_k$ is arbitrarily close to a plane in $B(0,1)$. In addition, by assumption
$\H^d(X_k \cap B(0,1)) \leq \omega_d + 2^{-k}$. We may now apply
either Almgren's regularity result for almost minimal sets 
\cite{AlmgrenMemoir}, 
or Allard's regularity result for stationary varifolds \cite{Allard}
(whichever the reader finds easiest), and get that for $k$ large,
$X_k$ is $C^1$ near the origin. But $X_k$ is a cone, and this means that
$X_k$ is a plane. This proves the lemma.
\end{proof}

The following consequence of Lemma \ref{t10.1} will be used in the next section.

\begin{cor} \label{t10.2}
Let $L$ be a vector space of dimension $m < d$ in $\R^n$, and let
$E$ be a coral sliding minimal cone, with boundary condition given by $L$.
Suppose that $\H^d(E \cap B(0,1)) \leq {\omega_d \over 2}$. Then
$m=d-1$ and $E$ is a half $d$-plane bounded by $L$.
\end{cor}

\begin{proof}
Of course we assume that $E \neq \emptyset$. Then, since $E$ is 
coral and rectifiable (by \eqref{2.11}), we can find $x\in E \sm L$ 
such that \eqref{10.4} holds.
Let us apply Theorems \ref{t1.2} and \ref{t1.3} to the set $E_x = E - x$, 
which is sliding minimal with a boundary condition coming from $L_x = L - x$. 
We take $U = \R^n$, $R_0 = 0$, and $R_1 = +\infty$ (or arbitrarily large).
Theorem \ref{t1.2} says that the function $F$ defined by \eqref{1.10} is nondecreasing.
But 
\begin{equation} \label{10.6}
\lim_{r \to 0} F(r) = \lim_{r \to 0} r^{-d} \H^d(E_x\cap B(0,r))
= \lim_{r \to 0} r^{-d} \H^d(E\cap B(x,r)) = \omega_d
\end{equation}
by \eqref{1.10}, because $\dist(0,L_x) = \dist(x,L) > 0$, and by \eqref{10.4},
while 
\begin{eqnarray} \label{10.7}
\lim_{r \to +\infty} F(r) 
&=& \lim_{r \to +\infty} r^{-d} [\H^d(S \cap B(0,r)) + \H^d(E_x\cap B(0,r))]
\nn\\
&=& {\omega_d \over 2}  +\lim_{r \to 0} r^{-d} \H^d(E_x\cap B(0,r)) 
\nn\\
&\leq& {\omega_d \over 2}  +\lim_{r \to 0} r^{-d} \H^d(E\cap B(0,r+|x|))
\leq \omega_d
\end{eqnarray}
by \eqref{1.10}, where $S$ is the shade of $L_x$ as in \eqref{1.9}, 
because $B(x,r) \i B(0,r+|x|)$, and by assumption.

So $F$ is constant on $(0,+\infty)$, and Theorem \ref{t1.3} applies to any interval.
Notice that $\dist(0,L_x) >0$; so we get a coral minimal cone $X$ (no boundary condition)
such that $X\sm S \subset E_x$ (as in \eqref{1.14}). Then
\begin{eqnarray} \label{10.8}
d(X) &=& \lim_{r \to +\infty}  r^{-d} \H^d(X\cap B(0,r))
\nn\\
&\leq& \liminf_{r \to +\infty}  r^{-d} [\H^d(E_x\cap B(0,r))+\H^d(S\cap B(0,r))]
= \lim_{r \to +\infty} F(r) = \omega_d, 
\end{eqnarray}
and by \eqref{10.5}, $X$ is a $d$-plane. 
Notice that if $m < d-1$ (the dimension of $L$),
the proof of \eqref{10.8} even gives that $d(X) \leq {\omega_d \over 2}$, 
which is impossible.

In addition, \eqref{1.15} tells us that $\H^d(S \sm X) = 0$, which implies
that $S \subset X$ (because $X$ is closed and $S$ is $d$-dimensional). So
$X$ is the $d$-plane that contains $S$. We know from \eqref{1.14} that
$X \sm S \subset E_x$, and the definition \eqref{1.13} says that $E_x \i X$.
Since in addition $\H^d(E_x \cap S) = 0$ by \eqref{1.12}, we get that
$E_x = (X \sm S) \cup L_x$, and $E$ is a half plane bounded by $L$,
as announced.
\end{proof}

For Section \ref{S12} we will need to restrict to dimensions $n$ and $d$ such that
\begin{equation} \label{10.9}
\text{$d(X) > {3 \omega_d \over 2}$ 
for $X\in MC(n,d) \sm [\bP_0(n,d) \cup \bY_0(n,d)]$},
\end{equation}
i.e., when $X\in MC(n,d)$ is neither a vector $d$-plane nor a cone of type $\bY$
(see the definitions above \eqref{1.28}).

The author does not know for which values of $n$ and $d$ this assumption
is satisfied. When $d=1$, \eqref{10.9} holds trivially because 
$MC(n,1) = \bP_0 \cup \bY_0$. When $d=2$ and $n=3$,
it follows from the explicit description of $MC(3,2)$ as the union of
$\bP_0$, $\bY_0$, and the cones of type $\bT$.
When $d=2$ and $n > 3$ we also get in Proposition 14.1 of \cite{Holder} a 
description of the minimal cones that implies \eqref{10.9}.
In all these cases, there is even a constant $d_{n,d} >  {3 \omega_d \over 2}$
such that 
\begin{equation} \label{10.10}
d(X) \geq d_{n,d} \ \text{ when } X \in MC(n,d)\sm (\bP_0 \cup \bY_0).
\end{equation}
See Lemma 14.2 in \cite{Holder} for the last case when $n > 3$.

Finally, we claim that \eqref{10.9} probably holds when $d=n-1$ and $n \leq 6$.
The proof is written in Lemmas 2.2 and 2.3 of \cite{Luu1},
in the special case of $d=3, n=4$, and Luu uses a result of
Almgren \cite{Almgren1} that says that the only minimal cones of dimension $3$
in $\R^4$ whose restriction to the unit sphere are smooth hypersurfaces are 
the $3$-planes. The same proof should work in codimension 1 when $n \leq 6$,
with a dimension reduction argument, and starting
from the generalization of Almgren's result by Simons \cite{Simons}.
But even when $d=3$ and $n=4$, it does not seem to be known whether
the analogue of \eqref{10.10} holds for some constant 
$d_{n,d} > {3 \omega_d \over 2}$.

Let us check that the assumption \eqref{10.9} implies an apparently slightly 
stronger one.

\ms
\begin{lem} \label{t10.3}
If $n$ and $d$ are such that \eqref{10.9} holds, then for 
for each small $\tau > 0$ we can find $\eta > 0$ such that if 
$X \in MC(n,d)$  is such that $d(X) \leq {3 \omega_d \over 2} + \eta$,
then either $X \in \bP_0$ or else there is a cone $Y \in \bY_0$ such that 
$d_{0,1}(X,Y) \leq \tau$ (where $d_{0,1}$ is as in \eqref{1.29}).
\end{lem}

\begin{proof}
The standard proof is the same as for the first part of Lemma \ref{t10.1}.
We suppose that this fails for some $\tau > 0$, and take a sequence $\{ X_k \}$ 
in $MC(n,d)$ such that $d(X_k) \leq {3 \omega_d \over 2} +2^{-k}$, 
but $X_k \notin \bP_0$ and there is no cone $Y\in \bY_0$ as in the statement.
By \eqref{10.9}, $d(X_k) \geq {3 \omega_d \over 2}$.

Then extract a subsequence that converges to a limit $X_\infty$.
Observe that $X_\infty \in MC(n,d)$ by Theorem 4.1 in \cite{limits}, that
$X_\infty$ is a coral cone (as a limit of coral cones), and
that $d(X_\infty) = \lim_{k \to + \infty}d(X_k) = {3 \omega_d \over 2}$ 
by Theorem 3.4 in \cite{limits} (for the lowersemicontinuity inequality) 
and Lemma 3.12 in \cite{Holder} (for the uppersemicontinuity, which 
unfortunately was not already included in \cite{limits}). 
By \eqref{10.9}, $X_\infty \in \bY_0$, this contradicts the definition of
the $X_k$ or the fact that they tend to $X_\infty$, and this proves the lemma.
\end{proof}

\section{Sliding almost minimal sets that look like a half plane}
\label{S11}

In this section we use the main results of the previous sections to prove 
Corollary \ref{t1.7}.

Let $L$ and $E$ be as in the statement. In particular,
$E$ is a coral sliding almost minimal set in $B(0,3)$, associated to
a unique boundary piece $L$, which is a $(d-1)$-plane through
the origin. We may choose the type of almost minimality as we wish
(that is, $A$, $A'$, or $A_+$ in Definition \ref{t2.1}).

We assume that $h$ is sufficiently small, as in \eqref{1.30}, and that
$E$ is sufficiently close in $B(0,3)$ to a half-plane $H \in \bH(L)$, 
as in \eqref{1.31}, and we want to approximate $E$ by planes and
half planes, in the Reifenberg way.

Recall that $\bH(L)$ is the set of $d$-dimensional half planes bounded by
$L$, and let $\bP$ be the set of (affine) $d$-planes.
The proof below will also follow known tracks. See for instance Section 16
of \cite{Holder}.
We first check that $E$ does not have too much mass in a slightly smaller ball.

\begin{lem} \label{t11.1}
Set $r_1 = {28 \over 10}$ and $B_1 = B(0,r_1)$. 
There is a constant $C \geq 0$, that depends only on $n$ and $d$,
such that 
\begin{equation} \label{11.1}
\H^d(E \cap B) \leq \H^d(H \cap B) + C \sqrt \varepsilon
\end{equation}
for each ball $B$ centered on $E$ and such that $\overline B \i B_1$.
\end{lem}

\begin{proof}
We are shall use the local Ahlfors regularity of $E$
(see \eqref{2.9}). First observe that
\begin{equation} \label{11.2}
h(r_1) \leq C \int_{r_1}^3 {h(t) dt \over t} \leq C \varepsilon 
\end{equation}
because $h$ is nondecreasing and by \eqref{1.30}. 
Then for $x\in E \cap B_1$ and $\rho \leq 10^{-2}$,
the pair $(x,\rho)$ satisfies \eqref{2.10} if $\varepsilon$
is small enough, hence by \eqref{2.9}
\begin{equation} \label{11.3}
\H^d(E \cap B(y,\rho)) \leq C \rho^d.
\end{equation}
Here and in the next lines, $C$ is a constant that depends only on $n$ and $d$.

Next let $B = B(x,r)$ be as in the statement.
If $r^d \leq \sqrt \varepsilon$, then  $r \leq 10^{-2}$ too
(if $\varepsilon$ is small enough), and \eqref{11.1}
follows brutally from \eqref{11.3}. So let us assume that 
$r^d \geq \sqrt \varepsilon$.
Define a cut-off function $\alpha$ by
\begin{equation} \label{11.4}
\begin{aligned}
\alpha(y) &= 0 \ \text{ when } |y-x| \geq r, 
\\
\alpha(y) &= 1 \ \text{ when } |y-x| \leq r - 6\varepsilon, 
\\
\alpha(y) &=  (6\varepsilon)^{-1}(r - |y-x|)
\ \text{ otherwise.} 
\end{aligned}
\end{equation}
Denote by $\pi$ the smallest distance projection on the convex
set $H$; notice that $\pi$ is $1$-Lipschitz. We set
\begin{equation} \label{11.5}
\varphi(y) = \alpha(y) \pi(y) + (1-\alpha(y)) y
\end{equation}
for $y\in \R^n$, and then $\varphi_t(y) = t \varphi(y) + (1-t) y$
for $0 \leq t \leq 1$. Notice that the $\varphi_t$ define an acceptable
deformation (see near \eqref{1.2}), in particular because $\pi(y) = y$ on
$L \i H$, and that $\wh W$ is compactly contained in $B_1$,
by \eqref{11.4} and because $\overline B \i B_1$,
$\pi(y) \in B_1$ when $y\in B_1$, and $B_1$ is convex. 
Let us assume for the moment that $E$ is an almost minimal set of type $A'$.
We deduce from \eqref{2.6} that 
\begin{equation} \label{11.6}
\H^d(E \sm \varphi(E)) \leq \H^d(\varphi(E) \sm E) + r_1^d h(r_1) 
\leq \H^d(\varphi(E) \sm E) + C \varepsilon.
\end{equation}
Notice that if $x\in E \cap B$, then either $x \in E \sm \varphi(E)$,
or else $x\in E \cap \varphi(E)$. In the last case, $x\in \varphi(E\cap B)$,
because $\varphi(y) = y$ on $\R^n \sm B$. Thus
\begin{equation} \label{11.7}
\H^d(E \cap B) \leq \H^d(E \sm \varphi(E)) + \H^d(E \cap \varphi(E\cap B)).
\end{equation}
By \eqref{11.6} this yields
\begin{equation} \label{11.8}
\H^d(E \cap B) 
\leq  \H^d(\varphi(E) \sm E) + \H^d(E \cap \varphi(E\cap B)) + C \varepsilon
= \H^d(\varphi(E\cap B)) + C\varepsilon,
\end{equation}
where the last part holds because $\varphi(E) \sm E = \varphi(E\cap B) \sm E$
(since $\varphi(y) = y$ when $y\in \R^n \sm B$).
Set $B' = B(x, r-6\varepsilon)$, and let us check that
\begin{equation} \label{11.9}
\varphi(y) \in H \cap B \ \text{ for } y \in E \cap B'.
\end{equation}
First recall from \eqref{1.31} and the definition \eqref{1.29} that
\begin{equation} \label{11.10}
\dist(y,H) \leq 3 \varepsilon \ \text{ for } y \in E \cap B(0,3),
\end{equation}
hence
\begin{equation} \label{11.11}
|\pi(y) - y| \leq 6 \varepsilon \ \text{ for } y\in E \cap B(0,3)
\end{equation}
(pick $z\in H$ such that $|z-y| \leq 3\varepsilon$, then use the fact
that $\pi$ is $1$-Lipschitz and $\pi(z)=z$). If furthermore
$y \in E \cap B'$, then $\pi(y) \in H \cap B$ and,
since $\varphi(x) = \pi(x)$ by \eqref{11.5}, we get \eqref{11.9}.
This takes care of the major part of $\varphi(E \cap B)$.

We are left with the contribution of $A = B \sm B'$. 
Let $P_0$ denote the $d$-plane that contains $H$,
and let $P$ denote the $d$-plane through $x$ parallel
to $P$. By \eqref{11.10}, $\dist(P,P_0) \leq 3\varepsilon$,
and every point of $E\cap A$ lies within $3 \varepsilon$ of $P_0$,
hence within $6\varepsilon$ of $P$ and within $12\varepsilon$
of $P \cap \d B$. If we cover $P \cap \d B$ by balls $B_j$
of radius $20 \varepsilon$, the double balls $2B_j$ will then cover $E \cap A$.
We can do this with less than $C (r/\varepsilon)^{d-1}$ balls $B_j$,
and $\H^d(E \cap 2B_j) \leq C \varepsilon^d$ by \eqref{2.9}
(maybe applied to a twice larger ball centered on $E$, if $2B_j$ meets
$E$ but is not centered on $E$). Thus
\begin{equation} \label{11.12}
\H^d(E \cap A) \leq C (r/\varepsilon)^{d-1} \varepsilon^d 
\leq C \varepsilon r^{d-1}.
\end{equation}
Also, we claim that $\varphi$ is $C$-Lipschitz on $E \cap A$.
This is because $\alpha$ is $(6\varepsilon)^{-1}$-Lipschitz,
but $|\pi(y)-y| \leq 6 \varepsilon$ on $E \cap A$; the verification
is the same as for \eqref{4.11}, so we skip it. Thus
\begin{equation} \label{11.13}
\H^d(\varphi(E \cap A)) \leq C \H^d(E \cap A) \leq C \varepsilon r^{d-1}.
\end{equation}
We put things together and get that
\begin{eqnarray} \label{11.14}
\H^d(E \cap B) &\leq& \H^d(\varphi(E\cap B)) + C\varepsilon
\leq \H^d(\varphi(E\cap B')) +\H^d(\varphi(E \cap A)) + C\varepsilon
\nn\\
&\leq& \H^d(H \cap B) +  C \varepsilon r^{d-1} + C\varepsilon
\end{eqnarray}
by \eqref{11.8}, \eqref{11.9}, and \eqref{11.13}.
But we are in the case when $r^d \geq \sqrt \varepsilon$,
so $\sqrt \varepsilon r^{d-1} \leq r^{2d-1} \leq C r^d$ and 
\eqref{11.1} follows from \eqref{11.14}.

We still need to say how we proceed when $E$ is of type $A$
or $A_+$. Of course we could say that in both cases, $E$ is also
an $A'$-almost minimal set, but the long proof can be avoided with
a small trick. Choose a possibly different $d$-dimensional half
plane $H_1$, with the same boundary $L$ as $H$, so that
\begin{equation} \label{11.15}
\dist(y,H_1) \leq 4 \varepsilon \ \text{ for } y \in E \cap B(0,3).
\end{equation}
We have uncountably many choices of $H_1$, all disjoint except
for the set $L$ of vanishing $H^d$-measure, so we can choose
$H_1$ so that $\H^d(E \cap B_1 \cap H_1) = 0$.
Then we repeat the construction above, with $3\varepsilon$ replaced
by $4\varepsilon$. Since the points of $B'$ are now sent to $H_1$,
we get that $\H^d$-almost every point of $E \cap B'$ lies in 
$W_1 = \big\{ x\in E \, ; \, \varphi(x) \neq x \big\}$.
Suppose that $E$ is of type $A$; then
\begin{equation} \label{11.16}
\H^d(E \cap B') \leq \H^d(W_1) \leq \H^d(\varphi(W_1)) + h(r_1) r_1^d
\leq \H^d(\varphi(W_1)) + C \varepsilon
\end{equation}
by the discussion above and \eqref{2.5}, and in turn
\begin{eqnarray} \label{11.17}
\H^d(\varphi(W_1)) &\leq& \H^d(\varphi(E \cap B))
\leq \H^d(\varphi(E \cap B')) + \H^d(\varphi(E \cap A))
\nn\\
&\leq& \H^d(H \cap B) + C \varepsilon r^{d-1}
\end{eqnarray}
because $\varphi(y) = y$ on $\R^n \sm B$,
and by \eqref{11.9} and \eqref{11.13}. Therefore
$\H^d(E \cap B') \leq \H^d(H \cap B) + C \varepsilon r^{d-1} + C \varepsilon$,
and we can conclude as before.

The case when $E$ is of type $A_+$ is as simple; just observe
that the error term in \eqref{11.16} is replaced by
$h(r_1) \H^d(W_1) \leq h(r_1) \H^d(E \cap B) \leq C \varepsilon$.
\end{proof}

For each $x\in E\cap B(0,2)$, we define a shade $S_x$ and a functional
$F_x$ as we did in the introduction, but with the center $x$. That is,
\begin{equation} \label{11.18}
S_x = \big\{ y \in \R^n \, ; \, x + \lambda (y-x) \in L 
\text{ for some } \lambda \in [0,1] \big\}
\end{equation}
(see \eqref{1.9} and \eqref{7.12}), then
\begin{equation} \label{11.19}
\theta_x(r) = r^{-d} \H^d(E \cap B(x,r))
\end{equation}
and 
\begin{equation} \label{11.20}
F_x(r) = \theta_x(r) + r^{-d} \H^d(S_x \cap B(x,r))
\end{equation}
for $r > 0$ (compare with \eqref{1.10}).
Here and below, densities will often be compared to the $\H^d$-measure
of the unit $d$-disk, i.e.,
\begin{equation} \label{11.21}
\omega_d = \H^d(\R^d \cap B(0,1)).
\end{equation}

\ms
\begin{lem} \label{t11.2} Set 
\begin{equation} \label{11.22}
r_2 = {7 \over 10} \ \text{ and } \ B_2 = B(0,21/10).
\end{equation}
There is a constant $C \geq 0$, that depends only on $n$ and $d$,
such that 
\begin{equation} \label{11.23}
F_x(r_2) \leq \omega_d + C \sqrt \varepsilon
\ \text{ for } x\in E \cap B_2.
\end{equation}
\end{lem}

\begin{proof}
First assume that $\dist(x,L) \leq \sqrt\varepsilon$
and choose $z\in L$ such that $|z-x| \leq \sqrt{\varepsilon}$. Then
\begin{eqnarray} \label{11.24}
H^d(E \cap B(x,r_2)) &\leq& 
H^d(H \cap B(x,r_2)) + C \sqrt \varepsilon
\nn\\
&\leq&
\H^d(H \cap B(z,r_2+\sqrt\varepsilon)) + C \sqrt \varepsilon
\nn\\
&\leq& {\omega_d (r_2+\sqrt\varepsilon)^d \over 2} + C \sqrt \varepsilon
\leq {\omega_d r_2^d \over 2} + C \sqrt \varepsilon
\end{eqnarray}
by Lemma \ref{t11.1} and because $B(x,r_2) \i B(z,r_2+\sqrt\varepsilon)$.
Since $\H^d(S_x \cap B(x,r_2)) \leq {\omega_d r_2^d \over 2}$
because $S_x$ is contained in a half plane centered at $x$, we
add and get \eqref{11.23}.

If instead $\dist(x,L) \geq \sqrt\varepsilon$, \eqref{11.10} says that 
$\dist(x,H) \leq 3\varepsilon$, hence the two half planes
$H$ and $S_x$ make an angle $\alpha \geq \pi - C\sqrt \varepsilon$
along $L$. Let $P_x$ denote the $d$-plane that contains $x$, $L$, and
hence $S_x$, and let $\pi$ be the orthogonal projection on $P_x$;
since $\pi(B(x,r_2)) \i B(x,r_2)$, we get that
$\pi(H \cap B(x,r_2)) \i P_x \cap B(x,r_2)$. Also,
$\pi(H \cap B(x,r_2)) \cap S_x \i L$ (because of the small angle),
and hence
\begin{eqnarray} \label{11.25}
r_2^d F_x(r_2) &=& \H^d(S_x \cap B(x,r_2)) + \H^d(E \cap B(x,r_2))
\nn\\
&\leq& \H^d(S_x \cap B(x,r_2)) + \H^d(H \cap B(x,r_2)) + C \sqrt \varepsilon
\nn\\
&\leq& \H^d(S_x \cap B(x,r_2)) + {1 \over \cos(\pi-\alpha)}\H^d(\pi(H \cap B(x,r_2)))
+ C \sqrt \varepsilon
\\
&\leq& \H^d(S_x \cap B(x,r_2)) + \H^d(\pi(H \cap B(x,r_2))) 
+ C \sqrt\varepsilon
\nn\\
&=& \H^d(P_x \cap B(x,r_2)) + C \sqrt\varepsilon
\leq \omega_d r_2^d + C \sqrt \varepsilon
\nn
\end{eqnarray}
by Lemma \ref{t11.1} and because $\cos(\pi-\alpha) \geq 1-C\varepsilon$;  
Lemma \ref{t11.2} follows.
\end{proof}

Let $a > 0$ denote the constant in Theorem \ref{t1.5},
and choose $\varepsilon$ so small that \eqref{11.2} implies that
$h(r_2) \leq h(r_1) \leq \tau$, where $\tau$ is the small constant
in Theorem \ref{t1.5}. We deduce from Theorem \ref{t1.5}
(applied to a translation of $E$ by $-x$) that for $x\in E \cap B_2$
and $0 < r < s \leq r_2$,
\begin{equation} \label{11.26}
F_x(r) \leq e^{a(A(s)-A(r))} F_x(s) \leq e^{a\varepsilon} F_x(s)
\leq e^{2a\varepsilon} F_x(r_2) \leq \omega_d + C \sqrt \varepsilon
\end{equation}
by \eqref{1.21}, \eqref{1.20}, \eqref{11.2} or \eqref{1.30},
a second application of the same inequalities, and \eqref{11.23}.
Because of the first inequality (or directly \eqref{1.21}), 
there exists a limit
\begin{equation} \label{11.27}
F_x(0) = \lim_{r \to 0} F_x(r),
\end{equation}
and \eqref{11.26} implies that
\begin{equation} \label{11.28}
F_x(0) e^{-a\varepsilon}\leq F_x(s) \leq \omega_d + C \sqrt \varepsilon
\ \text{ for } 0 < s \leq r_2.
\end{equation}

Let us restrict our attention to $x\in E \sm L$. Then
\begin{equation} \label{11.29}
F_x(r) = \theta_x(r) \ \text{ for } 0 < r \leq \dist(x,L)
\end{equation}
by \eqref{11.20}, and there exists
\begin{equation} \label{11.30}
\theta_x(0) = \lim_{r \to 0} \theta_x(r) = F_x(0).
\end{equation}
In fact, we already knew this, just from the almost monotonicity
of $\theta_x$ for almost minimal sets with no sliding condition
(see for instance Proposition 5.24 in \cite{Holder}).
We claim that for $H^d$-almost every $x\in E \cap B_2$,
\begin{equation} \label{11.31}
 F_x(0) = \theta_x(0) = \omega_d.
\end{equation}
Since $H^d(L) = 0$, we may restrict to $x\in E \sm L$.
But we know that $E \sm L$ is rectifiable; since we are far from
the boundary set $L$, we can even use the result of 
Almgren \cite{AlmgrenMemoir} instead of \eqref{2.11}). 
Now \eqref{11.31} follows directly from this and known density 
properties of rectifiable sets (see for instance Theorem 16.2 in \cite{Mattila}) 
or because $E$ has an approximate tangent $d$-plane at almost every point 
(by \eqref{2.9}, we could even say that this is a real tangent plane); indeed,
at such a point, every blow-up limit of $E$ is a plane, and then
\eqref{11.31} follows for instance from Proposition 7.31 in \cite{Holder}.
From \eqref{11.28} and \eqref{11.31} (when it holds), we deduce that
\begin{equation} \label{11.32}
\omega_d e^{-a\varepsilon}\leq F_x(s) \leq \omega_d + C \sqrt \varepsilon
\ \text{ for } 0 < s \leq r_2\, ;
\end{equation}
that is, $F_x$ is nearly constant on $(0,r_2]$, and this will allow us to
show that $E$ look like a minimal cone at the corresponding scales.

\begin{lem} \label{t11.3}
For each $\tau > 0$, there is a constant $\varepsilon_0 > 0$, that
depends only on $n$, $d$, and $\tau$, such that following property holds as soon 
as $L$ and $E$ are as above and $\varepsilon < \varepsilon_0$.
Let $x\in E \cap B_2\sm L$ be such that \eqref{11.31} holds, 
and let $r \in (0,r_2)$ be given. 

If $0 < r < \dist(x,L)$, there a plane $X = X(x,r)$ through $x$ such that
\begin{equation} \label{11.33}
d_{x,3r/4}(E, X) \leq 4\tau/3
\end{equation}
and
\begin{eqnarray} \label{11.34}
&\,&\big|\H^d(E \cap B(y,t)) - \H^d(X \cap B(y,t))\big| \leq \tau r^d
\nn\\
&\,&\hskip2cm
\text{ for $y\in \R^n$ and $t>0$ such that } B(y,t) \i B(x,(1-\tau)r).
\end{eqnarray}

If ${1\over 2}\, \dist(x,L) < r < \tau^{-1} \dist(x,L)$, let $X$
denote the affine $d$-plane that contains $x$ and $L$, and set
$X' = \overline{X \sm S_x}$. Then
\begin{equation} \label{11.35}
d_{x,3r/4}(E, X') \leq \tau
\end{equation}
and
\begin{eqnarray} \label{11.36}
&\,&\big|\H^d(E \cap B(y,t)) - \H^d(X' \cap B(y,t))\big| \leq \tau r^d
\nn\\
&\,&\hskip2cm
\text{ for $y\in \R^n$ and $t>0$ such that } B(y,t) \i B(x,(1-\tau)r).
\end{eqnarray}
\end{lem}

\begin{proof}
See \eqref{1.29} for the definition of $d_{x,r}$.
Let $\tau$, $x$ and $r$ be as in the statement, and set $B = B(x,r)$.
We start with the case when $r < \dist(x,L)$; then $E$ is a (plain)
almost minimal set in $B$, and we want to apply Proposition 7.24 in \cite{Holder}
to the set $E$ in the ball $B$. The fact that $E$ is almost minimal in $B$
follows at once from our assumptions (since $L \cap B = \emptyset$),
we can use the same gauge function $h$ as here, and the assumption
that $h(2r) \leq \varepsilon$ (for the small $\varepsilon$ of \cite{Holder})
follows from \eqref{1.30} if $\varepsilon_0$ is small enough. Then there is
the assumption that
\begin{equation} \label{11.37}
\theta_x(r) \leq \inf_{0 < t < r/100} \theta_x(t) + \varepsilon,
\end{equation}
where again $\varepsilon$ comes from Proposition 7.24 in \cite{Holder},
with the same $\tau$ as here. This follows from \eqref{11.32}, 
because $F_x(t) = \theta_x(t)$ for $t \leq r < \dist(x,L)$, and
if $\varepsilon_0$ is small enough. 
Then Proposition 7.24 in \cite{Holder} says that there is a minimal
cone $X$, centered at $x$, that satisfies our two conditions \eqref{11.33}
and \eqref{11.34}. In addition, \eqref{11.34} with the ball
$B(x,r/2)$ yields 
\begin{eqnarray} \label{11.38}
\big|\H^d(X \cap B(x,1))-\omega_d \big|
&=& \big({r\over2}\big)^{-d} \big|\H^d\big(X \cap B\big(x,{r\over2}\big)\big)
- \big({r\over2}\big)^{d} \omega_d  \big|
\nn\\
&\leq& \big({r\over2}\big)^{-d} 
\big|\H^d\big(X \cap B\big(x,{r\over2}\big)\big)
- \H^d\big(E \cap B\big(x,{r\over2}\big)\big)\big|
+ |\theta_x\big({r\over2}\big)-\omega_d|
\nn\\
&\leq& 2^d \tau + C \sqrt{\varepsilon}
\end{eqnarray}
by \eqref{11.32}. We may assume that $\tau$ was chosen smaller
than ${1 \over 2} (d(n,d)-\omega_d)$, where 
$d(n,d)$ is as in Lemma \ref{t10.1}; then, if $\varepsilon_0$ is small enough,
Lemma \ref{t10.1} and \eqref{11.38} imply that $X$ is a plane, and 
this completes our proof when $r < \dist(x,L)$.

\ms
Next we assume that ${10 \over 9} \dist(x,L) < r < \tau^{-1} \dist(x,L)$.
In this case we want to apply Corollary \ref{t9.3} to $E_x = r^{-1}(E - x)$,
$L_x = r^{-1}(E - x)$, and some small constant $\tau_1 < \tau$ 
that will be chosen soon. Set $t = \dist(0,L_x) = r^{-1} \dist(x,L)$;
our assumption says that
\begin{equation} \label{11.39}
\tau \leq t \leq {9 \over 10},
\end{equation}
and so the first assumption \eqref{9.26} is satisfied as soon as $\tau_1 \leq \tau$.
Here $L_x$ is a $(d-1)$-plane, which takes care of \eqref{9.27}, 
and \eqref{9.28} holds (if $\varepsilon_0$ is small enough), by 
\eqref{1.30} and \eqref{11.32}.
Thus we get a minimal cone $X_0$, which satisfies 
\eqref{9.29}-\eqref{9.34} with the constant $\tau_1$
(and with respect to $E_x$), and with $L' = L_x$
(by the comment after the statement of Corollary \ref{t9.3}).

By \eqref{11.39} we can apply \eqref{9.34} to $B(0,t)$; 
since $E_0 = X_0$ on $B(0,t)$  we get that
\begin{eqnarray} \label{11.40}
\H^d(X_0 \cap B(0,t)) &=& \H^d(E_0 \cap B(0,t))
\leq \H^d(E_x \cap B(0,t)) + \tau_1
\nn\\
&=& r^{-d}  \H^d(E \cap B(x,tr)) + \tau_1
= t^d \theta_x(tr) + \tau_1
\nn\\
&=& t^d F_x(tr) + \tau_1 \leq t^d \omega_d + C t^d \sqrt{\varepsilon_0} + \tau_1
\end{eqnarray}
because $E_x = r^{-1}(E - x)$ and $B(x,tr)$ does not meet $L$,
then by \eqref{11.32}. With the notation \eqref{10.1},
\begin{equation} \label{11.41}
d(X_0) = \H^d(X_0 \cap B(0,1)) \leq \omega_d + \sqrt{\varepsilon_0}
+ t^{-d}\tau_1;
\end{equation}
since $t \geq \tau$ by \eqref{11.39}, we see that if $\tau_1$ is
chosen small enough, and $\varepsilon_0$ is small enough, \eqref{11.41}
and lemma \ref{t10.1} imply that $X_0$ is a plane.

By \eqref{9.29} and because $L' = L_x$ really meets $B(0,99/100)$
(by \eqref{11.39} again), $X_0$ contains $L_x$.
That is, $X_0$ is the only vector $d$-plane that contains $L_x$,
and $X = x + r X_0 = x+X_0$ is the affine $d$-plane that contains
$x$ and $L$. We are now ready co conclude; 
the set $X' = \overline{X\sm S_x}$ is the same as
$x+ r \overline{X_0\sm S} = x+ r E_0$, where $E_0$ is
as in Corollary \ref{t9.3}, so \eqref{11.35} follows from \eqref{9.32} 
and \eqref{9.33}, and \eqref{11.36} follows from \eqref{9.34}
(if $\tau_1$ is small enough, as before).

\ms
We are left with the intermediate case when 
${1\over 2} \dist(x,L) \leq r \leq {10 \over 9} \dist(x,L)$.
If $3r \leq r_2$, we simply observe that we can apply the proof above to the 
radius $3r$, and with  choices of $\tau$ and $\tau_1$, and get the desired result for
$r$. When $r \geq r_2/3$ but $r \leq {10 \over 9} \dist(x,L)$, we deduce
\eqref{11.35} directly from \eqref{1.31}, because since $x$ itself lies
close to $H$ and far from $L$, $H$ is quite close to $X'$ at the unit scale.
As for \eqref{11.36}, we can easily deduce it from the same result
with $X'$ replaced by $H$, and this last is itself easy to deduce
from \eqref{1.31} and the same compactness argument as for 
Lemma \ref{t9.2} (also see near Lemma \ref{t9.5}). We skip the details; anyway,
the case when $r \geq r_2/3$ is far from being the most interesting.
This completes our proof of Lemma \ref{t11.3}.
\end{proof}

\ms
Let us now check \eqref{1.32}.
Let $z\in L \cap B(0,2)$ be given, suppose that $z \notin E$, and
set $d = \dist(z,E) > 0$. By \eqref{11.10}, $d\leq 3\varepsilon$.
Let $x\in E$ be such that $|x-z| = d \leq 3\varepsilon$. Then $x\in B_2$,
and we can apply Lemma \ref{t11.3} with $r = 3d$.
We get that if $X$ denotes the plane through $x$ that contains $L$,
then $X' = \overline{X \sm S_x}$ is very close to $E$ in $B(x,2r/3) = B(x,2d)$
(see \eqref{11.34}). But $X'$ meets $B(z,d/20)$, and all the points
of $X' \cap B(z,d/20)$ lie in $B(x,2d)$, hence very close to $E$;
thus $E$ meets $B(z,d/10)$; this contradiction with the definition
of $d$ proves \eqref{1.32}.

Next we check that Lemma \ref{t11.3} gives the existence of the desired
sets $Z= Z(x,r)$, $x\in E\cap B(0,2)$, as long as $r \leq \tau^{-1}\dist(x,L)$
and $x$ satisfies \eqref{11.31}.
When $r \leq \dist(x,L)/2$, we take $Z = X(x,4r/3)$, where the $d$-plane 
$X(x,4r/3)$ is obtained by applying Lemma \ref{t11.3}, 
with the slightly smaller constant $3\tau/4$; observe that
$4r/3 \leq r_2 = 7/10$ when $r \leq 1/2$ (see \eqref{11.22}).
Then \eqref{1.33} holds by definition and \eqref{1.36} and \eqref{1.37}
follow from \eqref{11.33} and \eqref{11.34}. 

When $\dist(x,L)/2 \leq r < \tau^{-1} \dist(x,L)$, we apply the
second case of Lemma \ref{t11.3} to the pair $(x,4r/3)$ and with
the constant $\tau/4$. We get that $E$ is well approximated
in $B(x,r)$ by $X' = \overline{X \sm S_x}$, where $X$ is the
$d$-plane that contains $x$ and $L$; thus $X'$ is the same half plane as 
$Z(x,r)$ in \eqref{1.34}, and \eqref{1.36} and \eqref{1.37}
follow from \eqref{11.35} and \eqref{11.36}.

In fact, by applying Lemma \ref{t11.3} with a much smaller $\tau_1$,
we even get the desired set $Z(x,r)\in \bH(L)$ for 
$\tau^{-1} \dist(x,L) \leq r \leq (2\tau_1)^{-1}\dist(x,L)$; we even get
a better approximation and the extra information that $Z(x,r)$
goes through $x$ (which was not required in \eqref{1.35}).

Also, our constraint that $x$ satisfy \eqref{11.31} can be lifted,
because if $x\in E \cap B(0,2)\sm L$ does not satisfy \eqref{11.31},
we can write $x$ as a limit of points $x_k \in E \cap B(0,2)$ that
satisfy \eqref{11.31} (because \eqref{11.31} holds almost everywhere on $E$), 
apply the result to these $x_k$ and a slightly larger radius $r'$, 
and get the desired $Z(x,r)$ as a minor modification of some $Z(x_k,r')$.

\ms
We are thus left with the case when 
\begin{equation} \label{11.42}
(2\tau_1)^{-1}\dist(x,L) \leq r \leq 1/2,
\end{equation}
where $\tau_1$ is as small as we want (we shall choose it soon,
depending on $\tau$). As before, it is enough to find
$Z(x,r)$ when $x\in E \sm L$ and $x$ satisfies \eqref{11.31},
which is useful because then \eqref{11.32} holds.

Since we declined to prove an analogue for Corollary \ref{t9.3}
for points that lie too close to the boundary $L$, 
we'll have to apply Proposition 30.3 in \cite{Sliding},
which provides a similar result for balls centered on the boundary.

Let $z$ denote the point of $L$ that lies closest to $x$; thus
\begin{equation} \label{11.43}
|z-x| = \dist(x,L) \leq 2\tau_1 r.
\end{equation}
We want to use \eqref{11.32} to find good bounds on the density $\theta_{z}$.
We start with the radius $r_3 = \sqrt{\tau_1} r$; set
$r'_3 = r_3 - |z-x|$ and notice that
\begin{eqnarray} \label{11.44}
\theta_z(r_3) &=& r_3^{-d} \H^d(E\cap B(z,r_3))
\geq r_3^{-d}  \H^d(E\cap B(x,r'_3))
= (r'_3/r_3)^d \,\theta_x(r'_3)
\nn\\
&=& (r'_3/r_3)^d F_z(r'_3) - r_3^{-d} \H^d(S_x \cap B(z,r'_3))
\geq (r'_3/r_3)^d F_z(r_3) - {\omega_d \over 2} - C |z-x| r_3^{-1}
\nn\\
&\geq& (1-|z-x| r_3^{-1})^d e^{-a \varepsilon} \omega_d - {\omega_d \over 2}
-C |z-x| r_3^{-1}
\geq {\omega_d \over 2} - C \sqrt{\tau_1} 
\end{eqnarray}
by \eqref{11.20}, \eqref{11.32}, a brutal estimate on $\H^d(S_x \cap B(z,r'_3))$
that uses the fact that $|z-x| = \dist(x,L) \leq 2\tau_1 r= 2\sqrt{\tau_1} r_3$, 
and if $\varepsilon$ is small enough.

We also want an upper bound on $\theta_z(4r/3)$. Set $r_4 = 4r/3$ and
$r'_4 = r_4 + |z-x| < r_2$; then
\begin{eqnarray} \label{11.45}
\theta_z(r_4) &=& r_4^{-d} \H^d(E\cap B(z,r_4))
\leq r_4^{-d} \H^d(E\cap B(x,r'_4))
= (r'_4/r_4)^d \,\theta_x(r')
\nn\\
&=& (r'_4/r_4)^d \big[F_z(r'_1) - (r'_4)^{-d} \H^d(S_x \cap B(z,r'_4))\big]
\nn\\
&\leq& (r'_4/r_4)^d \big[ F_z(r'_1) - {\omega_d \over 2} + {C |z-x| \over r'_4}\big]
\\
&\leq& \Big(1+{|z-x| \over r}\big)^d 
\Big(\big[\omega_d + C\sqrt{\varepsilon}\big] 
- {\omega_d \over 2} + {C |z-x| \over r}\Big)
\leq {\omega_d \over 2} + C \tau_1
\nn
\end{eqnarray}
by the same sort of estimates as above, including \eqref{11.32} and
\eqref{11.43}. 

We want to apply Proposition 30.3 in \cite{Sliding}, with 
$\tau$ replaced by a small constant $\eta > 0$ that will be chosen
soon (depending on our $\tau$), $x_0 = z$, the same $r_1$, 
and $r_0 = r_2 = r_4$. The main assumption (30.6)
(the fact that $\theta_z(4r/3)$ is at most barely larger than $\theta_z(r_1)$)
follows from \eqref{11.44} and \eqref{11.45}, if $\tau_1$ is small enough
(depending on $\eta$). The other assumptions are that $L$ be sufficiently close
to a $d$-plane through $z$ (in the bilipschitz sense), and that $E$ is a 
sliding quasiminimal set with small enough constants, and these are satisfied
if $\varepsilon$ is small enough. We get a coral sliding minimal cone $T$ 
centered at $z$, and which is sufficiently close to $E$. In particular,
\begin{equation} \label{11.46}
\dist(y,T) \leq \eta r_4 \ \text{ for } 
y\in E \cap B(z,(1-\eta) r_4) \sm B(z,r_3+\eta r_4)
\end{equation}
and
\begin{equation} \label{11.47}
\dist(y,E) \leq \eta r_4 \ \text{ for } 
y\in T \cap B(z,(1-\eta) r_4) \sm B(z,r_3+\eta r_4).
\end{equation}
Notice that this is not exactly the same as in (30.7) and (30.8) \cite{Sliding},
where $\eta r_4$ is replaced by $\eta$. So in fact we apply Proposition 30.3
to a dilation of $E$ by a factor $r_4^{-1}$, and then we get \eqref{11.46}
and \eqref{11.47}. The dilation does not matter here; we did not do it in
\cite{Sliding} because we were also authorizing boundaries $L_j^0$ that were
not planes, and the Lipschitz assumptions on those have less dilation invariance.

We also get, from (30.10) in the proposition, that
\begin{equation} \label{11.48}
\big| \H^d(E \cap B(z,r)) - \H^d(T \cap B(z,r)) \big| \leq \eta r_4^d.
\end{equation}
Since we also have that $|\theta_z(r) -{\omega_d \over 2}| \leq C\sqrt{\tau_1}$ 
by the proof of \eqref{11.44} and \eqref{11.45}, we get that
\begin{equation} \label{11.49}
|\H^d(T \cap B(z,1)) -  {\omega_d \over 2}|
= r^{-d} |\H^d(T \cap B(z,r)) -  {\omega_d \, r^d\over 2}|
\leq C\sqrt{\tau_1} + 2^d \eta.
\end{equation}

Let $\tau_2 > 0$ be small, to be chosen soon (depending on $\tau$);
we claim that if $\tau_1$ and $\eta$ are small enough, depending on $\tau_2$,
there is a half plane $Z\in \bH(L)$ such that
\begin{equation} \label{11.50}
d_{z,1}(T,Z) \leq \tau_2.
\end{equation}
In fact, we even claim the following apparently stronger result:
there is a constant $\eta_1 > 0$ such that, if
$H$ is a (nonempty) coral sliding minimal cone centered at the origin, 
relative to a boundary which is a vector space $L_0$ of dimension $d-1$,
and if $\H^d(H\cap B(0,1)) \leq {\omega_d \over 2} + \eta_1$,
then there is a half plane $Z\in \bH(L_0)$ such that $d_{0,1}(H,Z) \leq \tau_2$.

Let us prove this (stronger) claim by compactness. 
The proof is essentially the same as for Lemma \ref{t10.1},
to which we refer for details.
By rotation invariance, we may assume that $L_0$ is a given
vector $(d-1)$-plane. If the claim fails, then for each $k \geq 0$ we can find 
a coral sliding minimal cone $H_k$ centered at the origin, with 
$\H^d(H_k\cap B(0,1)) \leq {\omega_d \over 2} + 2^{-k}$,
and which is $\tau_2$-far from all $Z\in \bH(L_0)$.
We extract a subsequence for which $H_k$ tends to a limit
$H_\infty$, we observe that $H_\infty$ is also a coral sliding minimal cone
(by Theorem 4.1 in \cite{Sliding}), and 
$\H^d(H_\infty\cap B(0,1))\leq {\omega_d \over 2}$
by Theorem 3.4 in \cite{Sliding}. Then Corollary \ref{t10.2} says that
$H_\infty \in \bH(L_0)$, which contradicts the definition of the $H_k$
or the fact that they tend to $H_\infty$. This proves our two claims.

Return to $E$. If $\eta$ and $\tau_1$ are small enough, depending on
$\tau_2$, it follows from \eqref{11.46}, \eqref{11.47}, and \eqref{11.50}
that 
\begin{equation} \label{11.51}
d_{z,9r_4/10}(E,Z) \leq 2 \tau_2;
\end{equation}
we do not need to worry about what happens in the hole created
by $B(z,r_3+\eta r_4)$, because $r_3 +\eta r_4 = \sqrt{\tau_1} r +\eta r_4$
is much smaller than $\tau_2 r_4$, and $z$ lies on both sets
$E$ (by \eqref{1.32}) and $Z$.

Clearly $Z$ satisfies \eqref{1.35} and \eqref{11.51} implies 
\eqref{1.36} if $\tau_2 < \tau/2$. We still need to check that
\eqref{1.37} holds, and again it follows from \eqref{11.51} if $\tau_2$ 
is chosen small enough, depending on $\tau$. Otherwise, we could
find a sequence that contradicts Lemma \ref{t9.2}.

Let us summarize. Given $\tau > 0$ and $\tau_1 <<\tau$, 
we used the almost constant
density property \eqref{11.32} and Corollary \ref{t9.3} 
or its simpler variant with no boundary to prove \eqref{1.32} and establish 
the existence of $Z(x,r)$ for $r \leq \tau_1^{-1} \dist(x,L)$ (provided
$\varepsilon \leq \varepsilon_0$ is small enough).
Now we just checked that we can find constants $\tau_2$, then $\eta$
and $\tau_1$, so that if $\varepsilon$ is small enough (depending on these
constants too), we can find $Z(x,r)$ also when $r \geq \tau_1^{-1} \dist(x,L)$.
This completes the proof of Corollary \ref{t1.7}.
\qed

\ms
\begin{rem} \label{t11.4}
The author claims that Corollary \ref{t1.7} can be extended to the case when 
$L$ is a smooth embedded variety of dimension $(d-1)$ through the origin,
which is flat enough in $B(0,1)$. 
\end{rem}

The main ingredients, namely Corollary \ref{t9.3} and 
Proposition 30.3 in \cite{Sliding}, both work in this context (and even
for bilipschitz images with enough control); then it should only be a matter
of checking that the other estimates, such as Lemmas \ref{t11.1} and
\ref{t11.2}, only bring small additional error terms.

But the author was a little to lazy to write down the argument.
An excuse is that probably the right result is stronger, and says that
$E$ is a $C^{1+\alpha}$ version of a half plane locally, with a proof that
looks more like those of \cite{epi} when $d=2$. The main advantage of
Corollary~\ref{t1.7} is that it is relatively simple.

\section{Sliding almost minimal sets that look like a $V$}
\label{S12}

In this section we prove Corollary \ref{t1.8}, in a slightly more general 
setting than stated in the introduction. That is, we replace the precise
assumption that $d=2$ or $d = 3$ and $n=4$ with the assumption
that \eqref{10.9} holds.

We thus consider sliding almost minimal sets with respect to a single
boundary set $L$, which we assume to be a vector space of dimension $d-1$.
As in the introduction, denote by $\bV(L)$ the set of unions $V = H_1 \cup H_2$
of two half $d$-planes $H_i$, both bounded by $L$, and which make an angle
at least equal to $2\pi/3$ along $L$; see above \eqref{1.28}. 
We may call such a $V$ a \underbar{cone of type $\bV$}.
Notice that $V$ is allowed to be a $d$-plane that contain $L$.

In this section we prove the following slight extension of Corollary \ref{t1.8}. 

\ms
\begin{pro} \label{t12.1}
The statement of Corollary \ref{t1.8} is valid for all the integers
$d$ and $n$ such that \eqref{10.9} holds. 
\end{pro}

\ms
Let us start the proof. Let $L$, $E$, and $x\in E \cap B(0,1) \sm L$ 
be as in the statement.
We want to apply Theorem~\ref{t1.5} to the set $E_x = E-x$ and the boundary 
$L_x = L-x$. Let $S_x$ denote the shade of $L$ seen from $x$,
as in \eqref{1.45}, and notice that $S = S_x-x$ is the usual shade of $L_x$
(seen from the origin). We are interested in the functional $F_x$ defined by
\begin{eqnarray} \label{12.1}
F_x(r) &=& r^{-d} \H^d(E \cap B(x,r)) + r^{-d} \H^d(S_x \cap B(x,r)).
\end{eqnarray}
Note that $F_x(r) = r^{-d} \H^d(E_x \cap B(0,r)) + r^{-d} \H^d(S \cap B(0,r))$,
by \eqref{1.9} and \eqref{1.10}; this is the same thing as $F(r)$
relative to $E_x$ and $L_x$.
If $\varepsilon$ is small enough, the assumptions of Theorem~\ref{t1.5}
are satisfied (with $U = B(0,2)$ and $R_1 = 2$), and \eqref{1.21} says that 
\begin{equation} \label{12.2}
F_x(r) e^{aA(r)} \ \text{ is nondecreasing on } (0,2), 
\end{equation}
where $a > 0$ is a constant that depends on $n$ and $d$, and $A$ is 
the same function as in \eqref{1.20}.

Our next task is to evaluate $F_x(r)$ for some large $r$.
Let $\tau_1 > 0$ be small, to be chosen later.

\begin{lem} \label{t12.2}
If $\varepsilon$ is small enough (depending on $\tau_1$), then
\begin{equation} \label{12.3}
F_x(r) \leq {3 \omega_d \over 2} + 2\tau_1
\ \text{ for } x\in E \cap B(0,1) \sm L \text{ and } 0 < r \leq {19 \over 10}.
\end{equation}
\end{lem}

\ms
\begin{proof}
Because of \eqref{12.2} and the fact that $A(3)$ is as small as we want (by 
\eqref{1.42}), it is enough to show that
\begin{equation} \label{12.4}
F_x(19/10) \leq {3 \omega_d \over 2} + \tau_1.
\end{equation}

Suppose the lemma fails for some $\tau_1 > 0$; this means that if
we take $\varepsilon = 2^{-k}$, we can find $L_k$, $E_k$ and $h_k$
as above, and then $x_k \in E_k \cap B(0,1) \sm L_k$ such that \eqref{12.4}
fails (for $E_k$ and $x_k$). 
By rotation invariance, we may even assume that $L_k = L$ for some fixed 
vector space $L$. Let us also replace $\{ E_k \}$ by a subsequence 
which converges locally in $B(0,3)$ to a limit $E_\infty$, and for which
$x_k$ has a limit $x_\infty \in E_\infty \cap \overline B(0,1)$.

Since $E_k$ satisfies \eqref{1.43} with the constant $\varepsilon = 2^{-k}$,
we see that $E_\infty \in \bV(L)$, i.e., is the union of two half $d$-planes
$H_1$ and $H_2$, that are bounded by $L$ and make an angle at least
$2\pi/3$ along $L$. The following sublemma will help us with measure estimates.

\begin{lem} \label{t12.3}
If $E_\infty \in \bV(L)$, then 
\begin{equation} \label{12.5}
r^{-d} \H^d(E_\infty \cap B(x_\infty,r)) 
+ r^{-d} \H^d(S_{x_\infty} \cap B(x_\infty,r))
\leq {3\omega_2 \over 2} 
\end{equation}
for all $x_\infty \in E_\infty\sm L$ and $r > 0$.
\end{lem}

\ms
We shall prove this lemma soon, but let us first see why it implies
Lemma \ref{t12.2}. 
First assume that $x_\infty \notin L$. Then
\begin{equation} \label{12.6}
\lim_{k \to +\infty} \H^d(S_{x_k} \cap B(x_k,19/10)) 
= \H^d(S_{x_\infty} \cap B(x_\infty,19/10))
\end{equation}
(because $x_k$ too stays far from $L$). 

Denote by $F_{x_k}$ the functional associated to $E_k$ and 
$x_k$ as in \eqref{12.1}, and pick any $r_1 \in (19/10, 2)$. Then
set $\ell = \limsup_{k \to +\infty} F_{x_k}(19/10)$, and observe that
\begin{eqnarray} \label{12.7}
\ell 
&=& (19/10)^{-d} \limsup_{k \to +\infty} \Big[
\H^d(S_{x_k} \cap B(x_k,19/10)) 
+\H^d(E_k \cap B(x_k,19/10))\Big]
\nn\\
&\leq& (19/10)^{-d} \H^d(S_{x_\infty} \cap B(x_\infty,19/10))
+ (19/10)^{-d} \limsup_{k \to +\infty} \H^d(E_k \cap \overline B(x_\infty,r_1))
\nn\\
&\leq& (19/10)^{-d} \H^d(S_{x_\infty} \cap B(x_\infty,19/10))
+ (19/10)^{-d}  \H^d(E_\infty \cap \overline B(x_\infty,r_1))
\\
&\leq& (19/10)^{-d} \, {3\omega_2 r_1^d \over 2}
\nn
\end{eqnarray}
by \eqref{12.6} and Lemma 22.3 in \cite{Sliding}, applied to the compact 
set $\overline B(x_\infty,r_1)$ as we did for \eqref{9.12}, and then \eqref{12.5}.
If we choose $r_1$ close enough to $19/10$, we deduce \eqref{12.4}
for $E_k$ and $x_k$ (and if $k$ is large enough) from \eqref{12.7}, 
and this contradiction with the definition of $E_k$ proves the lemma.

If $x_\infty \in L$, we simply notice that 
$\H^d(S_{x_k} \cap B(x_k,19/10)) \leq (19/10)^d \, {\omega_d \over 2}$,
and $\H^d(E_\infty \cap B(x_k,r_1)) = r_1^d \omega_d$, so
\begin{eqnarray} \label{12.8}
\ell 
&=& (19/10)^{-d} \limsup_{k \to +\infty} \Big[
\H^d(S_{x_k} \cap B(x_k,19/10)) +\H^d(E_k \cap B(x_k,19/10))\Big]
\nn\\
&\leq&  {\omega_d \over 2}
+ (19/10)^{-d} \limsup_{k \to +\infty} \H^d(E_k \cap \overline B(x_\infty,r_1))
\nn\\
&\leq& {\omega_d \over 2}
+ (19/10)^{-d}  \H^d(E_\infty \cap \overline B(x_\infty,r_1))
\leq {\omega_d \over 2} + (19/10)^{-d} \, {\omega_2 r_1^d \over 2},
\end{eqnarray}
again by Lemma 22.3 in \cite{Sliding}. We can let $r_1$ tend to
$19/10$ and get \eqref{12.4} as above, so Lemma~\ref{t12.2}
follows from Lemma \ref{t12.3}.

\ms
\noindent
{\it Proof of Lemma \ref{t12.3}.}
We start with the case of a set $V$ of dimension $1$ in a $2$-plane $P$.
That is, let the points $y, \ell \in P$ be given, with $y \neq \ell$,
and let $v$ denote the union of two half lines $h_1, h_2$ in $P$, 
that both start at $\ell$ and make an angle at least $2\pi\over 3$; 
also denote by $s$ the shade of $\ell$ seen from $y$, i.e., the set of points $w\in P$
such that $y + \lambda (w-y) = \ell$ for some $\lambda \in [0,1]$.
See Figure 12.1 for an illustration.
We want to check that for $\rho > 0$,
\begin{equation}\label{12.9}
\H^1(v \cap B(y,\rho)) + \H^1(s \cap B(y,\rho)) \leq 3 \rho.
\end{equation}
By rotation and dilation invariance, it is enough to prove this when 
$P = \R^2$, $y=0$, $\rho = 1$, and $\ell$ lies on the positive real axis.

When $\ell \geq 1$, $\H^1(s \cap B(y,1))=0$; the result is clear when
$\H^1(h_1 \cap B(y,\rho)) \leq 1$ because $\H^1(h_2 \cap B(y,\rho)) \leq 2$;
otherwise, $h_1$ makes an angle at least $2\pi/3$ with the positive real axis
(see Figures 12.1 and 12.2); 
that, is $h_1$ lies in the shaded region of Figure 12.3, and our
angle condition forces $h_2$ to make an angle at most $2\pi/3$ with
the positive real axis, hence to lie on the right of the dotted half line
of Figure 12.3, so that $\H^1(h_2 \cap B(y,\rho)) \leq 1$, 
and the desired inequality follows.

\vfill
\ms
\centerline{
\includegraphics[height=3.6cm]{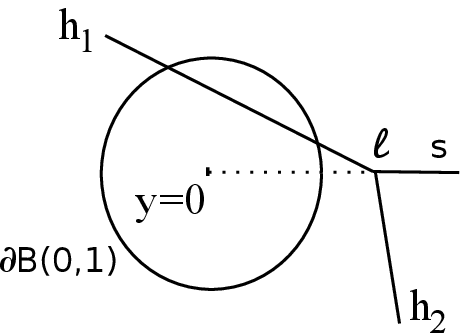}
\hskip0.8cm
\includegraphics[height=3.6cm]{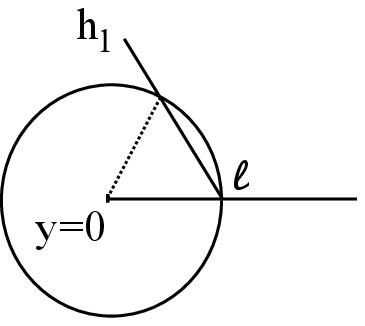}
\hskip0.5cm
\includegraphics[height=3.8cm]{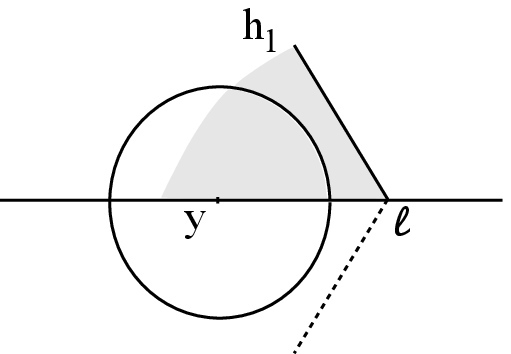}
}
\par\noindent
{\bf Figure 12.1(left).} General position with $\ell > 1$.
\par\noindent
{\bf Figure 12.2 (center).} The limiting case for $\H^1(h_1 \cap B(y,\rho)) \geq 1$.
\par\noindent
{\bf Figure 12.3 (right).} $h_1$ lies on the shaded region; then $h_2$
lies on the right of the dotted line.
\ms

\ms
So we may assume that $0 < \ell \leq 1$. For $i=1,2$, let $\theta_i \in [0,\pi]$
denote the angle between $h_i$ and the positive real axis; we may assume
that $\theta_1 \leq \theta_2$. 
Also denote by $\theta$ the angle between $h_1$ and $h_2$.
If $h_1$ and $h_2$ lie on the same side of the real axis, our
constraint that $\theta \geq 2\pi/3$ implies that $\theta_1 \leq \pi/3$
(see Figure 12.4).   
Then $\H^1(h_1 \cap B(y,\rho)) < 1$, while
$\H^1(h_2 \cap B(y,\rho)) + \H^1(s \cap B(y,\rho)) \leq 2$
(corresponding to $\theta_2=\pi$, look at Figure 12.4 again, and recall
that $\theta_2 \geq 2\pi/3$);
then \eqref{12.9} holds.

So we may assume that $h_1$ and $h_2$ lie on different sides of the real axis;
then $\theta_1+\theta_2+\theta = 2\pi$, our constraint that
$\theta \geq 2\pi/3$ yields $\theta_1 + \theta_2 \leq 4\pi/3$,
and since $\H^1(h_i \cap B(y,\rho))$ is easily seen to be an increasing
function of $\theta_i$, we may assume that $\theta_1 + \theta_2 = 4\pi/3$
and $\theta = 2\pi$. Let $h_3$ denote the half line in $P$ that makes an angle
of $2\pi/3$ with $h_1$ and $h_2$, and set $Y = h_1\cup h_2 \cup h_3$
(See Figure 12.5).

\vfill
\ms
\centerline{
\includegraphics[height=4cm]{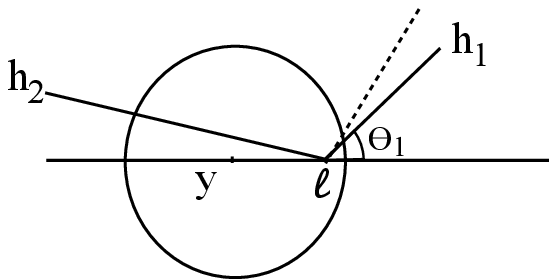}
\hskip0.4cm
\includegraphics[height=4cm]{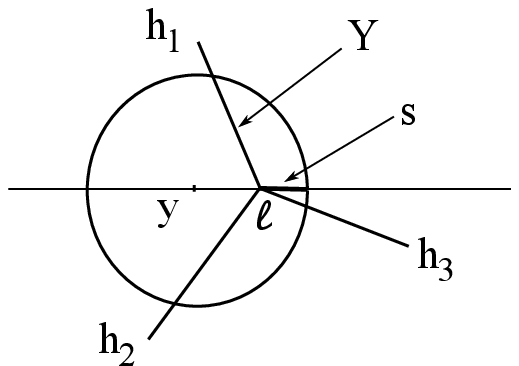}
}
\par\noindent
{\bf Figure 12.4(left).} General position when $\ell \leq 1$ and $h_1$ and
$h_2$ lie above the axis.
\par\noindent
{\bf Figure 12.5 (right).} The case when $h_1$ and
$h_2$ lie on different sides, and the set $Y = h_1\cup h_2 \cup h_3$.
\ms

Notice that $s\cap \overline B(0,1)$ is the shortest line segment from  $\ell$
to $\d B(0,1)$, hence $\H^1(s \cap B(y,1)) \leq \H^1(h_3 \cap B(y,1))$,
and \eqref{12.9} will follow as soon as we prove that 
\begin{equation}\label{12.10}
\H^1(Y\cap B(0,1)) \leq 3.
\end{equation}
This would probably not be so hard to compute, but it is simpler to notice that,
for $\ell$ fixed, $r^{-1}\H^1(Y\cap B(0,r))$ is a nondecreasing function of $r$, either
by  Proposition 5.16 in \cite{Holder} (the lazy way; notice that $0$ does not need to lie
on the minimal set $Y$), or (repeating the proof of monotonicity) because 
\begin{equation}\label{12.11}
{\d \over \d r}(\H^1(Y\cap B(0,r))\geq \sharp(Y \cap \d B(0,r))
= 3 \geq r^{-1} \H^1(Y\cap B(0,r))
\end{equation}
for $r > \ell$, and where the last inequality comes from the minimality
of $Y$ (compare $Y$ with the cone over $Y \cap \d B(0,r)$).

This proves \eqref{12.9}, and now we deduce \eqref{12.5} from \eqref{12.9}
and a slicing argument. Let $L$, $E_\infty$, $x_\infty$, and $r$ be as in 
Lemma \ref{t12.3}.
Recall that since $E_\infty \in \bV(L)$, it is the union of two half $d$-planes 
$H_1$ and $H_2 \in \bH(L)$ bounded by $L$. 
For $i=1,2$, let $e_i$ be the unit vector such that $e_i \in H_i$ 
and $e_i \perp L$, and let $h_i$ denote the half line 
$h_i = \big\{ t e_i, \, ; \, t \geq 0 \big\}$; 
notice that $H_i = h_i \times L$ (an orthogonal product), and that $h_1$ makes
an angle at least $2\pi/3$ with $h_2$, by definition of $\bV(L)$. 
See Figure 12.4.

\vfill
\centerline{
\includegraphics[height=5.cm]{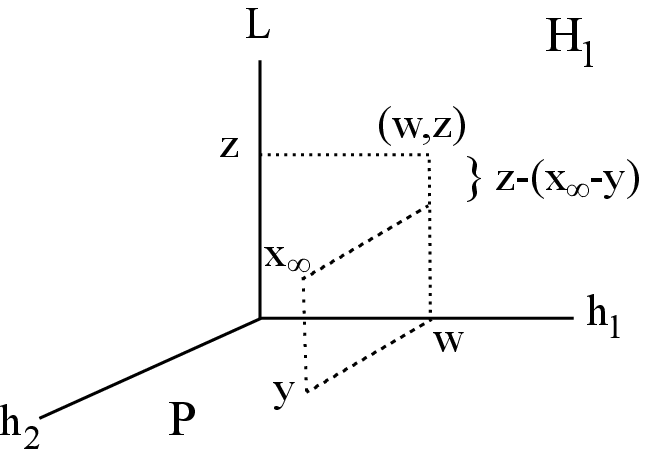}
}
\par\centerline{
{\bf Figure 12.6.} Notation for the slicing argument ($h_2$ is not orthogonal to
$h_1$).
} 

\ms
Denote by $P$ the $2$-plane that contains $e_1$ and $e_2$ and by $y$ the orthogonal projection of $x_{\infty}$ on $P$. If $e_2 = -e_1$, pick any plane
that contains $e_1$, or notice that \eqref{12.5} is easy to prove directly
(since $E_\infty$ is a $d$-plane through $L$).
Then fix $i \in \{1,2\}$, and write the current point of
$H_i$ as $(w,z)$, with $w\in h_i$ and $z\in L$. 
For $z\in L$, the slice of $H_i \cap B(x_\infty,r)$ at height $z$
is the possibly empty set
\begin{equation}\label{12.12}
W_i(z) = \big\{ w\in h_i \, ; \, (w,z) \in B(x_\infty,r) \big\}.
\end{equation}
Write $(w,z)-x_\infty$ as the sum of
$w-y \in P$ and $z-x_\infty+y \in P^\perp$ (recall that
$z\in L \subset P^\perp$); then
$|(w,z)-x_\infty|^2 = |w-y|^2+|z-(x_\infty-y)|^2$,
the condition $|(w,z)-x_\infty|^2 < r^2$ becomes
$|w-y|^2 < r_z^2$, with
\begin{equation}\label{12.13}
r_z^2 = \max(0, r^2 - |z-(x_\infty-y)|^2),
\end{equation}
and then 
\begin{equation} \label{12.14}
W_i(z) = h_i \cap B(y,r_z).
\end{equation}
Then we apply Fubini's theorem on the $d$-plane that contains $H_i$ and get that
\begin{equation} \label{12.15}
\H^d(H_i \cap B(x_\infty,r))
= \int_{z \in L} \H^1(W_i(z)) dz
= \int_{z \in L} \H^1(h_i \cap B(y,r_z)) dz
\end{equation}
where for this  computation we have assumed that the various Hausdorff measures
were normalized so that they coincide with the corresponding Lebesgue measures
on vector spaces. (After this, even if we chose different normalizations, \eqref{12.5}
will follow because the result of the computation is exact when $E_\infty$ 
is a plane through $x_\infty$.)

Denote by $s$ the shade of $y$ seen from $0$; the same computation, with
$H_i$ replaced by the half $d$-space $S_{x_\infty}$, shows that
\begin{equation} \label{12.16}
\H^d(S_{x_\infty} \cap B(x_\infty,r))
= \int_{z \in L} \H^1(s \cap B(y,r_z)) dz.
\end{equation}
Thus, setting 
$I = \H^d(E_\infty \cap B(x_\infty,r)) + \H^d(S_{x_\infty} \cap B(x,r))$,
\eqref{12.15} and \eqref{12.16} yield
\begin{eqnarray} \label{12.17}
I &=& \int_{z \in L} \big[\H^1(h_1 \cap B(y,r_z))
+\H^1(h_2 \cap B(y,r_z))+\H^1(s \cap B(y,r_z))\big] dz
\nn\\
&=& \int_{z \in L} \big[\H^1(v \cap B(y,r_z))+\H^1(s \cap B(y,r_z))\big] dz
\leq \int_{z \in L} 3 r_z dz
\end{eqnarray}
by \eqref{12.9} and because the angle of $h_1$ and $h_2$
is the same as the (smallest) angle of $H_1$ and $H_2$. 
Denote by $z_0$ the orthogonal projection of $x_{\infty}$ on $L$;
then $|z-(x_\infty-y)|^2 \geq |z-z_0|^2$ (because $z_0$ is also the orthogonal
projection of $x_\infty-y$), and hence $r_z^2  \leq \max(0,r^2-|z-z_0|^2)$. Thus
\eqref{12.17} yields
\begin{equation}\label{12.18}
I \leq 3 \int_{z \in L \cap B(z_0,r)}  \big(r^2-|z-z_0|^2\big)^{1/2}dz
= 3 \int_{w \in \R^{d-1} \cap B(0,r)} \big(r^2-|w|^2\big)^{1/2}dw
= {3 \omega_d r^d \over 2},
\end{equation}
because we recognize the way to compute the measure of a (half) ball of $\R^d$
by vertical slicing and induction. 
This completes our proof of \eqref{12.5}; Lemmas~\ref{t12.3} and \ref{t12.2} follow.
\qed\end{proof}

\ms
Let $x\in E \cap B(0,1) \sm L$ be given. Since 
$F_x(r) = r^{-d} \H^d(E\cap B(x,r))$ for $r$ small, 
Lemma~\ref{t12.2} implies that 
\begin{equation}\label{12.19}
\theta_x(0) = \lim_{r \to 0} r^{-d} \H^d(E\cap B(x,r))
\leq {3 \omega_d \over 2}+2\tau_1 \leq {3 \omega_d \over 2}+\tau 
\end{equation}
if we choose $2\tau_1 \leq \tau$, and so \eqref{1.44} holds.

\ms
From now on, we assume in addition that $\theta_x(0) > \omega_d$, and set
$\delta(x) = \dist(x,L) > 0$, as in the statement of Corollary \ref{t1.8}.
Let $X$ be any blow-up limit of $E$ at $x$; standard arguments show that
such things exist.
By Proposition 7.31 in \cite{Holder} (for instance), $X$ is a minimal cone 
with constant density $d(X) = \theta_x(0)$ (see the notation \eqref{10.1}). 
We assumed that $\theta_x(0) > \omega_d$, so $d(X) > \omega_d$
and $X$ is not a plane. By \eqref{10.9}, $d(X) \geq {3 \omega_d \over 2}$
(because $d(X) = {3 \omega_d \over 2}$ when $X\in \bY_0(n,d)$).
Hence 
\begin{equation} \label{12.20}
\lim_{r \to 0} F_x(r) = \theta_x(0) \geq {3 \omega_d \over 2},
\end{equation}
by the definitions \eqref{12.1} and \eqref{1.44}, and because $x\in E \sm L$.
We now deduce from \eqref{12.2} that for $0 < r \leq 19/10$,
\begin{equation} \label{12.21}
F_x(r) \geq e^{-a A(r)} \theta_x(0)
\geq e^{-a \varepsilon } \theta_x(0) 
\geq {3 \omega_d \over 2} - \tau_1
\end{equation}
by \eqref{1.42} and if $\varepsilon$ is small enough (depending on $\tau_1$).
Because of \eqref{12.3}, we thus get that
\begin{equation} \label{12.22}
{3 \omega_d \over 2} - \tau_1 \leq F_x(r) \leq {3 \omega_d \over 2} + 2 \tau_1
\ \text{ for } 0 < r \leq {19 \over 10}.
\end{equation}

Our next task is to use \eqref{12.22}, in conjunction with 
Corollary \ref{t9.3}, to get a local control of $E$ in balls $B(x,r)$,
where $x\in E \sm L$ is as above.  
We start with the small radii, for which the boundary $L$ does not
really interfere.

\begin{lem} \label{t12.4}
Let $n$, $d$, $E$, and $L$ satisfy the assumptions of Corollary \ref{t1.8} and let
$x\in E \cap B(0,1) \sm L$ be such that $\theta_x(0) > \omega_d$.
In particular, suppose that $\varepsilon$ in \eqref{1.42} and \eqref{1.43}
is small enough, depending on $n$, $d$, and $\tau >0$.
Then for $0 < r < {1 \over2} \dist(y,L)$ there is a minimal cone $Y\in \bY_x(n,d)$ 
(i.e., of type $\bY$ and centered at $x$) such that
\begin{equation} \label{12.23}
d_{x,r}(E,Y) \leq \tau
\end{equation}
and 
\begin{eqnarray} \label{12.24}
&\,&\big|\H^d(E \cap B(y,t)) - \H^d(Y\cap B(y,t))\big| \leq \tau r^d
\nn\\
&\,&\hskip2cm
\text{ for $y\in \R^n$ and $t>0$ such that } B(y,t) \i B(x,r).
\end{eqnarray}
\end{lem}

\begin{proof}
First observe that $r < {1 \over2}\dist(y,L) \leq {1 \over2}$
because $x\in B(0,1)$, so $E$ is almost minimal in $B(x,2r)$, with
no boundary condition, and we may try to apply 
Proposition 7.24 in \cite{Holder} in $B(x,2r)$, and with some constant
$\tau_2$ that will be chosen soon, to find an approximating cone.
The condition that $h(4r)$ be small enough comes from \eqref{1.42},
and the almost constant density condition (7.25), which requires that
$\theta_x(2r) \leq \int_{0 < t < r/50} \theta_x(t) + \varepsilon'$,
for some small $\varepsilon'$ that depends on $\tau_2$, is a consequence
of \eqref{12.22} (choose $\tau_1$ small, depending on $\tau_2$, and then
$\varepsilon$ even smaller), because
$\theta_x(t) = F_x(t)$ for $t < {1 \over2}\dist(y,L)$.
So Proposition 7.24 in \cite{Holder} yields the existence of a coral minimal
cone $X$, centered at $x$, such that 
\begin{equation} \label{12.25}
\dist(y,X) \leq 2\tau_2 r 
\ \text{ for } y\in E \cap B(x,2(1-\tau_2)r),
\end{equation}
\begin{equation} \label{12.26}
\dist(y,E) \leq 2\tau_2 r 
\ \text{ for } y\in X \cap B(x,2(1-\tau_2)r),
\end{equation}
and 
\begin{eqnarray} \label{12.27}
&\,&\big|\H^d(E \cap B(y,t)) - \H^d(X\cap B(y,t))\big| \leq 2^d \tau_2 r^d 
\nn\\
&\,&\hskip2cm
\text{ for $y\in \R^n$ and $t>0$ such that } B(y,t) \i B(x,2(1-\tau_2)r).
\end{eqnarray}
Let us first apply \eqref{12.27} with $B(y,t) = B(x,r)$; we get that
\begin{equation} \label{12.28}
|d(X)-\theta_x(r)| = r^{-d}\big| \H^d(X\cap B(x,r))-\H^d(E \cap B(x,r)) \big|
\leq 2^d \tau_2.
\end{equation}
Since $\theta_x(r) = F_x(r)$ because $B(x,r)$ does not meet $L$,
we deduce from \eqref{12.22} that
\begin{equation} \label{12.29}
\big|d(X)-{3 \omega_d \over 2}\big| \leq 2^d \tau_2 + 2\tau_1.
\end{equation}
Now we use our assumption \eqref{10.9}, and its consequence
in Lemma \ref{t10.3}, which we apply with some small constant $\tau_3$ 
that will be chosen soon. We get that if $\tau_1$ and $\tau_2$ are small enough,
depending on $\tau_3$, we can find a minimal cone $Y \in \bY_x(n,d)$
such that
\begin{equation} \label{12.30}
d_{x,1}(X,Y) \leq \tau_3
\end{equation}
(the case when $X$ is a plane is excluded by \eqref{12.29}). We just need 
to check that if $\tau_3$ is chosen small enough, $Y$ satisfies our conditions 
\eqref{12.23} and \eqref{12.24}.

For \eqref{12.23}, this is a simple consequence of \eqref{12.25}, \eqref{12.26},
\eqref{12.30}, and the triangle inequality, so we skip the details. 
For \eqref{12.24} we proceed, as usual, by contradiction and compactness: 
we suppose that \eqref{12.25}, \eqref{12.26}, and \eqref{12.30} do not 
imply \eqref{12.23} and \eqref{12.24} (for a same $Y$),
take a sequence $\{ E_k \}$ for which they are satisfied with 
$\varepsilon + \tau_1 + \tau_3 < 2^{-k}$ 
(some center $x_k$, and some radius $r_k$), but
\eqref{12.24} fails whenever \eqref{12.23} holds.
By \eqref{12.25} and \eqref{12.26}, the sets $E'_k = r_k^{-1}(E_k-x_x)$
converge, modulo extraction of a subsequence, to a cone $Y_0 \in \bY_0(n,d)$,
and then we apply Lemma \ref{t9.2} (with $2^{-d-1}\tau$, say) in
$B(0,3/2)$, and get that \eqref{12.24} holds for $k$ large, because
of \eqref{9.22}. Notice that we can take the cone $x+Y_0$ to check this,
and then \eqref{12.23} still holds with the cone $x+Y_0$; then we get the 
desired contradiction. This completes the proof of Lemma~\ref{t12.4}.
\end{proof}

Let $N$ be the large number that shows up in the statement of
Corollary \ref{t1.8}.
A simple consequence of Lemma \ref{t12.4} (applied with a small enough 
$\tau$) is that, if $x$ is as in Corollary~\ref{t1.8} or Lemma \ref{t12.4}, then
\begin{equation} \label{12.31}
\dist(x,L) \leq 10 N^{-1},
\end{equation}
because otherwise the good approximation of $E$ by a cone of type
$\bY$ in $B(x,{1\over 3}\dist(L))$ would contradict its good approximation
in $B(0,3)$ by a set of type $\bV$, given by \eqref{1.43}.
Notice that \eqref{12.31} is slightly better than what we announced in
Corollary \ref{t1.8}. 

We continue our local description of $E$ in balls $B(x,r)$ with the case of
intermediate radii $r$, for which we shall combine \eqref{12.22} with 
Corollary \ref{t9.3}.

\begin{lem} \label{t12.5}
Let $n$, $d$, $E$, and $L$ satisfy the assumptions of Corollary \ref{t1.8} and let
$x\in E \cap B(0,1) \sm L$ be such that $\theta_x(0) > \omega_d$.
In particular, suppose that $\varepsilon$ in \eqref{1.42} and \eqref{1.43}
is small enough, depending on $n$, $d$, and $\tau >0$.
Then let $r > 0$ be such that
\begin{equation} \label{12.32}
\delta(x) \leq r  \leq 10N \delta(x), 
\end{equation}
where $\delta(x) = \dist(x,L)$ as before. Let $Y$ be
the (unique) minimal cone $Y\in \bY_x(n,d)$ such that
\begin{equation} \label{12.33}
L \subset Y,
\end{equation}
denote by $S_x$ the shade of $L$ seen from $x$
and set $W = \overline{Y \sm S_x}$ (as in \eqref{1.45}); then
\begin{equation} \label{12.34}
d_{x,r}(E,W) \leq \tau
\end{equation}
and
\begin{eqnarray} \label{12.35}
\av{\H^d(E \cap B(y,t)) - \H^d(W \cap B(y,t))} \leq \tau r^d &&
\nonumber \\
&&\hskip-7cm
\ \text{ for all $y\in \R^n$ and $t>0$ such that } B(y,t) \subset B(x,r).
\end{eqnarray}
\end{lem}

\ms\noindent{\it Proof.}
We want to apply Corollary \ref{t9.3}, with a small constant $\tau_4$
that will be chosen later, to control $E$ in $B(x,r)$. 
Set $r_1 = {10 r \over 9}$ to have some room to play.
Because of the normalization in the corollary, 
we apply it to the set $E_x = r_1^{-1}(E-x)$, which is sliding almost minimal in an
open set that contains $B(0,1)$ (because $B(x,r_1) \i B(0,3)$ since $x\in B(0,1)$
and $r  \leq 10N \delta(x) \leq 1$ by \eqref{12.32} and \eqref{12.31}),
and with a boundary condition that comes from the set $L_x = r_1^{-1}(L-x)$.

The distance requirement \eqref{9.26} is that
$\tau_4 \leq \dist(0,L_x) \leq {9\over 10}$, and since 
\begin{equation} \label{12.36}
\dist(0,L_x) = r_1^{-1} \delta(x) = {9 \delta(x) \over 10r}
\end{equation}
we see that \eqref{9.26} follows from \eqref{12.32} as soon as we take 
$\tau_4 < {9 \over 100N}$. 

The bilipschitz condition on $L_x$ is trivially satisfied, because $L_x$ is an affine subspace; the first half of \eqref{9.28} 
(i.e., the fact that $h(r)$ is very small) follows from \eqref{1.42} if 
$\varepsilon$ is small enough; and the more important second half follows 
from \eqref{12.22}.

Then Corollary \ref{t9.3} yields the existence of a coral minimal cone $X_0$,
centered at the origin and with no boundary condition,
with the following properties. First,
\begin{equation} \label{12.37}
L_x \cap B(x,99/100) \i  X_0,
\end{equation}
by \eqref{9.29} and because the comment below 
Corollary \ref{t9.3} says that we can take $L'=L_x$.
Next denote by $S$ the shade of $L_x$ seen from the origin,
and set $E_0 = \overline{X_0 \sm S}$; then by \eqref{9.31}
$E_0$ is a coral minimal set in $B(0,1-\tau_4)$, with 
sliding boundary condition defined by $L_x$, and is $\tau_4$-close
to $E_x$ in $B(0,1-\tau_4)$, in the sense of \eqref{9.32}-\eqref{9.34}.
By \eqref{9.32} and \eqref{9.33},
\begin{equation} \label{12.38}
d_{0,1-\tau_4}(E_x,E_0) \leq \tau_4.
\end{equation}
Let us also apply \eqref{9.34} to the ball $B = B(0,t)$, with 
$t= (20N)^{-1}$. Notice that $t \leq {\delta(x) \over 2 r} < \dist(0,L_x)$
by \eqref{12.32} and \eqref{12.36}, so $B$ does not meet $L_x$; then
\begin{eqnarray} \label{12.39}
d(X_0) &=& \H^d(X_0 \cap B(0,1)) = t^{-d} \H^d(X_0 \cap B)
= t^{-d} \H^d(\overline{X_0 \sm S} \cap B)
= t^{-d} \H^d(E_0 \cap B)
\nn\\
&\leq&  t^{-d} \H^d(E_x \cap B)+ t^{-d}\tau_4 
= t^{-d} r_1^{-d} \H^d(E \cap B(x,tr_1)) + (20N)^d \tau_4
\nn\\
&=& F(tr_1) + (20N)^d \tau_4 
\leq {3 \omega_d \over 2} + 2 \tau_1 + (20N)^d\tau_4 
\end{eqnarray}
by the definition \eqref{10.1} of $d(X_0)$, where $S$ still denotes the shade of 
$L_x$, because $S \cap B = \emptyset$, by definition of $E_0$, by \eqref{9.34}
and because $t = (20N)^{-1}$, then by \eqref{12.1} and because 
$B(x,tr_1) \cap S_x = \emptyset$, and finally by \eqref{12.22}.
Thus $d(X_0)$ is as close to ${3 \omega_d \over 2}$ as we want.

When $d=2$, and more generally when there is a constant 
$d_{n,d} > {3 \omega_d \over 2}$ such that \eqref{10.10} holds,
we can deduce from this that $X_0 \in \bY_0(n,d)$, and this will simplify
our life. In the general case when we only have \eqref{10.9}, we 
can still use Lemma \ref{t10.3}, as for \eqref{12.30} above, to show
that if $\tau_1$ and $\tau_4$ are small enough, we can find a minimal cone 
$Y_0 \in \bY_0(n,d)$ such that
\begin{equation} \label{12.40}
d_{0,1}(X_0,Y_0) \leq \tau_5
\end{equation}
where $\tau_5 > 0$ is a new small constant, that will be chosen soon, 
depending on $\tau$. 

Pick a point $z_0 \in L_x \cap \overline B(0,9/10)$; such a point
exists because $\dist(0,L_x) = {9 \delta(x) \over 10r} \leq {9\over10}$.
Set $D = z_0 + [L \cap B(0,1/20)]$; notice that 
$D \i L_x \cap B(99/100) \subset X_0$ because $L$ is the vector
space parallel to $L_x$, and by \eqref{12.37}. By \eqref{12.40}, 
\begin{equation} \label{12.41}
\dist(z,Y_0) \leq \tau_5 \ \text{ for } z\in D.
\end{equation}
Also set $D' = z_0 + [L \cap B(0,1/40)]$; if $\tau_5$ is small enough,
it follows from \eqref{12.41} and the elementary geometry of 
$Y_0 \in \bY_0(n,d)$ that there is a single face $F$ of $Y_0$ such that
\begin{equation} \label{12.42}
\dist(z,F) \leq \tau_5 \ \text{ for } z\in D'.
\end{equation}
Finally denote by $Y_1$ the cone of  $\bY_0(n,d)$ that contains 
$L_x$ (or equivalently $D'$). The face of $Y_1$ that contains $L_x$
is quite close to $F$, by \eqref{12.42}, and hence
\begin{equation} \label{12.43}
d_{0,1}(Y_1,Y_0) \leq C\tau_5
\end{equation}
where $C$ depends only on $n$ and $d$. For the record, let us mention
that we can take $Y_1 = Y_0 = X_0$ when \eqref{10.10} holds for some
$d_{n,d} > {3 \omega_d \over 2}$, rather than the weaker \eqref{10.9}.

Now set $Y = x + Y_1 \in \bY_x(n,d)$, notice that $L \subset Y$
because $L_x \i Y_1$; we just need to check that $Y$ satisfies the
conditions \eqref{12.34} and \eqref{12.35}. For the general case
we will need the following lemma.

\begin{lem} \label{t12.6} 
Recall from below \eqref{12.37} that $S$ is the shade of $L_x$ 
seen from the origin, and that $E_0 = \overline{X_0 \sm S}$.
Also set $E_1= \overline{Y_1 \sm S}$. 
For each small $\tau_6 > 0$, we can choose
$\tau_5$ so small that in the present situation,
\begin{equation} \label{12.44}
d_{0,95/100}(E_0, E_1) \leq \tau_6
\end{equation} 
and 
\begin{eqnarray} \label{12.45}
&\,&\big|\H^d(E_0 \cap B(y,t)) 
- \H^d(E_1\cap B(y,t))\big| 
\leq \tau_6 
\nn\\
&\,&\hskip2cm
\text{ for $y\in \R^n$ and $t>0$ such that } B(y,t) \i B(0,96/100).
\end{eqnarray}
\end{lem}

Before we prove this, let us mention that in the simpler situation
when we have \eqref{10.10}, the lemma holds trivially because
$Y_1 = X_0$. Also, let us first see how Lemma \ref{t12.5} follows from
Lemma \ref{t12.6}, and prove Lemma \ref{t12.6} afterwards.

Recall that we just need to check that $Y=x+Y_1$ satisfies
\eqref{12.34} and \eqref{12.35}.
We deduce from \eqref{12.38} and \eqref{12.44} that
$d_{0,94/100}(E_x,E_1) \leq 2\tau_4 + 2 \tau_6$. Since
$E = x + r_1 E_x$ and $W = \overline{Y \sm S_x} = x + r_1 E_1$
(see above \eqref{12.34}), we also get that
\begin{equation} \label{12.46}
d_{x,94r_1/100}(E,W) \leq 2\tau_4 + 2 \tau_6 \, ;
\end{equation}
\eqref{12.34} follows, because 
$94r_1/100= 94r/90 > r$, and if $\tau_4$ and $\tau_6$ are small
enough. As for \eqref{12.34}, let $B = B(y,t) \subset B(x,r)$ be given,
set $B' = r_1^{-1}(B-x)$; then $B' \subset B(0,95/100)$ and
\begin{equation} \label{12.47}\begin{aligned}
\big|\H^d(E \cap B) &- \H^d(W \cap B)\big|
= r_1^d \av{\H^d(E_x \cap B') - \H^d(E_1 \cap B')}
\\
&= r_1^d \av{\H^d(E_x \cap B') - \H^d(E_0 \cap B')} 
+ r_1^d \av{\H^d(E_0 \cap B') - \H^d(E_1 \cap B')} 
\\
&\leq \tau_4 r_1^d + \tau_6 r_1^d
= (\tau_4 + \tau_6) (10 r /9)^d
\end{aligned}
\end{equation}
by \eqref{9.34} (for $E_x$ and with the constant $\tau_4$; see below
\eqref{12.37}) and \eqref{12.45} (recall that $E_0 = \overline{X_0 \sm S}$
and $E_1 = \overline{Y_1 \sm S}$). This proves \eqref{12.35}; thus
Lemma \ref{t12.5} will follow as soon as we prove Lemma \ref{t12.6}.

\ms\noindent{\it Proof of Lemma \ref{t12.6}.}
The reason why we need to prove something is that although 
\begin{equation} \label{12.48}
d_{0,1}(Y_1,X_0) \leq C \tau_5
\end{equation}
by \eqref{12.40} and \eqref{12.43},
removing $S$ could change the situation. At least, both cones contain
$L_x \cap B(0,99/100)$ by \eqref{12.37} and the definition of $Y_1$, 
hence they both contain $S\cap B(0,99/100)$. Even that way, we still
need to check that $X_0$ does not contain a piece that lies close to
$S\cap B(0,99/100)$ (hence, to $Y_1$), but not close to $E_1$,
and similarly with $X_0$ and $Y_1$ exchanged.

Set $\rho = 98/100$ and $B= B(0,\rho)$; we want to evaluate the size of 
$E_0 \cap S \cap B$.
In fact, because of the way $E_0$ was obtained in the proof of Corollary \ref{t9.3},
it is such that the functional $F$ is constant on $(0,1)$
(see below \eqref{9.64}), and we could deduce from \eqref{1.12} that
$\H^d(E_0 \cap B(0,1) \cap S) = 0$. But let us pretend we did not notice
and check that $\H^d(E_0 \cap B \cap S)$ is small in a way which is more 
complicated, but easier to track. As we just said,
\begin{equation} \label{12.49}
S \cap B \subset X_0,
\end{equation}
because $X_0$ is a cone that contains $L_x \cap B$
(by \eqref{12.37}). Let us check that
\begin{equation} \label{12.50}
S \cap B \sm E_0 = X_0 \cap B \sm E_0.
\end{equation}
The direct inclusion follows from \eqref{12.49}, and the converse from
the fact that $E_0 = \overline{X_0 \sm S} \supset X_0 \sm S$. Then
\begin{eqnarray} \label{12.51}
\H^d(X_0 \cap B) &=& \H^d(E_0 \cap B) + \H^d(X_0 \cap B \sm E_0)
= \H^d(E_0 \cap B) + \H^d(S \cap B \sm E_0)
\nn\\
&=& \H^d(E_0 \cap B) + \H^d(S \cap B) - \H^d(S \cap B \cap E_0).
\end{eqnarray}
But $\H^d(E_0 \cap B) \leq \H^d(E_x \cap B)+ \tau_4$ by \eqref{9.34}
and $\H^d(X_0 \cap B) = \rho^d d(X_0) 
\geq {3 \omega_d \rho^d \over 2}$
by \eqref{10.9} and because $X_0$ is not a plane, so
\begin{eqnarray} \label{12.52}
\H^d(S \cap B \cap E_0) 
&\leq& \H^d(E_x \cap B)+ \tau_4 + \H^d(S \cap B) 
- {3 \omega_d \rho^d\over 2} 
\nn\\
&=& r_1^{-d} \H^d(E \cap B(x,\rho r_1))+ \tau_4 
+ r_1^{-d}\H^d(S_x \cap B(x,\rho r_1)) - {3 \omega_d \rho^d\over 2} 
\nn\\
&=& \rho^d F_x(\rho r_1) + \tau_4  - {3 \omega_d  \rho^d\over 2}
\leq \tau_4 + 2 \rho^d \tau_1  \leq \tau_4 + 2 \tau_1
\end{eqnarray}
by \eqref{12.1} (and because $E_x = r_1^{-1}(E-x)$ and 
$L_x = r_1^{-1}(L-x)$), then by \eqref{12.22}.

For $B' = B(y,t) \i B$, the proof of \eqref{12.51} also yields
\begin{equation} \label{12.53}
\H^d(X_0 \cap B') 
= \H^d(E_0 \cap B') + \H^d(S \cap B') - \H^d(S \cap B' \cap E_0),
\end{equation}
which implies that
\begin{equation} \label{12.54}
|\H^d(X_0 \cap B') - \H^d(E_0 \cap B') - \H^d(S \cap B')|
\leq \H^d(S \cap B \cap E_0) \leq \tau_4 + 2 \tau_1.
\end{equation}
For the other cone $Y_1$, we have the simpler formula
\begin{equation} \label{12.55}
\H^d(Y_1 \cap B') = \H^d(\overline{Y_1 \sm S}\cap B') + \H^d(S \cap B')
= \H^d(E_1 \cap B') + \H^d(S \cap B')
\end{equation}
because $Y_1$ is a cone in $\bY_0(n,d)$ that contains $L_x$, and hence $S$.

By \eqref{12.48} and the same compactess argument using Lemma \ref{t9.2} as below \eqref{11.51} or \eqref{12.30},
\begin{equation} \label{12.56}
|\H^d(X_0 \cap B') - \H^d(Y_1\cap B')| \leq {\tau_6 \over 2}
\end{equation}
for every $B'$ as above, where $\tau_6$ is the small constant
that is given for Lemma \ref{t12.6}, provided that we take $\tau_5$
accordingly small. Thus
\begin{eqnarray} \label{12.57}
|\H^d(E_0 \cap B') - \H^d(E_1\cap B')| 
&\leq& \big|[\H^d(X_0 \cap B') - \H^d(S\cap B')]- \H^d(E_1\cap B')\big|
+\tau_4 + 2 \tau_1
\nn\\
&=&  \big|\H^d(X_0 \cap B') - \H^d(Y_1\cap B')\big| +\tau_4 + 2 \tau_1
\leq {\tau_6 \over 2}+\tau_4 + 2 \tau_1,
\end{eqnarray}
by \eqref{12.54}, \eqref{12.55}, and \eqref{12.56};
this implies \eqref{12.45} if $\tau_1$ and $\tau_4$ are small enough.

Now we check \eqref{12.44}. First we show that
\begin{equation} \label{12.58}
\dist(y, E_0) \leq C \tau_5
\ \text{ for } y \in B(0,96/100) \cap E_1. 
\end{equation}
Let $y \in B(0,96/100) \cap E_1$ be given; 
since $y\in Y_1$, \eqref{12.48} says that we can find $z\in X_0$ 
such that $|z-y| \leq C \tau_5$. If $z\in X_0 \sm S$, then $z\in E_0$ and 
$\dist(y, E_0) \leq |z-y| \leq C \tau_5$, as needed. So we may assume
that $z\in S$.

Notice that $\dist(y,S) = \dist(y,L_x)$ by elementary geometry
(because $y\in E_1 = \overline{Y_1 \sm S}$ and $Y_1$ is the 
cone of $\bY_0(n,d)$ that contains $L_x$). Then
\begin{equation} \label{12.59}
\dist(y,L_x) = \dist(y,S) \leq |y-z| \leq C \tau_5,
\end{equation}
and then $\dist(y, E_0) \leq C \tau_5$ too, because
$X_0$ is a cone that contains $B \cap L_x$, and then the points
$\lambda \xi$, $\xi\in B \cap L_x$ and $\lambda \in (0,1)$, lie in $E_0$.
This completes our proof of \eqref{12.58}.

Now we want to show that 
\begin{equation} \label{12.60}
\dist(y,E_1) \leq \tau_7 \ \text{ for } y \in E_0 \cap B(0,95/100),
\end{equation}
with a constant $\tau_7$ that is as small as we want;
\eqref{12.44} will follow at once from this and \eqref{12.58},
except that formally we would need to replace $\tau_6$ with $\tau_7$
in the statement. Suppose this fails; then of course
\begin{equation} \label{12.61}
\H^d(E_1 \cap B(z,\tau_7)) = 0.
\end{equation}
Recall from the lines below \eqref{12.37} that $E_0$ is a coral
minimal set in $B(0,1-\tau_4)$; hence, by the local Ahlfors-regularity
property \eqref{2.9}, 
\begin{equation} \label{12.62}
\H^d(E_0 \cap B(z,\tau_7)) \geq C^{-1} \tau_7^d,
\end{equation}
where $C$ depends only on $n$ and $d$. But 
$B(z,\tau_7) \i B(0,96/100)$, and \eqref{12.45} says that
$|\H^d(E_0 \cap B(z,\tau_7)) - \H^d(E_1\cap B(z,\tau_7))| \leq \tau_6$.
We can prove this with any $\tau_6 > 0$, and if we choose $\tau_6$
small enough, we get a contradiction with \eqref{12.62} or \eqref{12.61}
which proves \eqref{12.60}. This concludes our proof of \eqref{12.44};
Lemma \ref{t12.6} follows, and as was checked earlier, so does 
Lemma \ref{t12.5}.
\qed$\quad\Box$

We easily deduce Proposition \ref{t12.1} and Corollary \ref{t1.8}
from Lemma \ref{t12.5}: the assumptions (for the more general
Proposition \ref{t12.1}) are the same, the fact that $\delta(x) \leq N^{-1}$
follows from \eqref{12.31}, and the description of $E$ and $W$ in
\eqref{1.45}-\eqref{1.47} follows by applying Lemma \ref{t12.5}
with $r = 3\delta(x)$ and the constant $3^{-d}\tau$.

\ms
We complete this section with a rapid description of $E$ at the large scales 
that are not covered by Corollary \ref{t1.8}. 

\begin{pro} \label{t12.7}
Let the dimensions $n$ and $d$ satisfy \eqref{10.9} (thus $d=2$, or
$d=n-1$ and $n \leq 6$ work). For each choice of constants $\tau > 0$
and $A \geq 10$, we can find $\varepsilon > 0$ and $N \geq 100+ \tau^{-1}$,
with the following properties. 
Let $L$ be a vector $(d-1)$-plane and let $E$ be a coral sliding almost 
minimal set in $B(0,3)$, with boundary condition coming from $L$ 
and a gauge function $h$ that satisfies \eqref{1.42},
such that $d_{0,3}(E,V) \leq \varepsilon$ (as in \eqref{1.43}) 
for some $V \in \bV(L)$. Then let $x\in E \cap B(0,1) \sm L$
be a singular point, in the sense that $\theta_x(0) > \omega_d$
(see the definition in \eqref{1.44}), and denote by $y$ 
the orthogonal projection of $x$ on $L$. Then 
\begin{equation} \label{12.63}
\dist(x,L) \leq N^{-1}
\end{equation}
and, for every radius $r$ such that
\begin{equation} \label{12.64}
N \dist(x,L)\leq r \leq (2A)^{-1} 
\end{equation}
there is a cone $X$ centered at $y$, which is a coral sliding minimal set in $\R^n$, 
with the boundary condition coming from $L$, such that
\begin{equation} \label{12.65}
\dist(z,X) \leq \tau r \ \text{ for } z\in E \cap B(y,Ar) \sm B(y,r),
\end{equation}
\begin{equation} \label{12.66}
\dist(z,E) \leq \tau r \ \text{ for } z\in X \cap B(y,Ar) \sm B(y,r),
\end{equation}
\begin{eqnarray} \label{12.67}
&\,&\big|\H^d(E\cap B(z,t)) - \H^d(X\cap B(z,t))\big| 
\leq \tau r^d
\nn\\
&\,&\hskip2cm
\text{ for $z\in \R^n$ and $t>0$ such that } B(z,t) \i B(y,Ar)\sm B(y,r),
\end{eqnarray}
and
\begin{equation} \label{12.68}
|\H^d(E\cap B(y,t)) - \H^d(X\cap B(y,t))| \leq \tau r^d
\ \text{ for } r \leq t \leq Ar.
\end{equation}
Moreover, 
\begin{equation} \label{12.69}
|\H^d(X \cap B(y,1)) - \omega_d| \leq 2\tau.
\end{equation}
\end{pro}

\ms
Notice that since $N$ is larger than $\tau^{-1}$ (maybe much larger),
\eqref{12.65}-\eqref{12.67} do not give information at the scale
of $\dist(x,L)$; this is fair, because $E$ does not look like a cone at that scale.

Our measure estimate \eqref{12.67} only works outside of $B(y,r)$,
which is why we needed to add \eqref{12.68}. 

In the good cases, \eqref{12.69} should help us determine the type of $X$,
but we shall not try to do this here. See Remark \ref{t12.8}. 

\begin{proof}
Let $E$ and $x$ be as in the statement. Notice that the assumptions
of Proposition \ref{t12.1} are satisfied (if $\varepsilon$ is small enough,
depending on $N$), so we may use the previous results of this section.

Set $\delta(x) = \dist(x,L)$; the fact that $\delta(x) \leq N^{-1}$
(i.e., that \eqref{12.63} holds) as soon as $\varepsilon$ is small enough
follows from Proposition \ref{t11.1}, or directly \eqref{12.31}.
The main point is thus the existence of $X$.

Again we shall use \eqref{12.22}, but since we want to apply a near monotonicity
result for balls centered at $y$, we will translate it in terms of the functional $F_y$
defined as $F_x$ in \eqref{12.1}, but with $x$ replaced by $y$. 
Notice that for $r > 2|y-x| = 2\delta(x)$,
\begin{equation} \label{12.70}
\theta_y(r) = r^{-d} \H^d(E \cap B(y,r))
\geq r^{-d} \H^d(E \cap B(y,r-\delta(x)))
= \Big({r - \delta(x)\over r}\Big)^d \theta_x(r-\delta(x)).
\end{equation}
Similarly,
\begin{equation} \label{12.71}
\theta_y(r) \leq \Big({r + \delta(x)\over r}\Big)^d \theta_x(r+\delta(x)).
\end{equation}
We also deduce from the definition of $S_x$ (see \eqref{1.45}) that
\begin{equation} \label{12.72}
{\omega_d \over 2} - C\,{\delta(x) \over r} \leq \rho^{-d} \H^d(S_x\cap B(x,\rho))
\leq {\omega_d \over 2}
\end{equation}
for $\rho \geq \delta(x)$. Let us use this to show that
\begin{equation} \label{12.73}
|\theta_y(r) - \omega_d| \leq C \,{\delta(x) \over r} + 3\tau_1
\ \text{ for } 2\delta(x) < r < {18 \over 10}.
\end{equation}
We first apply \eqref{12.70}, set $\rho = r - \delta(x)$, and then
apply \eqref{12.1}, \eqref{12.72} and \eqref{12.22} to get that
\begin{eqnarray} \label{12.74}
\theta_y(r) &\geq& \rho^d r^{-d} \theta_x(\rho)
= \rho^d r^{-d} \big(F_x(\rho) - \rho^{-d} \H^d(S_x\cap B(x,\rho))\big)
\nn\\ 
&\geq& \rho^d r^{-d} \big(F_x(\rho) - {\omega_d \over 2}\big)
\geq \rho^d r^{-d} \big(\omega_d - \tau_1 \big)
\nn\\
&\geq& \Big(1-C\,{\delta(x) \over r}\Big)\big(\omega_d - \tau_1 \big)
\geq \omega_d - C\omega_d \,{\delta(x) \over r} - \tau_1,
\end{eqnarray}
which gives the lower bound in \eqref{12.73}. Similarly, we apply 
\eqref{12.71}, set $\rho = r + \delta(x)$, and continue as above to get that
\begin{eqnarray} \label{12.75}
\theta_y(r) &\leq& \rho^d r^{-d} \theta_x(\rho)
= \rho^d r^{-d} \big(F_x(\rho) - \rho^{-d} \H^d(S_x\cap B(x,\rho))\big)
\nn\\ 
&\leq& \rho^d r^{-d} \Big(F_x(\rho) - {\omega_d \over 2} 
+C \,{\delta(x) \over r}\Big)
\geq \rho^d r^{-d} \Big(\omega_d + C\,{\delta(x) \over r} + 2\tau_1 \Big)
\nn\\
&\leq& \Big(1 + C\,{\delta(x) \over r}\Big)
\Big(\omega_d  + C\,{\delta(x) \over r} +2\tau_1 \Big)
\leq \omega_d + C \,{\delta(x) \over r} +3 \tau_1;
\end{eqnarray}
\eqref{12.73} follows, and now we are ready to apply an almost-constant
density result from \cite{Sliding}.

Now let $r$ satisfy \eqref{12.64}. We want to apply Proposition 30.3
of \cite{Sliding} to the set $E$, with the radii
$r_1 = r/2$, $r_0 = r_2 = {3Ar \over 2}$, and the small constant 
$\tau_8 = (3A/2)^{-d} \tau$. So we check the various assumptions.

We take $U = B(y,r_0)$, a single $L_j$ equal to $L$,
and we do not need to straighten things here, i.e., we can take the
bilipschitz mapping $\xi$ of (30.2) to be the identity.
Since we took $r \leq (2A)^{-1}$ in \eqref{12.64}, we get that
$2r_0 \leq {3 \over 2}$ and $E$ is almost minimal in $B(y,2r_0)$,
with $h(2r_0)$ as small as we want (by \eqref{1.42}). This takes care
of (30.5) in \cite{Sliding} (where the small $\varepsilon$ depends on 
$\tau_8$, but this is all right); similarly, \eqref{12.73} says that
\begin{equation} \label{12.76}
|\theta_y(r_2)-\theta_y(r_1)|
\leq C r_1^{-1}\delta(x) + 6 \tau_1
\leq C N^{-1} + 6 \tau_1
\end{equation}
because $2 \delta(x) < r_1 < r_2 < {18 \over 10}$ and by \eqref{12.64}.
If $N^{-1}$ and $\tau_1$ are small enough, this implies 
(30.6) in \cite{Sliding}, and we can apply Proposition 30.3 there.
We get the existence of a coral minimal cone $X$ (with a boundary condition
coming from $L$), which satisfies \eqref{12.65} and \eqref{12.66} 
(by (30.7) and (30.8) there), and \eqref{12.67} and \eqref{12.68}
(by (30.9) and (30.10) there). 
Finally, \eqref{12.69} follows from \eqref{12.68} (applied with $t=r$)
and the fact that $|\theta_y(r)-\omega_d| \leq \tau$, by \eqref{12.73}
(and if $N^{-1}$ and $\tau_1$ are small enough).
This completes our proof of Proposition~\ref{t12.7}.
\end{proof}

\begin{rem} \label{t12.8}
When $d=2$, the author believes that any sliding minimal cone $X$ that satisfies 
\eqref{12.69} must lie in $\bV(L)$. The proof would follow, for instance, the proof
of the description of the minimal cones that was given in \cite{Holder}, 
and show that $X \cap \d B(y,1)$ is composed of arcs of great circles 
with constraints on how they meet $L$ or each other. 
But he did not check the details. Possibly this is also true in 
some higher dimensions too. Then the description in Proposition \ref{t12.7}
becomes a little better. The author also expects that when $d=2$, we should
be able to get a better local description of $E$, possibly even with a local $C^1$
parameterization.
\end{rem}

\begin{rem} \label{t12.9}
We did not try to state Corollary \ref{t1.8} or the results of this section when
$L$ is a smooth $(d-1)$-dimensional surface. The main ingredients, namely
the almost monotonicity formula (Theorem \ref{t7.1}), and the approximation
results (Corollary \ref{t9.3} and its earlier analogues in \cite{Holder} and 
\cite{Sliding}) are valid in this context, but the author did not check
that the rest of the proofs in this section goes through too.
\end{rem}

\bigskip
\vfill \vfill \vfill\vfill
\noindent Guy David,  
\par\noindent 
Math\'{e}matiques, B\^atiment 425,
\par\noindent 
Universit\'{e} de Paris-Sud, 
\par\noindent 
91405 Orsay Cedex, France
\par\noindent 
guy.david@math.u-psud.fr
\par\noindent 
http://www.math.u-psud.fr/$\sim$gdavid/

\end{document}